\documentclass[reqno]{amsart}

\usepackage{amsmath}
\usepackage{amssymb}
\usepackage{amsthm}
\usepackage{dsfont}
\numberwithin{equation}{section}
\newtheorem{theorem}{Theorem}[section]
\newtheorem{proposition}[theorem]{Proposition}
\newtheorem{corollary}[theorem]{Corollary}
\newtheorem{lemma}[theorem]{Lemma}

\theoremstyle{definition}

\theoremstyle{definition} %%{remark}
\newtheorem{remark}[theorem]{Remark}

\newcommand{\beq}{\begin{equation}}
\newcommand{\eeq}{\end{equation}}
\newcommand{\bea}{\begin{eqnarray}}
\newcommand{\eea}{\end{eqnarray}}

\newcommand{\rr}{\mathbb R}

\newcommand{\R}{\mathbb{R}}
\newcommand{\K}{\mathbb{K}}

\newcommand{\EE}{\mathbb E}
\newcommand{\FF}{\mathbb F}
\newcommand{\LL}{\mathbb L}
\newcommand{\NN}{\mathbb N}

\newcommand{\cA}{{\mathcal A}}
\newcommand{\cB}{{\mathcal B}}

\newcommand{\cD}{{\mathcal D}}
\newcommand{\cE}{{\mathcal E}}

\newcommand{\cG}{{\mathcal G}}
\newcommand{\cL}{{\mathcal L}}
\newcommand{\cP}{{\mathcal P}}
\newcommand{\cR}{{\mathcal R}}

\newcommand{\cMH}{{\mathcal MH}}
\newcommand{\cPM}{{\mathcal PM}}

\newcommand{\cN}{{\mathcal N}}
\newcommand{\q}{q}

\newcommand{\cxx}{\sf x}
\newcommand{\cy}{\sf y}

\newcommand{\Li}{\mathcal L}

\newcommand{\compemb}{\lhook\hspace{-0.151cm}{-}\hspace{-0.3cm}\hookrightarrow}

\newcommand{\ve}{\varepsilon}

\DeclareMathSymbol{\complement}{\mathord}{AMSa}{"7B}

\def\vv<#1>{\langle #1\rangle}
\def\Vv<#1>{\bigl\langle #1\bigr\rangle}

\begin{document}
\setlength{\arraycolsep}{0.5ex}

\title[Two-Phase Navier-Stokes Equations]
{Qualitative Behaviour of Solutions\\ for the
Two-Phase Navier-Stokes Equations\\ with  Surface Tension }

\author[Matthias K\"ohne]{Matthias K\"ohne}
\address{Center for Smart Interfaces \& International Research Training Group "Mathematical Fluid Dynamics"\\
          Technical University Darmstadt\\
         Petersenstr.~32\\
         D-64287 Darmstadt, Germany}
\email{koehne@csi.tu-darmstadt.de}

\author[Jan Pr\"uss]{Jan Pr\"uss}
\address{Institut f\"ur Mathematik  \\
         Martin-Luther-Universit\"at Halle-Witten\-berg\\
         Theodor-Lieser-Str.~5\\
         D-06120 Halle, Germany}
\email{jan.pruess@mathematik.uni-halle.de}

\author[Mathias Wilke]{Mathias Wilke}
\address{Institut f\"ur Mathematik  \\
         Martin-Luther-Universit\"at Halle-Witten\-berg\\
         Theodor-Lieser-Str.~5\\
         D-06120 Halle, Germany}
\email{mathias.wilke@mathematik.uni-halle.de}

\dedicatory{In memory of Professor Tetsuro Miyakawa\ \dag}

\begin{abstract}
The two-phase free boundary value problem for the isothermal Navier-Stokes system
is studied for general bounded geometries in absence of phase transitions,
external forces and boundary contacts. It is shown that the problem is
well-posed in an $L_p$-setting, and that it generates a local semiflow
on the induced phase manifold.
If the phases are connected, the set of equilibria of the system forms a
$(n+1)$-dimensional manifold,
each equilibrium is stable, and it is shown that global solutions
which do not develop singularities converge to an equilibrium as time
goes to infinity. The latter is proved by means of the energy functional combined with the {\it generalized
principle of linearized stability}.
\end{abstract}
\maketitle
{\small\noindent
\vspace{0.1in} \vspace{0.1in}\\
\date{today}}

%%%%%%%%%%%%%%%%%%%%%%%%%%%%%%%%%%%%%%%%%%%%%
\section{Introduction}\label{sect-intro}
%%%%%%%%%%%%%%%%%%%%%%%%%%%%%%%%%%%%%%%%%%%%%
In this paper we consider a free boundary problem
that describes the motion of two isothermal, viscous, incompressible
 Newtonian fluids in $\R^3$.
The fluids are separated by an interface
that is unknown and has to be determined as part of the problem.

More precisely, we consider two fluids that fill a region $\Omega\subset\R^3$.
Let $\Gamma_0\subset\Omega$ be a given surface which bounds
the region $\Omega_{1}(0)$ occupied by a viscous incompressible
fluid, $fluid_{1}$, the {\em dispersed phase}
and let $\Omega_{2}(0)$ be the complement of
the closure of $\Omega_{1}(0)$ in $\Omega$, corresponding
to the region occupied by a second incompressible viscous fluid,
$fluid_{2}$, the {\em continuous phase}. Note that the dispersed phase
is assumed not to be in contact with the boundary $\partial\Omega$ of $\Omega$.
We assume that the two fluids are immiscible, and that no phase transitions occur.
The velocity of the fluids is
denoted by $u(t,x)$, and the pressure field by $\pi(t,x)$.

Let $\Gamma(t)$ denote the position of $\Gamma_{0}$
at time $t$. Thus, $\Gamma(t)$ is a sharp interface
which separates the fluids occupying the
regions $\Omega_1(t)$ and $\Omega_2(t)$, respectively.
We denote the normal field on $\Gamma(t)$,
pointing from $\Omega_1(t)$ into $\Omega_2(t)$, by $\nu_\Gamma(t,\cdot)$.
Moreover,
$V_\Gamma(t,\cdot)$ and $H_\Gamma(t,\cdot)$ mean the normal velocity and
the curvature of $\Gamma(t)$ with respect to $\nu_\Gamma(t,\cdot)$,
respectively. Here the curvature $H_\Gamma:=-{\rm div}_\Gamma \nu_\Gamma$ is  negative
when $\Omega_1(t)$ is convex in a neighborhood of $x\in\Gamma(t)$, in particular the
curvature of a sphere $S_R(x_0)$ is $-(n-1)/R$.
The motion of the  fluids is governed by the following system of
equations for $i=1,2.$
\begin{equation}
\label{NS-2phase}
\begin{aligned}
\rho_i\big(\partial_tu+(u|\nabla)u\big)
                   -\mu_i\Delta u+\nabla \pi
                       & = 0 &\ \hbox{in}\quad &\Omega_i(t),\\
    {\rm div}\,u & = 0 &\ \hbox{in}\quad &\Omega_i(t), \\
             %  u=\ab 0 &\ \hbox{on}&\partial\Omega \\
           -{[\![S(u,\pi)\nu_\Gamma]\!]}& = \sigma H_\Gamma \nu_\Gamma &\
\hbox{on}\quad &\Gamma(t),\\
            {[\![u]\!]}& = 0 &\ \hbox{on}\quad &\Gamma(t), \\
u & = 0 &\ \hbox{on}\quad & \partial\Omega, \\
                      V_\Gamma& = (u| \nu_\Gamma) &\ \hbox{on}\quad &\Gamma(t),\\
             u(0)& = u_0&\ \hbox{in}\quad &\Omega\backslash\Gamma_0,\\
              \Gamma(0)& = \Gamma_0\,.\\
\end{aligned}
\end{equation}
Here, $S$ is the stress tensor defined by
\begin{equation*}
S=S(u,\pi)=\mu_i\big(\nabla u+[\nabla u]^{\sf T}\big)-\pi I =2\mu_i E -\pi I
\quad \text{in}\quad \Omega_i(t),
\end{equation*}
and
\begin{equation*}
[\![\phi]\!](t,x)=\lim_{h \rightarrow 0+}\Big(\phi(t,x + h\nu_\Gamma(x)) - \phi(t,x - h\nu_\Gamma(x))\Big), \quad x \in \Gamma(t)
\end{equation*}
denotes the jump of the quantity $\phi$, defined on the respective
domains $\Omega_i(t)$, across the interface $\Gamma(t)$.
\smallskip\\
Given are the initial position $\Gamma_{0}$
of the interface and the initial velocity
$u_0:\Omega\backslash\Gamma_0\to \R^{3}$.
The unknowns are the velocity field
$u(t,\cdot):\Omega\backslash\Gamma(t)\to \R^{3}$,
the pressure field $\pi(t,\cdot):\Omega\backslash\Gamma(t)\to \R$,
and the free boundary $\Gamma(t)$.
\\
The constants $\rho_i>0$ and $\mu_i>0$ denote the densities
and the viscosities of the respective fluids,
and the constant $\sigma>0$ stands for the surface tension. In the sequel
we drop the index $i$ since there is no danger of confusion; however, we keep in mind
that $\mu$ and $\rho$ have jumps across the interface, in general.

System \eqref{NS-2phase} comprises the
{\em two-phase Navier-Stokes equations with surface tension}.
The corresponding one-phase problem is obtained by setting $\rho_2=\mu_2=0$
and discarding $\Omega_2$. Here we concentrate the discussion on
the two-phase problem.

There are several papers in the literature dealing with problem  \eqref{NS-2phase};
cf.\ \cite{Deni91,Deni93,Deni94,DS91,DS95,ShSh09,Tanaka93b,Tanaka93}. All of them employ
Lagrangian coordinates to obtain local well-posedness. This way it seems difficult
to establish smoothing of the unknown interface, and this method is hardly useful
in case phase transitions have to be taken into account.
Here we employ a different approach, namely the  {\em Direct Mapping Method} via a {\em Hanzawa transform},
which has been quite efficient in the study of Stefan problems,
i.e.\ phase transitions involving temperature, only.

In a recent paper \cite{PrSi09a} we have shown that
problem \eqref{NS-2phase} is locally well-posed
in an $L_p$-setting provided $\Omega=\R^n$ and the initial interface $\Gamma_0$
is sufficiently close to a flat configuration. In addition, the interface as well as
the solution are proved to become instantaneously real analytic. This result is based
on a careful analysis of the underlying linear problem. Building on the latter results
we show in this paper local well-posedness for  arbitrary bounded geometries
as described above. This induces a local semiflow on a well-defined
nonlinear phase manifold.

It is known that the set $\cE$ of equilibria of the system are zero velocities,
constant pressures in the components of the phases and the dispersed phase is a
union of disjoint balls.
Concentrating on the case of connected phases, we prove that equilibria are stable
and any solution starting in
a neighbourhood of such a steady state exists globally and converges to another equilibrium.

The energy of the system serves as a strict Ljapunov functional, hence the limit sets
of the solutions are contained in the set of equilibria $\cE$. Combining these results
we show that any solution which does not develop singularities converges
to an equilibrium  in the topology of the phase manifold.

%%%%%%%%%%%%%%%%%%%%%%%%%%%%%%%%%%%%%%%%%%%%
\section{ Transformation to a Fixed Domain}\label{sect-tfd}
%%%%%%%%%%%%%%%%%%%%%%%%%%%%%%%%%%%%%%%%%%%%
Let $\Omega\subset\R^n$ be a bounded domain with boundary $\partial\Omega$ of class $C^2$, and suppose
$\Gamma\subset\Omega$ is a hypersurface of class $C^2$,
i.e.\ a $C^2$-manifold which is the boundary of a bounded domain
$\Omega_1\subset\Omega$; we then set
$\Omega_2=\Omega\backslash\bar{\Omega}_1$. Note that $\Omega_2$ is connected,
but $\Omega_1$ maybe  disconnected, however, it consists of finitely
many components only, since $\partial\Omega_1=\Gamma$ by assumption is a manifold, at least  of class $C^2$.
Recall that the {\em second order bundle} of $\Gamma$ is given by
$$\cN^2\Gamma:=\{(p,\nu_\Gamma(p),\nabla_\Gamma\nu_\Gamma(p)):\, p\in\Gamma\}.$$
Here $\nabla_\Gamma$ denotes the surface gradient on $\Gamma$. Recall also the
{\em Haussdorff distance} $d_H$ between the two closed subsets $A,B\subset\R^m$,
defined by
$$d_H(A,B):= \max\{\sup_{a\in A}{\rm dist}(a,B),\sup_{b\in B}{\rm dist}(b,A)\}.$$
Then we may approximate $\Gamma$ by a real analytic hypersurface $\Sigma$,
in the sense that the Haussdorff distance of the second order bundles of
$\Gamma$ and $\Sigma$ is as small as we want. More precisely, for each $\eta>0$ there is a
real analytic closed hypersurface $\Sigma$ such that
$d_H(\cN^2\Sigma,\cN^2\Gamma)\leq\eta$. If $\eta>0$ is small enough, then $\Sigma$
bounds a domain $G_1$ with  $\overline{G_1}\subset\Omega$, and we set $G_2=\Omega\backslash\bar{G}_1$.

It is well known that such a hypersurface $\Sigma$ admits a tubular neighbourhood,
which means that there is $a>0$ such that the map
\begin{equation*}
\Lambda: \Sigma \times (-a,a)\to \R^n, \quad \quad
\Lambda(p,r):= p+r\nu_\Sigma(p), \quad p \in \Sigma,\ |r| < a
\end{equation*}
is a diffeomorphism from $\Sigma \times (-a,a)$
onto $\cR(\Lambda)$. The inverse
$$\Lambda^{-1}:\cR(\Lambda)\mapsto \Sigma\times (-a,a)$$ of this map
is conveniently decomposed as
$$\Lambda^{-1}(x)=(\Pi(x),d(x)),\quad x\in\cR(\Lambda).$$
Here $\Pi(x)$ means the orthogonal projection of $x$ to $\Sigma$ and $d(x)$ the signed
distance from $x$ to $\Sigma$; so $|d(x)|={\rm dist}(x,\Sigma)$ and $d(x)<0$ iff
$x\in G_1$. In particular we have $\cR(\Lambda)=\{x\in \R^n:\, {\rm dist}(x,\Sigma)<a\}$.

Note that on the one hand, $a$ is determined by the curvatures of $\Sigma$, i.e.\ we must have
$$0<a<\min\{1/|\kappa_j(p)|: j=1,\ldots,n-1,\; p\in\Sigma\},$$
where $\kappa_j(p)$ mean the principal curvatures of $\Sigma$ at $p\in\Sigma$.
But on the other hand, $a$ is also connected to the topology of $\Sigma$,
which can be expressed as follows. Since $\Sigma$ is a compact manifold of
dimension $n-1$ it satisfies the ball condition, which means that
there is a radius $r_\Sigma>0$ such that for each point $p\in \Sigma$
there are $x_j\in G_j$, $j=1,2$, such that $B_{r_\Sigma}(x_j)\subset G_j$, and
$\bar{B}_{r_\Sigma}(x_j)\cap\Sigma=\{p\}$. Choosing $r_\Sigma$ maximal,
we then must also have $a<r_\Sigma$.

Setting $\Gamma(0)=\Gamma_0$, we may use the map $\Lambda$ to parametrize the unknown free
boundary $\Gamma(t)$ over $\Sigma$ by means of a height function $h$ via
$$\Gamma(t) = \cR(p\mapsto p+ h(t,p)\nu_\Sigma(p),\quad p\in\Sigma),\quad t\geq0,$$
for small $t\geq0$, at least.
Extend this diffeomorphism to all of $\bar{\Omega}$ by means of
$$ \Theta_h(t,x) = x +\chi(d(x)/a)h(t,\Pi(x))\nu_\Sigma(\Pi(x))=:x+\theta_h(t,x).$$
Here $\chi$ denotes a suitable cut-off function; more precisely, $\chi\in\cD(\R)$,
$0\leq\chi\leq 1$, $\chi(r)=1$ for $|r|<1/3$, and $\chi(r)=0$ for $|r|>2/3$.
This way $\Omega\backslash\Gamma(t)$ is transformed to the fixed domain
$\Omega\backslash\Sigma$. Note that $\Theta_h(t,x)=x$ for $|d(x)|>2a/3$, and
$$\Theta_h^{-1}(t,x)= x-h(t,\Pi(x))\nu_\Sigma(\Pi(x))\quad \mbox{ for }\; |d(x)|<a/3,$$
in particular,
$$\Theta_h^{-1}(t,x)= x-h(t,x)\nu_\Sigma(x)\quad \mbox{ for }\; x\in\Sigma.$$
Now we define the transformed quantities
\begin{eqnarray*}&&\bar{u}(t,x)= u(t,\Theta_h(t,x)),\\
&& \bar{\pi}(t,x)=\pi(t,\Theta_h(t,x)),\quad t>0,\; x\in\Omega\backslash\Sigma,
\end{eqnarray*}
the {\em pull backs} of $u$ and $\pi$.
This gives the following problem for $\bar{u},\bar{\pi},h$.
\begin{equation}\label{tfbns1}
\begin{array}{rcll}
\rho\partial_t \bar{u} -\mu\cA(h)\bar{u} +\cG(h)\bar{\pi}
& = & \cR(\bar{u},h) & \quad \mbox{ in }\Omega\backslash\Sigma,\\[0.25em]
(\cG(h)|\bar{u})
& = & 0 & \quad \mbox{ in }\Omega\backslash\Sigma,\\[0.25em]
{-[\![\mu([\cG(h)\bar{u}]+[\cG(h)\bar{u}]^{\sf{T}})-\bar{\pi}]\!]}\nu_\Gamma(h)
& = & \sigma H_\Gamma(h) \nu_\Gamma(h) & \quad \mbox{ on } \Sigma,\\[0.25em]
{[\![\bar{u}]\!]}
& = & 0 & \quad \mbox{ on } \Sigma,\\[0.25em]
\bar{u}
& = & 0 & \quad \mbox{ on }\partial\Omega,\\[0.25em]
\partial_t h - (\bar{u}|\nu_\Sigma)
& = & -(\bar{u}|\alpha(h)) & \quad \mbox{ on }\Sigma, \\[0.25em]
\bar{u}(0) = \bar{u}_0,\ \mbox{in}\ \Omega\backslash\Sigma,\ h(0) &=& h_0  &\quad \mbox{on}\ \Sigma.
\end{array}
\end{equation}
Here $\cA(h)$, $\cG(h)$ and $H_\Gamma(h)$ denote the transformed Laplacian, gradient and curvature, respectively. More precisely, we have
$$\Theta_h^\prime = I + \theta_h^\prime, \quad \quad \Theta_h^{\prime-1} = I - {[I + \theta_h^\prime]}^{-1}\theta_h^\prime$$
and
\begin{align*}
[\nabla\pi]\circ\Theta_h
&= \cG(h)\bar{\pi}\\
&=  [{(\Theta_h^{-1})}^{\prime{\sf T}}\circ\Theta_h]\nabla\bar{\pi}
 =  \nabla\bar{\pi} - \theta_h^{\prime\sf{T}}{[I + \theta_h^\prime]}^{-\sf{T}}\nabla\bar{\pi}\\
&=: (I - M_1(h))\nabla\bar{\pi}\\
[\mbox{div}\,u]\circ\Theta_h
&= (\cG(h)|\bar{u})\\
&=  ([{(\Theta_h^{-1})}^{\prime{\sf T}}\circ\Theta_h]\nabla|\bar{u})
 =  (\nabla|\bar{u}) - (\theta_h^{\prime\sf{T}}{[I + \theta_h^\prime]}^{-\sf{T}}\nabla|\bar{u})\\
&=  ((I - M_1(h))\nabla|\bar{u})
\end{align*}
and
\begin{align*}
[\Delta u]\circ\Theta_h
&= \cA(h)\bar{u}\\
&= [{(\Theta_h^{-1})}^{\prime}{(\Theta_h^{-1})}^{\prime\sf{T}}\circ\Theta_h]:\nabla^2\bar{u}
+ ([\Delta\Theta_h^{-1}]\circ\Theta_h|\nabla)\bar{u}\\
&= \Delta \bar{u} - M_4(h):\nabla^2\bar{u} - M_2(h)\nabla\bar{u}
\end{align*}
with
\begin{align*}
-M_2(h)\nabla\bar{u}
&:=([\Delta\Theta_h^{-1}]\circ\Theta_h|\nabla)\bar{u}\\
M_4(h):\nabla^2\bar{u}
&:= [2\textrm{sym}(\theta_h^{\prime\sf{T}}{[I + \theta_h^\prime]}^{-\sf{T}})
    - {[I + \theta_h^\prime]}^{-1}\theta_h^\prime \theta_h^{\prime\sf{T}}{[I + \theta_h^\prime]}^{-\sf{T}}]:\nabla^2\bar{u}.
\end{align*}
Note that
\begin{align*}
[\partial_t u]\circ\Theta_h
&=  \partial_t\bar{u} - \bar{u}^\prime[(\partial_t\Theta_h^{-1})\circ\Theta_h]
 =  \partial_t\bar{u} - \bar{u}^\prime\Theta_h^{\prime-1}\partial_t\Theta_h\\
&=  \partial_t\bar{u} - \bar{u}^\prime{[I + \theta_h^\prime]}^{-1}\theta_h^\prime\partial_t\theta_h
 =: \partial_t\bar{u} - M_3(h)\nabla\bar{u},
\end{align*}
hence
$$R(\bar{u},h)=-\rho(\bar{u}\cdot\cG(h)\bar{u})+M_3(h)\nabla\bar{u}.$$
With the curvature tensor $L_\Sigma$ and the surface gradient
 $\nabla_\Sigma$ we have
\begin{eqnarray*}
\nu_\Gamma(h)= \beta(h)(\nu_\Sigma-\alpha(h)),&& \alpha(h)= M_0(h)\nabla_\Sigma h,\\
M_0(h)=(I-hL_\Sigma)^{-1},&& \beta(h) = (1+|\alpha(h)|^2)^{-1/2},
\end{eqnarray*}
and
$$V=(\partial_t\Theta|\nu_\Gamma) = \partial_t h (\nu_\Gamma|\nu_\Sigma)=
\beta(h)\partial_t h.$$
Employing this notation, we have
\begin{align*}
\theta_h^\prime(t,x)
&= \nu_\Sigma(\Pi(x))\otimes M_0(d(x))\nabla_\Sigma h(t,\Pi(x))\\
&-h(t,\Pi(x))L_\Sigma(\Pi(x))M_0(d(x))\cP_\Sigma
\quad \mbox{ for }\; |d(x)|<a/3,\\
\theta_h^\prime(t,x)&=0\quad \mbox{ for }\; |d(x)|>2a/3,
\end{align*}
and
\begin{align*}
\theta_h^\prime(t,x)&= \frac{1}{a}\chi'(d(x)/a)h(t,\Pi(x))\nu_\Sigma(\Pi(x))\otimes\nu_\Sigma(\Pi(x))\\
& + \chi(d(x)/a)\nu_\Sigma(\Pi(x))\otimes M_0(d(x))\nabla_\Sigma h(t,\Pi(x))\\
&-\chi(d(x)/a)h(t,\Pi(x))L_\Sigma(\Pi(x))M_0(d(x))\cP_\Sigma\\
&\quad \mbox{ for }\; a/3<|d(x)|<2a/3,
\end{align*}
where $\cP_{\Sigma}= I -\nu_{\Sigma}\otimes\nu_{\Sigma}$ denotes the projection
onto the tangent space of $\Sigma$. Thus, $[I + \theta_h^\prime]$ is boundedly invertible, if $|h|_\infty$ and $|\nabla_\Sigma h|_\infty$ are sufficiently small.
%, i.e.
%\begin{equation}
%\label{hanzawainv}
%{|h|}_\infty < \frac{1}{3}\min\{a/|\chi'|_\infty,1/|L_\Sigma|_\infty\}
%\quad \mbox{ and }\; {|\nabla_\Sigma h|}_\infty < \frac{1}{3}.
%\end{equation}
The curvature $H_\Gamma(h)$ becomes
$$ H_\Gamma(h) = \beta(h)\{ {\rm tr} [M_0(h)(L_\Sigma+\nabla_\Sigma \alpha(h))]
-\beta^2(h)(M_0(h)\alpha(h)|[\nabla_\Sigma\alpha(h)]\alpha(h))\},$$
a differential expression involving second order derivatives of $h$ only linearly.
Its linearization is given by
$$H^\prime_\Gamma(0)= {\rm tr}\, L_\Sigma^2 +\Delta_\Sigma.$$
Here $\Delta_\Sigma$ denotes the Laplace-Beltrami operator on $\Sigma$.

It is convenient to decompose the stress boundary condition into tangential and normal parts.
Multiplying the stress interface condition with $\nu_\Sigma/\beta$ we obtain
$$[\![\bar{\pi}]\!]- \sigma H_\Gamma(h)= ([\![\mu([\cG(h)\bar{u]}+[\cG(h)\bar{u}]^{\sf T})]\!](\nu_\Sigma-M_0(h)\nabla_\Sigma h|\nu_\Sigma),$$
for the normal part of the stress boundary condition, and
\begin{align*}&-\cP_\Sigma[\![\mu([\cG(h)\bar{u}]+[\cG(h)\bar{u}]^{\sf T})]\!](\nu_\Sigma-M_0(h)\nabla_\Sigma h)\\
&\qquad= \left([\![\mu([\cG(h)\bar{u}]+[\cG(h)\bar{u}]^{\sf T})]\!](\nu_\Sigma-M_0(h)\nabla_\Sigma h)|\nu_\Sigma\right)M_0(h)\nabla_\Sigma h,\end{align*}
for the tangential part. Note that the latter neither contains the pressure jump nor the curvature!

We rewrite this problem in quasilinear form, dropping the bars and collecting its principal linear part
on the left hand side.
\begin{eqnarray}
\label{tfbns2}
&&\rho\partial_tu -\mu\Delta u+\nabla \pi= F(h,u)\nabla u + M_4(h):\nabla^2 u + M_1(h)\nabla\pi \quad
\mbox{ in }\Omega\backslash\Sigma,\nonumber\\
&&\mbox{div}\,u= M_1(h):\nabla u\quad
\mbox{ in }\Omega\backslash\Sigma,\nonumber\\
&&\cP_\Sigma[\![-\mu(\nabla u+\nabla u^\textsf{T})]\!]\nu_\Sigma=G_\tau(h)\nabla u,\\
&&-([\![-\mu(\nabla u+\nabla u^\textsf{T})]\!]\nu_\Sigma|\nu_\Sigma) +[\![\pi]\!]-\sigma H_\Gamma^\prime(0) h
=G_\nu(h)\nabla u + G_\gamma(h)
\quad \mbox{ on } \Sigma,\nonumber\\
&&[\![u]\!]=0\quad \mbox {on } \Sigma,\nonumber\\
&&u=0\quad \mbox{ on } \partial\Omega,\nonumber\\
&&\partial_th-(u|\nu_\Sigma)=(M_0(h)\nabla_\Sigma h| u)\quad \mbox{ on } \Sigma,\nonumber\\
&&u(0)=u_0 \; \mbox{ in }\Omega\backslash\Sigma,\quad h(0)  =  h_0  \; \mbox{ on }\Sigma.\nonumber
\end{eqnarray}
The right-hand sides in this problem are either lower order terms or are
of the same order appearing on the left, but carrying a factor $h$ or $\nabla_\Sigma h$,
which are small by construction. In fact, since  $\Gamma_0$ is approximated
by $\Sigma$ in the second order bundle we have smallness of $h_0$, $\nabla_\Sigma h_0$,
and even of $\nabla_\Sigma^2 h_0$, uniformly on $\Sigma$. All terms on the right-hand side
are at least quadratic. More precisely, besides the $M_j(h)$ which have been introduced before,
the nonlinearities have the following form:
\begin{align*}
F(h,u)\nabla u =
& -(u|\nabla u) + [M_1(h) + M_2(h) + M_3(h)]\nabla u,\\
G_\nu(h)\nabla u =
& - ([\![\mu([\nabla u] +[\nabla u]^{\sf T})]\!]M_0(h)\nabla_\Sigma h|\nu_\Sigma)\\
& - ([\![\mu([M_1(h)\nabla u]+[M_1(h) \nabla u]^{\sf T})]\!](\nu_\Sigma-M_0(h)\nabla_\Sigma h)|\nu_\Sigma),\\
G_\tau(h)\nabla u =
& ([\![\mu([(I-M_1(h))\nabla u]+[(I-M_1(h))\nabla u]^{\sf T})]\!](\nu_\Sigma - M_0(h)\nabla_\Sigma h)|\nu_\Sigma)\cdot\\
&\cdot M_0(h)\nabla_\Sigma h\\
& - \cP_\Sigma[\![\mu([(I-M_1(h))\nabla u]+[(I-M_1(h))\nabla u]^{\sf T})]\!]M_0(h)\nabla_\Sigma h\\
& - \cP_\Sigma[\![\mu([M_1(h)\nabla u]+[M_1(h)\nabla u]^{\sf T})]\!]\nu_\Sigma,\\
G_\gamma(h) =
& \ \sigma(H_\Gamma(h)-H_\Gamma^\prime(0)h).
\end{align*}
The idea of our approach can be described as follows.
We consider the transformed problem \eqref{tfbns2}.
Based on maximal $L_p$-regularity of the linear problem given by the left
hand side of \eqref{tfbns2}, we employ the contraction mapping principle to obtain local
well-posedness of the nonlinear problem.
The solutions of the transformed  problem will belong to the following class:
$$u\in H^1_p(J;L_p(\Omega)^n)\cap L_p(J;H^2_p(\Omega\backslash\Sigma)^n),\quad
\pi \in L_p(J; \dot{H}^1_p(\Omega\backslash\Sigma)),$$
$$[\![\pi]\!]\in W^{1/2-1/2p}_p(J;L_p(\Sigma))\cap L_p(J;W^{1-1/p}_p(\Sigma)),$$
$$h\in W^{2-1/2p}_p(J;L_p(\Sigma))\cap H^1_p(J;W^{2-1/p}_p(\Sigma))\cap
L_p(J;W^{3-1/p}_p(\Sigma)).$$
%provided
%$$ u_0\in W^{2-2/p}_p(\Omega\backslash\Sigma)^n,\quad h_0\in W^{3-2/p}_p(\Sigma),$$
%and the  natural compatibility conditions hold, which are given by
%\begin{equation}\label{compatibilities}
%\begin{array}{c}
%{\rm div}\, u_0 = 0 \; \mbox{ in } \Omega\backslash\Sigma,\quad
%u_0=0 \; \mbox{ on } \partial\Omega,\\[0.25em]
%\cP_{\Gamma_0}[\![\mu(\nabla u_0 +[\nabla u_0]^{\sf T})]\!]\nu_\Gamma=0 \quad \mbox{ and } \quad
%[\![u_0]\!]=0 \; \mbox{ on } \Gamma.
%\end{array}
%\end{equation}
This program will be carried out in the next sections.

%\newpage

%%%%%%%%%%%%%%%%%%%%%%%%%%%%%%%%%%
\section{The Linearized Problem}\label{sect-lp}
%%%%%%%%%%%%%%%%%%%%%%%%%%%%%%%%%%
We consider now the inhomogeneous linear problem
\begin{equation}\label{linFB}
\begin{array}{rcll}
\rho\partial_tu -\mu\Delta u+\nabla \pi & = & \rho f & \; \mbox{ in }\Omega\backslash\Sigma,\\[0.25em]
{\rm div}\ u & = & f_d & \; \mbox{ in }\Omega\backslash\Sigma,\\[0.25em]
[\![-\mu(\nabla u+[\nabla u]^{\sf T})+I\pi]\!]\nu_\Sigma -\sigma(\Delta_\Sigma h) \nu_\Sigma & = & g & \; \mbox{ on } \Sigma,\\[0.25em]
[\![u]\!] & = & u_\Sigma & \; \mbox{ on } \Sigma,\\[0.25em]
u & = & u_b & \; \mbox{ on } \partial\Omega,\\[0.25em]
\partial_th-(u|\nu_\Sigma) +(b|\nabla_\Sigma h) & = & g_h & \; \mbox{ on } \Sigma,\\[0.25em]
u(0)=u_0 \; \mbox{ in }\Omega\backslash\Sigma,\quad h(0) & = & h_0 & \; \mbox{ on }\Sigma
\end{array}
\end{equation}
on a finite time-interval $J=[0,a]$.
We choose the same regularity classes for $u$ and $\pi$
as before, i.e.
$$u\in Z_u:= H^1_p(J;L_p(\Omega)^n)\cap L_p(J;H^2_p(\Omega\backslash\Sigma)^n),$$
and
$$\pi\in Z_\pi:= L_p(J;\dot{H}^1_p(\Omega\backslash\Sigma)).$$
Then
$$u_\Sigma\in W^{1-1/2p}_p(J;L_p(\Sigma)^n)\cap L_p(J;W^{2-1/p}_p(\Sigma)^n),$$
and
$$u_b\in W^{1-1/2p}_p(J;L_p(\partial\Omega)^n)\cap L_p(J;W^{2-1/p}_p(\partial\Omega)^n).$$
Therefore the equation for the height function
$h$ lives in the trace space for the components of $u$, i.e.
$$ g_h\in Y^0_u:=W^{1-1/2p}_p(J;L_p(\Sigma))\cap L_p(J;W^{2-1/p}_p(\Sigma)),$$
hence the natural space for  $h$ is given by
$$ h\in Z_h:=W^{2-1/2p}_p(J;L_p(\Sigma))\cap H^1_p(J;W^{2-1/p}_p(\Sigma))
\cap L_p(J;W^{3-1/p}_p(\Sigma)).$$
Here the last space comes from the curvature term in the stress boundary condition,
which induces an additional order in spatial regularity.
Assuming that $g$ belongs to the trace space of $\nabla u$, i.e.
$$g\in Y_u^1:= W^{1/2-1/2p}_p(J;L_p(\Sigma)^n)\cap L_p(J;W^{1-1/p}_p(\Sigma)^n),$$
we have the additional regularity $[\![\pi]\!]\in Y^1_u$ for the pressure jump across
the interface $\Sigma$. The function $b\in Y_u^0$
is given; we will choose $b$ appropriately in Section \ref{sect-lwp}.

There is another hidden regularity which comes from the divergence equation. To identify it, let $\phi\in\dot{H}^1_{p^\prime}(\Omega)$.
An integration by parts yields
\begin{align*}
(u|\nabla\phi)_\Omega&=-({\rm div}\, u|\phi)_\Omega+(u\cdot\nu_{\partial\Omega}|\phi)_{\partial\Omega}-([\![u\cdot\nu_\Sigma]\!]|\phi)_\Sigma \\
&=-(f_d|\phi)_\Omega +(u_b\cdot\nu_{\partial\Omega}|\phi)_{\partial\Omega}-(u_\Sigma\cdot\nu_\Sigma|\phi)_\Sigma .
\end{align*}
Set $ \widehat{H}^{-1}_p(\Omega)=(\dot{H}^1_{p^\prime}(\Omega))^*$ and define the functional $(f_d,u_b\cdot\nu_{\partial\Omega},u_\Sigma\cdot\nu_\Sigma)\in\widehat{H}^{-1}_p(\Omega)$ by means of
$$\langle(f_d,u_b\cdot\nu_{\partial\Omega},u_\Sigma\cdot\nu_\Sigma)|\phi\rangle :=-(f_d|\phi)_\Omega +(u_b\cdot\nu_{\partial\Omega}|\phi)_{\partial\Omega}-(u_\Sigma\cdot\nu_\Sigma|\phi)_\Sigma.$$
Then we have
$$\langle(f_d,u_b\cdot\nu_{\partial\Omega},u_\Sigma\cdot\nu_\Sigma)|\phi\rangle=(u|\nabla\phi)_\Omega.$$
Since $u\in H^1_p(J;L_p(\Omega)^n)$ this implies $(f_d,u_b\cdot\nu_{\partial\Omega},u_\Sigma\cdot\nu_\Sigma)\in H^1_p(J;\widehat{H}^{-1}_p(\Omega))$.
Observe that this condition contains the compatibility condition
$$\int_\Omega f_d \,dx = \int_{\partial\Omega}u_b\cdot\nu_{\partial\Omega}\,d{\partial\Omega}-\int_\Sigma u_\Sigma\cdot\nu_\Sigma \,d\Sigma,$$
which appears choosing $\phi\equiv 1$.

In the particular case $f_d=0$ we have $(f_d,u_b\cdot\nu_{\partial\Omega},u_\Sigma\cdot\nu_\Sigma)\in H^1_p(J;\widehat{H}^{-1}_p(\Omega))$
if and only if $u_b\cdot\nu_{\partial\Omega}\in H^1_p(J;\dot{W}^{-1/p}_p(\partial\Omega))$ and $u_\Sigma\cdot\nu_\Sigma \in H^1_p(J;\dot{W}^{-1/p}_p(\Sigma))$.

The main theorem of this section states that problem (\ref{linFB})
admits maximal regularity,
in particular, it defines an isomorphism between
the solution space and the space of data.

\begin{theorem}\label{linearproblem}
Let $p>n+2$, $\Omega\subset\R^n$ a bounded domain with $\partial\Omega\in C^3$, $\Sigma\subset\Omega$ a closed hypersurface of class $C^3$ and $\rho_j$, $\mu_j$, $\sigma$ be positive constants, $j=1,2$; set $J=[0,a]$, and suppose
$$b\in W^{1-1/2p}_p(J;L_p(\Sigma))^n\cap L_p(J;W^{2-1/p}_p(\Sigma))^n.$$
Then the two-phase Stokes problem (\ref{linFB}) admits a unique solution $(u,\pi,h)$
with regularity
$$u\in H^1(J;L_p(\Omega)^n)\cap L_p(J;H^2_p(\Omega\backslash\Sigma)^n),
\quad \pi\in L_p(J;\dot{H}^1_p(\Omega\backslash\Sigma)),$$
$$[\![\pi]\!]\in W^{1/2-1/2p}_p(J;L_p(\Sigma))\cap L_p(J;W^{1-1/2p}_p(\Sigma)),$$
$$ h\in W^{2-1/2p}_p(J;L_p(\Sigma))\cap H^1_p(J;W^{2-1/p}_p(\Sigma))
\cap L_p(J;W^{3-1/p}_p(\Sigma)),$$
if and only if the data $(u_0, h_0, u_b, u_\Sigma, f, f_d, g, g_h)$ satisfy the
following regularity and compatibility conditions:\\
(a) $f\in L_p(J\times\Omega)^n$, $u_0\in W^{2-2/p}_p(\Omega\backslash\Sigma)^n$,
and ${u_0}_{|_{\partial\Omega}}={u_b}_{|_{t=0}}$;\\
(b) $f_d\in  L_p(J; H^1_p(\Omega\backslash\Sigma))$, and ${\rm div}\, u_0 = {f_d}_{|_{t=0}}$;\\
(c) $u_b \in W^{1-1/2p}_p(J;L_p(\partial\Omega)^n)\cap L_p(J;W^{2-1/p}_p(\partial\Omega)^n)$,\\
(d) $u_\Sigma\in W^{1-1/2p}_p(J;L_p(\Sigma)^n)\cap L_p(J;W^{2-1/p}_p(\Sigma)^n)$;\\
(e) $(f_d,u_b\cdot\nu_{\partial\Omega},u_\Sigma\cdot\nu_\Sigma)\in H^1_p(J;\widehat{H}^{-1}_p(\Omega))$;\\
(f) $ g\in W^{1/2-1/2p}_p(J;L_p(\Sigma))^n\cap L_p(J;W^{1-1/p}_p(\Sigma))^n$;\\
(g) $[\![u_0]\!]={u_\Sigma}_{|_{t=0}}$, and
$\cP_\Sigma [\![\mu(\nabla u_0 +[\nabla u_0]^{\sf T})]\!] =\cP_\Sigma g_{|_{t=0}}$;\\
(h) $h_0\in W^{3-2/p}_p(\Sigma)$, and
$g_h\in W^{1-1/2p}_p(J;L_p(\Sigma))\cap L_p(J;W^{2-1/2p}_p(\Sigma))$.

The solution map $(u_0,h_0,u_b,u_\Sigma,f,f_d,g,g_h,b)\mapsto (u,\pi,[\![\pi]\!],h)$ is continuous
between the corresponding spaces.
\end{theorem}

The proof will be carried out in the following subsections.

\noindent
In general the pressure $\pi$ has no more regularity as stated in Theorem \ref{linearproblem}. However, there are situations where $\pi$ enjoys extra time-regularity, as stated in the following

\begin{corollary}\label{timeregpressure}
Assume in addition to the hypotheses of Theorem \ref{linearproblem} that
$$u_0=h_0=f_d=0,\quad {\rm div}\, f=0 \mbox{ in } \Omega\backslash\Sigma,$$
$$u_b\cdot\nu_{\partial\Omega}=0 \; \mbox{ on  } \;\partial\Omega,\quad  u_\Sigma\cdot\nu_\Sigma=0\;\mbox{ on  } \;\Sigma,$$
and
$$[\![(f|\nu_\Sigma)]\!]=0 \mbox{ on } \Sigma, \quad  (f|\nu_{\partial\Omega})=0 \mbox{ on } \partial\Omega.$$
Then $ \pi\in {_0H^\alpha_p}(J;L_p(\Omega))$, for each $\alpha\in(0,1/2-1/2p)$.
\end{corollary}

\begin{proof}
Let $g\in L_{p^\prime}(\Omega)$
be given and solve the problem
\begin{align}\label{reducedf2}
\Delta\phi&=\rho g\quad \mbox{ in } \; \Omega\backslash\Sigma,\nonumber\\
[\![\phi]\!]&=0\quad \mbox{ on } \;\Sigma\nonumber,\\
[\![\rho^{-1}\partial_\nu\phi]\!]&=0\quad \mbox{ on } \;\Sigma,\\
\partial_\nu\phi&=0\quad \mbox{ on }\; \partial\Omega,\nonumber
\end{align}
by Theorem \ref{thmtrans0}. Since $(f|\nabla\phi)=(u|\nabla\phi)=0$ we obtain by integration by parts
\begin{align*}
(\pi|g)_\Omega &=\left(\frac{\pi}{\rho}|\Delta \phi\right)_\Omega = -\int_\Sigma[\![\frac{\pi}{\rho}\partial_{\nu_{\Sigma}} \phi]\!]d\Sigma - \left(\frac{\nabla\pi}{\rho}|\nabla \phi\right)_\Omega\\
&= -\int_\Sigma[\![\pi]\!]\frac{\partial_{\nu_{\Sigma}} \phi}{\rho}d\Sigma - \left(\frac{\mu}{\rho}\Delta u|\nabla \phi\right)_\Omega\\
&= \int_\Omega \frac{\mu}{\rho}\nabla u:\nabla^2 \phi dx +\int_\Sigma \{[\![\frac{\mu\partial_{\nu_{\Sigma}} u}{\rho}\nabla\phi]\!]-[\![\pi]\!]\frac{\partial_{\nu_{\Sigma}} \phi}{\rho}\}d\Sigma.
\end{align*}
Since $\nabla u\in {_0H}^{1/2}_p(J;L_p(\Omega)^{n\times n})$ and $[\![\pi]\!],\partial_ku_l\in {_0W_p}^{1/2-1/2p}(J;L_p(\Sigma))$, applying $\partial_t^\alpha$ to this identity, we obtain the estimate
$$ |\partial_t^\alpha \pi|_{L_p(J\times\Omega)}\leq C\{|\partial_t^\alpha \nabla u|_{L_p(J\times\Omega)}+
|\partial_t^\alpha [\![\pi]\!]|_{L_p(J\times\Sigma)}+|\partial_t^\alpha \partial_{\nu_{\Sigma}} u|_{L_p(J\times\Sigma)} \},$$
for each $\alpha\in(0,1/2-1/2p)$, hence $\pi\in {_0H^\alpha_p}(J;L_p(\Omega))$.
\end{proof}

\noindent
It is convenient to reduce the problem to the case $$u_0=h_0=f=f_d=u_\Sigma\cdot\nu_\Sigma=u_b\cdot\nu_{\partial\Omega}=0.$$ This
can be achieved as follows. Suppose $(u,\pi,h)$ is a solution of \eqref{linFB}.
We introduce a further dummy variable $q:=[\![\pi]\!]$; note that $q\in Z_q:=Y_u^1$.
We decompose $u=u_* + u_1$, $\pi=\pi_*+\pi_1$, $q=q_*+q_1$, $h=h_*+h_1$ where
    \begin{multline*}
    h_*(t) = [2e^{-(I-\Delta_\Sigma)^{1/2} t}-e^{-2(I-\Delta_\Sigma)^{1/2} t}]h_0 +\\
           [e^{-(I-\Delta_\Sigma) t}-e^{-2(I-\Delta_\Sigma) t}](I-\Delta_\Sigma)^{-1}\{(u_0|\nu_\Sigma)-(b|\nabla_\Sigma h)+g_h(0)\},
    \quad t\geq0.
    \end{multline*}
The function $h_*$ belongs to $Z_h$ and satisfies $h_*(0)=h_0$ and
$\partial_t h_*(0)=(u_0|\nu_\Sigma)-(b|\nabla_\Sigma h)+g_h(0)$. Then $h_1$ has initial value zero, and
also $\partial_th_1(0)=0$. We set $q_*(t)=e^{\Delta_\Sigma t}\q_0$ where
$$q_0:= ([\![\mu(\nabla{u_0}+[\nabla{u_0}]^{\sf T})]\!]\nu_\Sigma|\nu_\Sigma)
 + \sigma \Delta_\Sigma h_0+ (g(0)|\nu_\Sigma)$$
is determined by the data, and we define $\pi_*$ as the solution of
\begin{align*}
\Delta\pi_*&=0\quad \mbox{ in }\;\Omega\backslash\Sigma,\\
\partial_\nu \pi_*&=0 \quad \mbox{ on }\; \partial\Omega,\\
[\![\partial_{\nu_\Sigma}\pi_*]\!]&=0,\quad [\![\pi_*]\!]=q_*\quad \mbox{ on }\;\Sigma.
\end{align*}
Note that $q_*\in Y_u^1$ and $\pi_*\in Z_\pi$, by Theorem \ref{thmtrans}.
The function $u_*\in Z_u$ is defined as the solution of the parabolic problem
\begin{eqnarray}\label{linparab}
&&\rho\partial_tu -\mu\Delta u=-\nabla \pi_*+ \rho f\quad
\mbox{ in }\Omega\backslash\Sigma,\nonumber\\
&&u=u_b\quad \mbox{ on } \partial\Omega,\nonumber\\
&&[\![-\mu(\nabla u+[\nabla u]^{\sf T})]\!]\nu_\Sigma
 =g-q_*\nu_\Sigma+\sigma (\Delta_\Sigma h_*)\nu_\Sigma
\quad \mbox{ on } \Sigma,\\
&&[\![u]\!]=u_\Sigma\quad \mbox {on } \Sigma,\nonumber\\
&&u(0)=u_0,\quad\mbox{in}\ \Omega\backslash\Sigma,\nonumber
\end{eqnarray}
which is uniquely solvable since the appropriate Lopatinskii-Shapiro conditions are satisfied; see \cite{DHP07}.
Thus we may assume w.l.o.g. $u_0=h_0=f=0$ and that the time traces of $f_d$, $g$ and $g_h$ are zero at time zero.  Finally, to remove $f_d$, we solve the transmission problem
\begin{equation*}
\begin{aligned}
\Delta \psi &= \tilde{f}_d &\ \mbox{in}\quad &\Omega\backslash\Sigma,\\
\mbox{}[\![\rho\psi]\!] &= 0 &\ \mbox{on}\quad &\Sigma,\\
[\![\partial_{\nu_{\Sigma}}\psi]\!] &= 0 &\ \mbox{on}\quad &\Sigma,\\
\partial_{\nu_{\partial\Omega}}\psi &= 0 &\ \mbox{on}\quad &\partial\Omega,
\end{aligned}
\end{equation*}
according to Theorem \ref{thmtrans2}, where $\tilde{f}_d:=f_d-{\rm div}\ u_*$. Since $\partial\Omega\in C^3$, the solution satisfies $\nabla\psi\in Z_u$. Then setting $u_2=u_1-\nabla\psi$ and $\pi_2=\pi_1+\rho\partial_t\psi-\mu\Delta\psi$, $h_2=h_1$, we see that we may assume also $f_d=u_\Sigma\cdot\nu_\Sigma=u_b\cdot\nu_{\partial\Omega}=0$, the only non-vanishing data which remain are $g,g_h,u_\Sigma,u_b$; note that the time traces at $t=0$ of these functions are zero, and
$u_b\cdot\nu_{\partial\Omega}=0$ on $\partial\Omega$ and $u_\Sigma\cdot\nu_\Sigma=0$ on $\Sigma$.
%%%%%%%%%%%%%%%%%%%%%%%%%%%%%%%%
\subsection{Flat Interface}
%%%%%%%%%%%%%%%%%%%%%%%%%%%%%%%%
In this subsection we consider the linearized problem for a flat interface.
\begin{equation}
\label{linFBflat}
\begin{aligned}
\rho\partial_tu -\mu\Delta u+\nabla \pi&=\rho f
    &\ \hbox{in}\quad &\dot\R^{n},\\
{\rm div}\,u&= f_d&\ \hbox{in}\quad &\dot\R^{n},\\
-[\![\mu\partial_y v]\!] -[\![\mu\nabla_{x}w]\!]&=g_v
    &\ \hbox{on}\quad &\R^{n-1},\\
-2[\![\mu\partial_y w]\!] +[\![\pi]\!] -\sigma\Delta h &=g_w
    &\ \hbox{on}\quad &\R^{n-1},\\
[\![u]\!] &=u_\Sigma&\ \hbox{on}\quad &\R^{n-1},\\
\partial_th- w + (b|\nabla h)&=g_h &\ \hbox{on}\quad &\R^{n-1},\\
u(0)=u_0,\ h(0)&=h_0&\ \hbox{in}\quad&\dot{\R}^n,\ \hbox{on}\ \R^{n-1}.
 \end{aligned}
\end{equation}
Here we have identified $\R^{n-1}=\R^{n-1}\times\{0\}$ and
$\dot{\R}^n=\R^n\backslash\R^{n-1}$. It is convenient to split $u=(v,w)$, $f=(f_v,f_w)$,
$g=(g_v,g_w)$ into tangential and normal components.

The  following result, which is implied by \cite[Theorem 3.1]{PrSi09b}, states that problem
(\ref{linFBflat}) admits maximal regularity,
in particular defines an isomorphism between the solution space $Z:=Z_u\times Z_\pi\times Z_q\times Z_h$
and the product-space of data $(u_0,h_0,u_\Sigma,f,f_d,g,g_h,b)$ which we denote for short by $Y$.
%%%%%%%%%%%%%%%%%%%%%%%%%%%%%%%%%%%%%%%%%%%%%%%%%%%%%
\begin{proposition}
\label{proplinflat}
Let $p>n+2$ be fixed, and assume that $\rho_j$, $\mu_j$, $\sigma$
are positive constants for $j=1,2$, and let $J=[0,a]$. Suppose
$$b_0\in \R^{n-1},\quad b_1\in  W^{1-1/2p}_p(J;L_p(\R^{n-1}))^{n-1}\cap L_p(J;W^{2-1/p}_p(\R^{n-1}))^{n-1},$$
and set $b=b_0+b_1$.
Then the Stokes problem with flat boundary \eqref{linFBflat}
admits a unique solution \\ $(u,\pi,h)$ with regularity
$$u\in H^1_p(J;L_p(\R^{n})^n)\cap L_p(J;H^2_p(\dot{\R}^{n})^n),\quad
\pi\in L_p(J;\dot{H}^1_p(\dot{\R}^{n})),$$
$$[\![\pi]\!]\in W^{1/2-1/2p}_p(J;L_p(\R^{n-1}))\cap L_p(J;W^{1-1/p}_p(\R^{n-1})),$$
$$h\in W^{2-1/2p}_p(J;L_p(\R^{n-1}))\cap H^1_p(J;W^{2-1/p}_p(\R^{n-1}))
\cap L_p(J;W^{3-1/p}_p(\R^{n-1}))$$
if and only if the data
$(f,f_d,g,g_h,u_0,h_0,u_\Sigma)$
satisfy the following regularity and compatibility conditions:
\begin{itemize}
\item[(a)]
$f\in L_p(J\times\R^{n})^n,\;u_\Sigma\in  W^{1-1/2p}_p(J;L_p(\R^{n-1})^n)\cap L_p(J;W^{2-1/p}_p(\R^{n-1})^n)$,
\vspace{1mm}
\item[(b)]
$f_d\in  L_p(J; H^1_p(\dot{\R}^{n}))$, $(f_d,u_\Sigma\cdot\nu_\Sigma)\in H^1_p(J;\widehat{H}^{-1}_p(\R^n))$,
\vspace{1mm}
\item[(c)]
$g=(g_v,g_w)\in W^{1/2-1/2p}_p(J;L_p(\R^{n-1}))^n\cap L_p(J;W^{1-1/p}_p(\R^{n-1}))^n$,
\vspace{1mm}
\item[(d)]
$g_h\in W^{1-1/2p}_p(J;L_p(\R^{n-1}))\cap L_p(J;W^{2-1/p}_p(\R^{n-1}))$,
\vspace{1mm}
\item[(e)]
$u_0\in W^{2-2/p}_p(\dot{\R}^{n})^n$, $h_0\in W^{3-2/p}_p(\R^{n-1})$,
\vspace{1mm}
\item[(f)]
${\rm div}\, u_0={f_d}|_{t=0}$ in $\,\dot\R^{n}$ and $[\![u_0]\!]={u_\Sigma}|_{t=0}$
on $\,\R^{n-1}$,
\vspace{1mm}
\item[(g)]
$-[\![\mu\partial_y v_0]\!] -[\![\mu\nabla_{x}w_0]\!] ={g_v}|_{t=0}$ on
$\,\R^{n-1}$.
\end{itemize}
The solution map $[(u_0,h_0,u_\Sigma,f,f_d,g,g_h,b)\mapsto (u,\pi,[\![\pi]\!],h)]$ is continuous between the corresponding spaces.
\end{proposition}

%%%%%%%%%%%%%%%%%%%%%%%%%%%%%%%%%%%%%%%%%%%%%%%%%%%%%%%%%

\subsection{Bent Interfaces}
Next we consider the case of a bent interface. By this we mean a situation
where the interface $\Sigma$ is given as a graph of a function $\phi:\R^{n-1}\to\R$
of class $BC^3$; thus $\Sigma=\{ (x,\phi(x)):\, x\in \R^{n-1}\}$. The normal $\nu_\Sigma$
is then given by
$$\nu_\Sigma(x)= \beta(x)\left[\begin{array}{c} -\nabla_x\phi(x)\\1\end{array}\right],
\quad \beta(x) =1/\sqrt{1+|\nabla_x\phi(x)|^2},$$
and the Laplace-Beltrami operator for such a surface
with $$\bar{h}(t,x)= h(t,(x,\phi(x)))$$ reads as
$$\Delta_\Sigma h = \Delta \bar{h} - \beta^2(\nabla^2\bar{h}\nabla\phi|\nabla\phi)
 -\beta^2[\Delta \phi -\beta^2 (\nabla^2\phi\nabla\phi|\nabla\phi)]
(\nabla\phi|\nabla \bar{h}).$$
We may assume by the reduction explained above that $u_0=h_0=f=f_d=u_\Sigma\cdot\nu_\Sigma=0$.
Set
\begin{align*}
\bar{u}(t,x,y)&= u(t,x,y+\phi(x)), \quad \bar{\pi}(t,x,y)= \pi(t,x,y+\phi(x)),
\end{align*}
for $t\in J=[0,a]$, $x\in\R^{n-1}$, $y\neq0$, and observe
$$\nabla u= \nabla\bar{u} - \nabla\phi\otimes\partial_y \bar{ u}.$$
Then we obtain for the
new variables $(\bar{u},\bar{\pi},\bar{h})$ the following problem.
For convenience we drop the bars, and split $u=(v,w)$ and $g=(g_v,g_w)$ as before.
\begin{equation}
\label{linFBpert}
\begin{aligned}
\rho\partial_tu -\mu\Delta u+\nabla \pi&=\mu B_1(u,\pi)
    &\ \hbox{in}\quad &\dot\R^{n},\\
{\rm div}\,u&=   B_2u&\ \hbox{in}\quad &\dot\R^{n},\\
-[\![\mu\partial_y v]\!] -[\![\mu\nabla_{x}w]\!]&=g_v + B_3(u,[\![\pi]\!],h)
    &\ \hbox{on}\quad &\R^{n-1},\\
-2[\![\mu\partial_y w]\!] +[\![\pi]\!] -\sigma\Delta_x h &=[g_w/\beta] +B_4(u,h)
    &\ \hbox{on}\quad &\R^{n-1},\\
[\![u]\!] &=u_\Sigma &\ \hbox{on}\quad &\R^{n-1},\\
\partial_th- w+ (b|\nabla h) &=g_h +B_5u+B_6 h &\ \hbox{on}\quad &\R^{n-1},\\
u(0)=0,\ h(0)&=0&\ \hbox{in}\quad&\dot{\R}^n,\ \hbox{on}\ \R^{n-1}.
 \end{aligned}
\end{equation}
Here we have set
\begin{align*}
B_1(u,\pi)&= |\nabla\phi|^2\partial_y^2u-2(\nabla\phi|\nabla_x\partial_y u)
+(\nabla\phi)\partial_y\pi -(\Delta\phi)\partial_yu \\
B_2u&= (\nabla\phi|\partial_y u),\\
B_3(u,[\![\pi]\!],h)&=-[\![\mu(\nabla_x v+[\nabla_x v]^{\sf T})]\!]\nabla\phi
- [\![\mu\partial_y v]\!]|\nabla\phi|^2\\
&+\{-[\![\mu (\partial_y v|\nabla\phi)]\!]+ [\![\pi]\!] -[\![\mu\partial_y w]\!]
-\sigma \Delta_\Sigma h\}\nabla\phi\\
B_4(u,h)&= -([\![\mu(\partial_y v+\nabla_x w)]\!]|\nabla\phi)
-[\![\mu\partial_y w]\!]|\nabla\phi|^2 +\sigma(\Delta_\Sigma h-\Delta h)\\
B_5u&=(\beta-1)w-\beta(\nabla\phi|v)= -\frac{\beta^2|\nabla\phi|^2}{1+\beta}w -\beta(\nabla\phi|v)\\
B_6h&=\beta^2[(b|\nabla \phi)-(b|e_n)|\nabla\phi|^2](\nabla\phi|\nabla {h}).
\end{align*}
Now suppose $(u,\pi,h)$ belongs to the maximal regularity class. We estimate
the perturbations $B_j$ as follows:
\begin{align*}|B_1(u,\pi)|_{L_p}&\leq \|\nabla\phi\|_{L_\infty}
[(2+|\nabla\phi|_{L_\infty})|\nabla^2u|_{L_p}+|\nabla\pi|_{L_p}]\\
&\quad+|\Delta\phi|_{L_\infty}|\nabla u|_{L_p},\\
|B_2u|_{L_p(H^1_p)}&\leq |\nabla \phi|_{L_\infty} |\nabla^2u|_{L_p}+
(|\nabla^2\phi|_{L_\infty}+|\nabla\phi|_{L_\infty})|\nabla u|_{L_p},\\
|\partial_tB_2u|_{L_p(H^{-1}_p)}&\leq |\nabla\phi\partial_t u|_{L_p(L_p)}
\leq |\nabla\phi|_{L_\infty}|\partial_t u|_{L_p},\\
|B_3(u,[\![\pi]\!],h)|_{W^{s}_p(L_p)}&\leq C|\nabla\phi|_{L_\infty}
(1+|\nabla\phi|_{L_\infty})
[| \nabla u|_{W^{s}_p(L_p)}+
|\nabla^2h|_{W^{s}_p(L_p)}]\\
&\quad + C|\nabla\phi|_{L_\infty}[|[\![\pi]\!]|_{W^{s}_p(L_p)}+|\nabla^2\phi|_{L_\infty}
|\nabla h|_{W^s_p(L_p)}],\\
|B_4(u,h)|_{W^{s}_p(L_p)}&\leq C|\nabla\phi|_{L_\infty}
(1+\|\nabla\phi\|_{L_\infty})
[| \nabla u|_{W^{s}_p(L_p)}+ |\nabla^2h|_{W^{s}_p(L_p)}]\\
&\quad+ C|\nabla^2\phi|_{L_\infty}|\nabla\phi|_{L_\infty}|\nabla h|_{W^s_p(L_p)},\\
|B_5u|_{W^{1-1/2p}_p(L_p)}&\leq 2|\nabla\phi|_{L_\infty}|u|_{W^{1-1/2p}_p(L_p)}.\\
\end{align*}
Here $C$ denotes a constant only depending on the parameters $\mu$ and $\sigma$,
and we have set $s=1/2-1/2p$. For the estimations in $L_p(J;W^{1-1/p}_p(\R^{n-1}))$ we
observe that
$$|\psi|_{W^s_p}= |\psi|_{L_p} +[\psi]_{s,p},
\quad [\psi]_{s,p}^p=\int_{|h|\leq 1}\int_{\R^{n-1}}|\psi(x+h)-\psi(x)|^p
\frac{dx\, dh}{|h|^{n-1+sp}},$$
defines a norm on $W^s_p(\R^{n-1})$. This implies
$$|a\psi|_{W^s_p}\leq |a|_{L_\infty}|\psi|_{W^s_p}+
c_{s,p}|\nabla a|_{L_\infty}|\psi|_{L_p},$$
for $a\in W_\infty^1(\R^{n-1})$, with some constant $c_{s,p}$ which only depends on $s\in(0,1)$,
$p\in(1,\infty)$ and on $n$. With this observation we have
\begin{align*}
|B_3(u,[\![\pi]\!],h)|_{L_p(W^{2s}_p)}&\leq C(1+|\nabla\phi|_{L_\infty})
\{|\nabla\phi|_{L_\infty}\cdot\\
&\cdot[|\nabla u|_{L_p(W^{2s}_p)} +|[\![\pi]\!]|_{L_p(W^{2s}_p)}
+|\nabla^2h|_{L_p(W^{2s}_p)}]\\
&+ |\nabla^2\phi|_{L_\infty}[|\nabla u|_{L_p} +|[\![\pi]\!]|_{L_p}
+|\nabla h|_{L_p(H^1_p)}]\}\\
&+C(|\nabla^3\phi|_{L_\infty}+|\nabla^2\phi|^2_{L_\infty})|\nabla\phi|_{L_\infty}
|\nabla h|_{L_p},\\
|B_4(u,h)|_{L_p(W^{2s}_p)}&\leq C(1+|\nabla\phi|_{L_\infty})
\{|\nabla\phi|_{L_\infty}|\nabla u|_{L_p(W^{2s}_p)}\\
&+|\nabla\phi|_{L_\infty}|\nabla^2 h|_{L_p(W^{2s}_p)}
+|\nabla^2\phi|_{L_\infty}[|\nabla u|_{L_p}+|\nabla h|_{L_p(H^1_p)}]\\
&+C(|\nabla^3\phi|_{L_\infty}+|\nabla^2\phi|^2_{L_\infty})|\nabla\phi|_{L_\infty}
|\nabla h|_{L_p}\}.\\
|B_5u|_{L_p(W^{1+2s}_p)}&\leq C|\nabla\phi|_{L_\infty}\{|u|_{L_p(W^{1+2s}_p)}
+|\nabla^2\phi|_{L_\infty}|u|_{L_p(H^1_p)}\}\\
&+C(|\nabla^3\phi|_{L_\infty}+|\nabla^2\phi|^2_{L_\infty})|u|_{L_p}.
\end{align*}
Here $C$ denotes a constant only depending on  $\mu$, $\sigma$, $p$, and $2s=1-1/p$.
To estimate $B_6 h$ we note that $Y_u^0$ is a Banach algebra since $p>n+2$. This yields
$$|B_6 h|_{Y^0_u}\leq C |\nabla\phi|_{L_\infty}(|b_0|+|b_1|_{Y_u^0})|\nabla h|_{Y_u^0}\leq C |\nabla\phi|_{L_\infty}(|b_0|+|b_1|_{Y_u^0})| h|_{Z_h}.$$

To solve the problem \eqref{linFBpert}, let $z=(u,\pi,[\![\pi]\!],h)\in {_0Z}$, where ${_0Z}$ means the solution space with
zero time trace at $t=0$,
$F:=(0,0,g_v,g_w/\beta,u_\Sigma,g_h)\in {_0Y}$,  the space of data with zero time trace,
and let $B:{_0Z}\to {_0Y}$ defined by
$$Bz=(B_1(u,\pi),B_2u,B_3(u,[\![\pi]\!],h),B_4(u,h),0,B_5u+B_6h).$$
Denoting the isomorphism from
${_0Z}$ to  ${_0Y}$ defined by the left hand side of \eqref{linFBpert} by $L$
we may rewrite problem \eqref{linFBpert} in abstract form as
\begin{equation}\label{LBF} Lz = Bz +F.\end{equation}
The above estimates for the components of $B$  imply
$$|Bz|_Y \leq C|\nabla\phi|_{L_\infty}|z|_Z+ M [|u|_{L_p(H^1_p)}+
|[\![\pi]\!]|_{L_p} + |\nabla h|_{W^s_p(L_p)\cap L_p(H^1_p)}],
$$
with some constants $C>0$ depending only on the parameters and $M>0$, which depends also
on $|\nabla\phi|_{BUC^2}$. Let $\eta>0$ be given and suppose
$|\nabla\phi|_{L_\infty}<\eta$. By means of an interpolation argument we find a
constant $\gamma>0$, depending only on $p$ such that there is a constant $M(\eta)>$ such that
$$|Bz|_Y \leq C[2\eta +a^\gamma M(\eta)]|z|_Z,\quad z\in{_0Z}.$$
Choosing first $\eta>0$ and then $a>0$ small enough, we can solve \eqref{LBF}
by a Neumann series argument for $J=[0,a]$.

Since problem \eqref{linFBpert} is time-invariant, we may repeat these arguments
finitely many times, including the reduction procedure,  to solve \eqref{linFBpert} for $J=[0,a]$,
where now $a>0$ is arbitrary.

\subsection{General Bounded Geometries}
Here we use the method of localization.
By assumption, $\partial\Omega$ is of class $C^3$ and $\Sigma$ will even be real analytic, so in particular of class $C^4$.
Therefore we may cover $\Sigma$ by $N$ balls $B_{r/2}(x_j)$ with radius $r>0$ and centers
$x_j\in \Sigma$ such that $\Sigma\cap B_r(x_j)$ can be parameterized over the
tangent space $T_{x_j}\Sigma$ by a function
$\theta_j\in C^4$ such that $|\nabla\theta_j|_{L_\infty}\leq \eta$, with $\eta>0$ defined
as in the previous subsection. We extend these functions $\theta_j$ to all of
$T_{x_j}\Sigma$  retaining the bound on $\nabla\theta_j$. This way we have created $N$
bent half-spaces $\Sigma_j$ to which the result proved in the previous subsection applies. We also
suppose that $B_r(x_j)\subset\Omega$ for each $j$.  Set
$U:=\Omega\backslash \bigcup_{j=1}^N \bar{B}_{r/2}(x_j)$ and $U_j=B_r(x_j)$, $j=1,\ldots,N$.
The open set $U$ consists of one component $U_0$ characterized by
$\partial\Omega\subset\bar{U}_0$ and an open set say $U_{N+1}$,
which is interior to $\Sigma$, i.e.\ $U_j\subset\Omega_1$. Fix a partition of unity
$\{\varphi_j\}_{j=0}^{N+1}$ subject to the covering $\{U_j\}_{j=0}^{N+1}$ of $\Omega$, i.e.\
$\varphi_j\in\cD(\R^n)$, $0\leq\varphi_j\leq 1$, and $\sum_{j=0}^{N+1}\varphi_j \equiv 1$.
Note that $\varphi_0=1$ in a neighborhood of $\partial\Omega$. Let $\tilde{\varphi}_j$ denote cut-off functions with support in $U_j$ such that $\tilde{\varphi}_j=1$ on the support of $\varphi_j$, and set $b_j=b\tilde{\varphi}_j$.

Let $z:=(u,\pi,q,h)$ with $q=[\![\pi]\!]$ be a solution of \eqref{linFB} where we assume w.l.o.g.\ $u_0=h_0=h_1=f=f_d=0$, and $(u_b|\nu_{\partial\Omega})=[\![(u_\Sigma|\nu_\Sigma)]\!]=0$.
We then set $u_j=\varphi_j u$, $\pi_j=\varphi_j \pi$, $q_j=\varphi_jq$, $h_j=\varphi_j h$,
as well as $u_{bj}=\varphi_j u_b$, $u_{\Sigma j}=\varphi_j u_\Sigma$, $g_j=\varphi_j g$, and $g_{hj}=\varphi_j g_h$. Then for $j=1,\ldots,N$, the quadruples
$z_j:=(u_j,\pi_j,q_j,h_j)$ satisfy the problems
\begin{equation}\label{linFBj}
\begin{array}{rcll}
\rho\partial_tu_j -\mu\Delta u_j+\nabla \pi_j & = & F_j(u,\pi) & \mbox{ in }\R^n\backslash\Sigma_j,\\[0.25em]
{\rm div}\ u_j & = & (\nabla\varphi_j|u) & \mbox{ in }\R^n\backslash\Sigma_j,\\[0.25em]
[\![-\mu([\nabla u_j]+[\nabla u_j]^{\sf T}) + q_j]\!]\nu_{\Sigma_j} - \sigma(\Delta_{\Sigma_j} h_j) \nu_{\Sigma_j} & = & g_j+G_j(u) & \mbox{ on }\Sigma_j,\\[0.25em]
[\![u_j]\!] = u_{\Sigma j},\quad q_j & = & [\![\pi_j]\!] & \mbox { on }\Sigma_j,\\[0.25em]
\partial_th_j-(u_j|\nu_{\Sigma_j})+(b_j|\nabla_\Sigma h_j) & = & g_{hj}+G_{hj}(h) & \mbox{ on }\Sigma_j,\\[0.25em]
u_j(0)=0 \; \mbox{ in }\R^n\backslash\Sigma_j, \quad h_j(0) & = & 0 & \mbox{ on }\Sigma_j.
\end{array}
\end{equation}
Here we used the abbreviations
$$F_j(u,\pi)=[\nabla\varphi_j]\pi
-\mu[\Delta,\varphi_j]u,$$
$$G_j(u)=[\![-\mu(\nabla \varphi_j\otimes u+ u\otimes\nabla\varphi_j)]\!] \nu_{\Sigma_j}
-\sigma[\Delta_\Sigma,\varphi_j]h\nu_{\Sigma_j},$$
and
$$ G_{hj}(h) = (b_j|\nabla_{\Sigma} \varphi_j)h.$$
For $j=0$ we have the standard one-phase Stokes problem with parameters
$\rho_2,\mu_2$ on $\Omega$ with Dirichlet boundary conditions on $\partial\Omega$, i.e.
\begin{equation*}
\begin{aligned}
\rho_2\partial_tu_0 -\mu_2\Delta u_0+\nabla \pi_j&=  F_0(u,\pi)&\quad
\mbox{ in }&\Omega,\\
{\rm div}\ u_0&=  (\nabla\varphi_0|u),&\quad
\mbox{ in }&\Omega,\\
u_0&=u_{b0} &\quad \mbox{ on }&\partial\Omega,\\
u_0(0)&=0&\quad\mbox{ in }&\Omega.
\end{aligned}
\end{equation*}
For $j=N+1$ we obtain the one-phase Stokes problem on
$\R^n$ with parameters $\rho_1,\mu_1$, i.e.
\begin{equation*}
\begin{aligned}
\rho_1\partial_tu_{N+1} -\mu_1\Delta u_{N+1}+\nabla \pi_{N+1}&=  F_{N+1}(u,\pi)&\quad
\mbox{ in }&\R^n,\\
{\rm div}\ u_{N+1}&=  (\nabla\varphi_{N+1}|u)&\quad
\mbox{ in }&\R^n,\\
u_{N+1}(0)&=0&\quad \mbox{ in }&\R^n.
\end{aligned}
\end{equation*}

Concentrating on $j=1,\ldots,N$, we first note that
$[\Delta,\varphi_j]$ are differential operators of order 1, hence if $u\in {_0Z}_u$ then
$$[\Delta,\varphi_j]u\in {_0H_p^{1/2}}(J;L_p(\R^n)^n)\cap L_p(J; H^1_p(\R^n\backslash\Sigma_j)^n).$$
Since $f=f_d=0$ the pressure $\pi$ belongs to
$$\pi \in {_0H}^\alpha_p(J;L_p(\R^n))\cap L_p(J; H^1_p(\R^n\backslash\Sigma_j)),$$
by Corollary \ref{timeregpressure}, hence we have
$$F_j(u,\pi)\in {_0H}^\alpha_p(J;L_p(\R^n))^n\cap L_p(J; H^{1}_p(\R^n\backslash\Sigma_j))^n,$$
for some fixed $0<\alpha<\frac{1}{2}-\frac{1}{2p}.$
 Similarly we have
$$\nabla\varphi_j(u|\nu_{\Sigma_j}) +(\nabla\varphi_j|\nu_{\Sigma_j})u \in
{_0W}^{1-1/2p}_p(J;L_p(\Sigma_j)^n)\cap L_p(J; W^{2-1/p}_p(\Sigma_j)^n),$$
and since $[\Delta_{\Sigma_j},\varphi_j]$ is of order 1 as well, we obtain
$$[\Delta_\Sigma,\varphi_j]h\in {_0H}^1_p(J;W^{1-1/p}_p(\Sigma_j))
\cap L_p(J; W^{2-1/p}_p(\Sigma_j)).$$
This shows that we  have
$$G_j(u)\in {_0W}^{1-1/2p}_p(J;L_p(\Sigma_j)^n)\cap L_p(J; W^{2-1/p}_p(\Sigma_j)^n).$$
The terms $G_{hj}(h)$ do not have more regularity, however, the Banach algebra property yields the estimate
$$|G_{hj}(h)|_{Y^0_u}\leq C |b|_{Y^0_u}|h|_{Y^0_u}\leq C|b|_{Y^0_u}a^\gamma|h|_{Z_h},$$
with an appropriate exponent $\gamma>0$.
Next we decompose
$$F_j(u,\pi)= \tilde{F}_j(u,\pi) +\nabla\psi_j,$$
such that ${\rm div} \tilde{F}_j(u,\pi)=0$ in $\R^n\backslash\Sigma_j$ and $([\![\tilde{F}_j(u,\pi)]\!]|\nu_{\Sigma_j})=0$ on $\Sigma_j$.
Thus $\tilde{F}_j(u,\pi)$ is the Helmholtz projection of $F_j(u,\pi)$ in $\R^n$.
Then
$$\tilde{F}_j(u,\pi)\in {_0H}^\alpha_p(J;L_p(\R^n))^n\cap L_p(J; H^{2\alpha}_p(\R^n))^n.$$
Also, we decompose $u_j=\tilde{u}_j +\nabla\phi_j$, where  $\phi_j$ solves the transmission problem
\begin{equation*}
\begin{aligned}
\Delta\phi_j&=(\nabla\varphi_j|u)&\quad \mbox{ in }&\R^n\backslash\Sigma_j,\\
[\![\rho\phi_j]\!]&=0&\quad \mbox{ on }&\Sigma_j,\\
[\![\partial_{\nu_{\Sigma_j}}\phi_j]\!]&=0&\quad \mbox{ on }&\Sigma_j.
\end{aligned}
\end{equation*}
Note that
\begin{equation}\label{reg-nabla-phi}
\nabla\phi_j\in {_0H}^1_p(J; H^1_p(\R^n\backslash\Sigma_j)^n)\cap L_p(J;H^3_p(\R^n\backslash\Sigma_j)^n),
\end{equation}
by Theorems \ref{thmtrans0} and \ref{thmtrans2}, since $\Sigma_j$ is smooth. The jump of its trace on $\Sigma_j$ then belongs to
$$[\![\nabla\phi_j]\!]\in {_0H}^1_p(J; W^{1-1/p}_p(\Sigma_j)^n)\cap L_p(J;W^{3-1/p}_p(\Sigma_j)^n),$$
and its normal part vanishes, by construction. Further we have
$$[\![\mu\nabla^2\phi_j]\!]\in {_0W}^{1-1/2p}_p(J; L_p(\Sigma_j)^{n\times n})\cap L_p(J;W^{2-1/p}_p(\Sigma_j)^{n\times n}).$$
Then we set
$$\tilde{\pi}_j=\pi_j -\psi_j +\rho\partial_t \phi_j -\mu\Delta\phi_j,$$
and we observe that on $\Sigma_j$
$$\tilde{q}_j:=[\![\tilde{\pi}_j]\!]=[\![\pi_j]\!] -[\![\mu\Delta\phi_j]\!]=[\![\pi_j]\!] -[\![\mu(\nabla\varphi_j|u)]\!],$$
since by construction $\psi_j$ and $\rho\phi_j$ have no jump across $\Sigma_j$.
Now the quadrupel $\tilde{z}_j:=(\tilde{u}_j,\tilde{\pi}_j,\tilde{q}_j,h_j)$ satisfies the problem
\begin{equation}\label{modlinFBj}
\begin{array}{rcll}
\rho\partial_t\tilde{u}_j -\mu\Delta \tilde{u}_j+\nabla \tilde{\pi}_j & = & \tilde{F}_j(u,\pi) & \mbox{ in }\R^n\backslash\Sigma_j,\\[0.25em]
{\rm div}\ \tilde{u}_j & = & (\nabla\varphi_j|u) & \mbox{ in }\R^n\backslash\Sigma_j,\\[0.25em]
[\![-\mu([\nabla \tilde{u}_j]+[\nabla \tilde{u}_j]^{\sf T}) + \tilde{q}_j]\!]\nu_{\Sigma_j} - \sigma(\Delta_{\Sigma_j} h_j) \nu_{\Sigma_j} & = & g_j+\tilde{G}_j(u) & \mbox{ on }\Sigma_j,\\[0.25em]
[\![\tilde{u}_j]\!] = u_{\Sigma j} - [\![\nabla\phi_j]\!],\quad \tilde{q}_j & = & [\![\tilde{\pi}_j]\!] & \mbox { on }\Sigma_j,\\[0.25em]
\partial_th_j-(\tilde{u}_j|\nu_{\Sigma_j})+(b_j|\nabla_\Sigma h_j) & = & g_{hj}+\tilde{G}_{hj}(h) & \mbox{ on }\Sigma_j,\\[0.25em]
\tilde{u}_j(0)=0,\quad \mbox{in}\ \R^n\backslash\Sigma_j,\quad h_j(0) & = & 0 & \mbox{ on }\Sigma_j.
\end{array}
\end{equation}
Here $\tilde{G}_j$ and $\tilde{G}_{hj}$ are given by
$$\tilde{G}_j(u)= G_j(u)+2[\![\mu\nabla^2\phi_j]\!]\nu_{\Sigma_j}-[\![\mu(\nabla\varphi_j|u)]\!]\nu_{\Sigma_j}$$
and
$$\tilde{G}_{hj}(h) = G_{hj}(h) + \partial_{\nu_{\Sigma_j}}\phi_j.$$
For the remaining charts with index $j=0,N+1$, i.e.\ the one-phase problems, the procedure is similar. In case $j=0$ we use the regularity
$$\nabla\phi_0\in H^1_p(J;H^1_p(\Omega)^n)\cap H^{1/2}_p(J;H^2_p(\Omega)^n)$$
instead of \eqref{reg-nabla-phi}.

We write \eqref{modlinFBj} abstractly as
$$ L_j \tilde{z}_j = H_j + B_j z,$$
and by Theorem \ref{linearproblem} for bent interfaces we obtain an estimate of the form
$$|\tilde{z}_j|_{\EE}\leq C_0(|H_j|_{\FF} + |B_j z|_{\FF}),$$
with some constant $C_0$ independent of $j$. Here $\EE$ means the space of solutions and $\FF$ the space of data. Since all components of $B_jz$ (except for $G_{hj}(h)$) have some extra regularity, there is an exponent $\gamma>0$ and a constant $C_1$ independent of $j$ such that
$$|B_j z|_{\FF}\leq a^\gamma C_1|z|_{\EE}.$$
In addition, by Corollary \ref{timeregpressure} we obtain
$$|\tilde{\pi}_j|_{H^\alpha_p(J;L_p(\Omega))} \leq C_2 (|H_j|_{\FF} + |B_j z|_{\FF})\leq C_2 (|H_j|_{\FF} + a^\gamma C_1C_2|z|_{\EE}).$$
This in turn implies
$$|\partial_t\phi_j|_{H^\alpha_p(J;L_p(\Omega))}\leq  C_2 |H_j|_{\FF} + a^\gamma C_3|z|_{\EE},$$
and then also
$$|z_j|_{\EE}\leq C_4|H_j|_{\FF} + a^\gamma C_5|z|_{\EE}.$$
Summing over all $j$ yields $z=\sum_j z_j$, hence
$$|z|_{\EE}\leq C_6|H|_{\FF} + a^\gamma C_7|z|_{\EE}.$$
Therefore, choosing the length $a$ of the time interval small enough, we obtain the a priori estimate
\begin{equation}
\label{linFBapriori}
|z|_{\EE}\leq C_8|H|_{\FF}.
\end{equation}
Since the equations under consideration are time invariant, repeating this argument finitely many times we may conclude that the operator
$L:{_0\EE}\to{_0\FF}$ which maps solutions to their data is injective and has closed range, i.e.\ $L$ is a semi-Fredholm operator.

\bigskip

It remains to prove surjectivity of $L$. For this we employ the continuation method for semi-Fredholm operators. The estimates are uniform in the densities $\rho_j$ and the viscosities $\mu_j$, as long as these parameters are bounded and bounded away from zero. Hence $L=L(\rho_1,\rho_2,\mu_1,\mu_2)$ is surjective, if $L(1,1,1,1)$ has this property. Next we introduce an artificial continuation parameter $\tau\in[0,1]$ by replacing the equation for the free boundary $h$ with
$$\partial_t h+ \tau(-\Delta_\Sigma)^{1/2}h -(1-\tau)\{(u|\nu_\Sigma)-(b|\nabla_\Sigma h)\}=g_h\quad \mbox{on}\ \Sigma.$$
The arguments in \cite{PrSi09a,PrSi09b} show that the corresponding problem is well-posed for each $\tau\in[0,1]$ in the case of a flat interface, with bounds independent of $\tau\in[0,1]$. Therefore the same is true for bent interfaces and then by the above estimates also for a general geometry.
Thus we only need to consider the case $\rho_1=\rho_2=\mu_1=\mu_2=\tau=1$.

To prove surjectivity in this case, note that the equation for $h$ is decoupled from those for $u$ and $\pi$, and it is uniquely solvable in the right regularity class because of maximal regularity for the Laplace-Beltrami operator. So we may set now $h=0$. Next we solve the parabolic transmission problem to remove the jump of $u$ across $\Sigma$ and the inhomogeneity $g$ in the stress boundary condition. The remaining problem is a one-phase Stokes problem on the domain $\Omega$, which is well-known to be solvable. This shows that we have surjectivity in the case $\rho_1=\rho_2=\mu_1=\mu_2=\tau=1$, hence also for arbitrary $\rho$, $\mu$ and $\tau=0$ and the proof of Theorem~\ref{linearproblem} is complete.

\bigskip

We close this section with a remark on the situation, which occurs in the treatment of two-phase flows with variable surface tension.
To be precise, Theorem~\ref{linearproblem} may be generalized to this situation by means of the following
\begin{corollary}
\label{variable-st}
The statements of Theorem~\ref{linearproblem} remain valid, if the surface tension is not a constant but a function $\sigma \in C^{0,1/2}(J, BC(\Sigma, \rr_+)) \cap BC(J, C^{0,1}(\Sigma, \rr_+))$.
\end{corollary}
This generalization is possible, since the variability of the surface-tension may be handeled by a freezing technique during the localization procedure in the previous paragraphs.
The only difference compared to the case of a constant surface-tension occurs in the jump condition of the normal stress in problem (\ref{linFBj}), which has to be modified to
\begin{equation*}
[\![-\mu(\nabla u_j+[\nabla u_j]^{\sf T})]\!]\nu_{\Sigma_j} +\q_j\nu_{\Sigma_j}
-\sigma_j(\Delta_{\Sigma_j} h_j) \nu_{\Sigma_j}=g_j+G_j(u)+G_{\sigma j}(h)
\; \mbox{ on } \Sigma_j
\end{equation*}
with $\sigma_j = \sigma(0,x_j) > 0$ and $$G_{\sigma j}(h) = \varphi_j(\sigma - \sigma_j)(\Delta_{\Sigma} h) \nu_\Sigma.$$
Now, if the radius $r > 0$ of the balls $U_j$, $j = 1, \dots, N$ is sufficiently small, we may use the H{\"o}lder/Lipschitz constants $L_1, L_2 > 0$ of $\sigma$ to estimate
\begin{equation*}
{|\varphi_j(\sigma - \sigma_j)|}_{Y^1_u} \leq C_9 (L_1 a^{\gamma_1} + L_2 r^{\gamma_2})
\end{equation*}
with exponents $\gamma_1, \gamma_2 > 0$ and a constant $C_9 > 0$ depending only on the spatial dimension $n$ and the geometry of $\Sigma$.
Since $Y^1_u$ is a Banach algebra, we obtain
\begin{equation*}
{|G_{\sigma j}(h)|}_{Y^1_u} \leq C_{10} {|\varphi_j(\sigma - \sigma_j)|}_{Y^1_u} {|\Delta_{\Sigma} h|}_{Y^1_u} \leq C_{11} (a^{\gamma_1} + r^{\gamma_2}) {|z|}_{\EE}
\end{equation*}
and the a priori estimate (\ref{linFBapriori}) remains valid, if both $a$ and $r$ are choosen sufficiently small.

%%%%%%%%%%%%%%%%%%%%%%%%%%%%%%%
\section{Local Well-Posedness}\label{sect-lwp}
%%%%%%%%%%%%%%%%%%%%%%%%%%%%%%%
We now turn to problem (\ref{NS-2phase}),
and show its local well-posedness for given initial data
$\Gamma_0\in W^{3-2/p}_p$ and $u_0\in W^{2-2/p}_p(\Omega\backslash\Gamma_0)^n$, which are subject to the compatibility conditions
\begin{align}\label{compatibilities}
\begin{split}
{\rm div} \; u_0=0\; \mbox{ in } \Omega\backslash \Gamma_0,
\quad u_0=0 \mbox{ on } \partial\Omega,\\
[\![\cP_{\Gamma_0}\mu(\nabla u_0+[\nabla u_0]^{\sf{T}}) \nu_{\Gamma_0}]\!]=0,\; [\![u_0]\!]=0\; \mbox{
on } \Gamma_0,
\end{split}
\end{align}
where
$\cP_{\Gamma_0}=I-\nu_{\Gamma_0}\times\nu_{\Gamma_0} $.
According to the considerations in Section~\ref{sect-tfd}
we approximate $\Gamma_0$ for any prescribed $\eta > 0$
by a real analytic hypersurface $\Sigma$,
in the sense that $d_H(\cN^2\Sigma,\cN^2\Gamma_0) < \eta$,
and $\Gamma_0$ is parametrized over $\Sigma$ by
$h_0 \in W^{3-2/p}_p(\Sigma)$.
Employing the transformation from Section~\ref{sect-tfd} to the fixed domain,
 it is sufficient to prove the local well-posedness
of the quasi-linear problem (\ref{tfbns2}). We keep the notation and denote by $u$ and $\pi$ the transformed velocity field and pressure, respectively.

We are interested in solutions of \eqref{tfbns2} having maximal regularity,
and hence, we determine a first approximation of the local solution of (\ref{tfbns2})
by using Theorem~\ref{linearproblem}.
Since the time traces
\begin{equation*}
(\nabla|u_0) \in W^{1-2/p}_p(\Omega\backslash\Sigma)
\quad \textrm{and} \quad
\cP_\Sigma[\![\mu(\nabla u_0 + {[\nabla u_0]}^{\sf{T}})\nu_\Sigma]\!] \in W^{1-3/p}_p(\Sigma)
\end{equation*}
will not be trivial due to the transformation,
and since the compatibility conditions imposed in Theorem~\ref{linearproblem} have to be satisfied,
we must be able to construct extensions in the right regularity classes to be used as the right-hand sides in (\ref{linFB}).
\begin{proposition}\label{lwp-extensions}
Let $p>3$, $\partial\Omega\in C^3$, and set $J=[0,a]$.
Let $\Sigma \subset \Omega$ be a closed hypersurface of class $C^3$.
Then $f_0 \in \dot{H}_p^{-1}(\Omega)\cap  W^{1-2/p}_p(\Omega\backslash\Sigma)$ and $g_0 \in W^{1-3/p}_p(\Sigma)$ admit extensions
\begin{align*}
f & \in H^1_p(J;\dot{H}^{-1}_p(\Omega\backslash\Sigma)) \cap L_p(J;H^1_p(\Omega\backslash\Sigma)) \quad \mbox{and} \\
g & \in W^{1/2-1/2p}_p(J;L_p(\Sigma)) \cap L_p(J;W^{1-1/p}_p(\Sigma))
\end{align*}
with $f(0) = f_0$ and $g(0) = g_0$.
\end{proposition}
\begin{proof}
We take $\phi_0 \in H^2_p(\Omega) \cap W^{3-2/p}_p(\Omega\backslash\Sigma)$ to be the unique solution of
\begin{equation*}
\left\{
\begin{array}{rcll}
\Delta\phi_0 & = & f_0 & \mbox{in}\ \Omega\backslash\Sigma, \\[0.25em]
[\![\phi_0]\!] & = & 0 & \mbox{on}\ \Sigma, \\[0.25em]
[\![\partial_\nu \phi_0]\!] & = & 0 & \mbox{on}\ \Sigma, \\[0.25em]
\phi_0 & = & 0 & \mbox{on}\ \partial\Omega,
\end{array}
\right.
\end{equation*}
which exists due to $f_0 \in \dot{H}^{-1}_p(\Omega) \cap W^{1-2/p}_p(\Omega\backslash\Sigma)$; cf.\ Theorems \ref{thmtrans} and \ref{thmtrans2}.
Setting $v_0 := \nabla\phi_0 \in W^{2-2/p}_p(\Omega\backslash\Sigma)^n$, the parabolic problem
\begin{equation*}
\left\{
\begin{array}{rclll}
\partial_t v - \Delta v & = & 0 & \mbox{in} & J\times\Omega\backslash\Sigma, \\[0.25em]
[\![v]\!] & = & 0 & \mbox{on} & J\times\Sigma, \\[0.25em]
[\![\partial_\nu v]\!] & = & e^{\Delta_\Sigma t} [\![\partial_\nu v_0]\!] & \mbox{on} & J\times\Sigma, \\[0.25em]
v & = & e^{\Delta_{\partial\Omega}t} \left({\left. v_0 \right|}_{\partial\Omega}\right) & \mbox{on} & J\times\partial\Omega, \\[0.25em]
v(0) & = & v_0 & \mbox{in} & \Omega\backslash\Sigma,
\end{array}
\right.
\end{equation*}
where $\Delta_\Sigma$, respectively on $\Delta_{\partial\Omega}$ denotes the Laplace-Beltrami operator on $\Sigma$ respectively $\partial\Omega$,
admits a unique solution $v \in H^1_p(J;L_p(\Omega)^n) \cap L_p(J;H^2_p(\Omega\backslash\Sigma)^n)$.
We may now define
\begin{equation*}
f := \mbox{div}\,v \in H^1_p(J;\dot{H}^{-1}_p(\Omega)) \cap L_p(J;H^1_p(\Omega\backslash\Sigma))
\end{equation*}
and by construction we have
\begin{equation*}
f(0) = \mbox{div}\,v_0 = \Delta\phi_0 = f_0.
\end{equation*}
Finally, we define  $g$ by means of $g=e^{\Delta_\Sigma t}g_0$ to the result
\begin{equation*}
g \in W^{1/2-1/2p}_p(J,L_p(\Sigma)) \cap L_p(J,W^{1-1/p}_p(\Sigma)),
\end{equation*}
by the properties of the analytic semigroup $e^{\Delta_\Sigma t}$.
\end{proof}

Next we introduce the linear operator
$L=(L_1,\ldots,L_4)$ defined by the left-hand side of (\ref{tfbns2}), i.e.
\begin{align*}
L_1(u,\pi)&:=  \rho\partial_t u - \mu\Delta u + \nabla\pi, \\
L_2(u) &: =  (\nabla|u), \\
L_3(u,q,h) & :=  [\![-\mu(\nabla u + {[\nabla u]}^{\sf{T}})]\!]\nu_\Sigma + q\nu_\Sigma - (\Delta_\Sigma h)\nu_\Sigma, \\
L_4(u,h) & :=  \partial_t h - (u|\nu_\Sigma)+ (b|\nabla_\Sigma h),
\end{align*}
and the nonlinearity
$N=(N_1,\ldots,N_4)$ defined by the right hand side of (\ref{tfbns2}), i.e.
\begin{align*}
N_1(u,\pi,h) & :=  F(h,u)\nabla u + M_4(h):\nabla^2 u + M_1(h)\nabla\pi ,\\
N_2(u,h) & :=  M_1(h):\nabla u, \\
N_3(u,h) & :=  G_\tau(h)\nabla u + (G_\nu(h)\nabla u + G_\gamma(h))\nu_\Sigma, \\
N_4(u,h) & :=  ([M_0(h)-I]\nabla_\Sigma h|u)+(u-b|\nabla_\Sigma h).
\end{align*}
For $J = [0,a]$ let the solution spaces be defined by
\begin{align*}
\mathbb{E}_1(a) &:= \{u \in H^1_p(J;L_p{(\Omega)}^n) \cap L_p(J;H^2_p{(\Omega\backslash\Sigma)}^n) : u = 0\ \textrm{on}\ \partial\Omega,\ [\![u]\!] = 0\}, \\
\mathbb{E}_2(a) &:= L_p(J;\dot{H}^1_p(\Omega\backslash\Sigma)), \\
\mathbb{E}_3(a) &:= W^{1/2-1/2p}_p(J;L_p(\Sigma)^n) \cap L_p(J;W^{1-1/p}_p(\Sigma)^n), \\
\mathbb{E}_4(a) &:= W^{2-1/2p}_p(J;L_p(\Sigma)) \cap H^1_p(J;W^{2-1/p}_p(\Sigma)) \cap L_p(J;W^{3-1/p}_p(\Sigma)).
\end{align*}
We abbreviate
$$\mathbb{E}(a) := \{(u,\pi,q,h) \in \mathbb{E}_1(a) \times \mathbb{E}_2(a) \times \mathbb{E}_3(a) \times \mathbb{E}_4(a) : [\![\pi]\!] = q\},$$
and equip $\mathbb{E}_1(a)$, $\mathbb{E}_2(a)$, $\mathbb{E}_3(a)$ and $\mathbb{E}_4(a)$
with their natural norms, which turn them into Banach spaces;
$\mathbb{E}(a)$ carries the natural norm of the underlying product space. A left subscript $0$ always means that the time trace of the function is zero whenever it exists.
Furthermore, the data spaces are defined by
\begin{align*}
\mathbb{F}_1(a) &:= L_p(J;{L_p(\Omega)}^n), \\
\mathbb{F}_2(a) &:= H^1_p(J;\dot{H}^{-1}_p(\Omega)) \cap L_p(J;H^1_p(\Omega\backslash\Sigma)), \\
\mathbb{F}_3(a) &:= W^{1/2-1/2p}_p(J;L_p{(\Sigma)}^n) \cap L_p(J;W^{1-1/p}_p{(\Sigma)}^n), \\
\mathbb{F}_4(a) &:= W^{1-1/2p}_p(J;L_p(\Sigma)) \cap L_p(J;W^{2-1/p}_p(\Sigma))
\end{align*}
and
$$\mathbb{F}(a) := \mathbb{F}_1(a) \times \mathbb{F}_2(a) \times \mathbb{F}_3(a) \times \mathbb{F}_4(a),$$
where we equip these spaces again with their natural norms. The generic elements of $\FF(a)$ are $(f,f_d,g,g_h)$.

To shorten the notation we set $z = (u,\pi,q,h) \in \mathbb{E}(a)$ and reformulate
the quasi-linear problem (\ref{tfbns2}) as
\begin{equation}\label{lwp-tfp}
Lz = N(z), \quad \quad (u(0),h(0)) = (u_0,h_0).
\end{equation}
From Section~\ref{sect-lp} we already know
that $L: \mathbb{E}(a) \to \mathbb{F}(a)$ is bounded and linear and that
$L: {_0\mathbb{E}}(a) \to {_0\mathbb{F}}(a)$ is an isomorphism, for each $a>0$.

Concerning the nonlinearity $N$, the following result has been shown in \cite[Proposition 4.1]{PrSi09b} for the case where $\Sigma$ is a graph over $\R^n$.

\begin{proposition}
\label{pro:estimates-K} Suppose $p>n+2$ and  $b\in \FF_4(a)^n$. Then
\begin{equation}
\label{K-analytic}
N\in C^\omega(\EE(a)\,,\FF(a)),\quad a>0.
\end{equation}
Let $DN(u,\pi,q,h)$ denote the Fr\'echet derivative of
$N$ at $(u,\pi,q,h)\in \EE(a)$. Then
$ DN(u,\pi,q,h) \in\Li({_0}\EE(a), {_0}\FF(a))$,
and for any number $a_0>0$
there is a positive constant
$M_0=M_0(a_0,p)$ such that
\begin{equation*}
\label{Frechet}
\begin{split}
 &|DN(u,\pi,q,h)|_{\Li({_0}\EE(a),\, {_0}\FF(a))} \\
 &\le M_0\big[
%\|\nabla h\|_{BC(J;BC)+
  |b- u|_{BC(J;BC)\,\cap\, \FF_4(a)}
 +|(u,\pi,q,h)|_{\EE(a)}  \big]\\
& + M_0\big[\big(|\nabla h|_{BC(J;BC^1)}
+|h|_{\EE_4(a)}+|u|_{BC(J;BC)}\big)|u|_{\EE_1(a)}\big]\\
&+ M_0\big[ P(|\nabla h|_{BC(J;BC)})|\nabla h|_{BC(J;BC)}
+Q\big(|\nabla h|_{BC(J;BC^1)},|h|_{\EE_4(a)}\big)|h|_{\EE_4(a)}\big]
\end{split}
\end{equation*}
for all $(u,\pi,q,h)\in \EE(a)$ and all $a\in (0,a_0]$.
Here, $P$ and $Q$ are fixed polynomials with coefficients
equal to one.
\end{proposition}
\noindent
The proof carries over to the general case considered here. The basic ingredients are still the polynomial structure of the nonlinearity $N$ w.r.t.\ $u$ and $\pi$, which is the same as in \cite{PrSi09b}, and the embeddings
$$ W^{2-2/p}_p(\Omega\backslash\Sigma)\hookrightarrow BUC^{1+\alpha}(\Omega\backslash\Sigma),\quad
W^{3-2/p}_p(\Sigma)\hookrightarrow BUC^{2+\alpha}(\Sigma),$$
with $\alpha =1-(n+2)/p>0$, which show that $\EE_3(a)$ and $\FF_4(a)$ are Banach algebras. The difference lies only in the operators $M_j(h)$ which are more complicated in the case of general domains, but analytic in $h$.

Now, we are able to establish local well-posedness.

\begin{theorem}\label{wellposed}
Fix $p>n+2$, let $\partial\Omega\in C^3$,
and suppose
$$\Gamma_0\in
W^{3-2/p}_p,\quad u_0\in W^{2-2/p}_p(\Omega\backslash\Gamma_0)^n.$$
Assume the compatibility conditions
\begin{eqnarray*}&&{\rm div} \; u_0=0\; \mbox{ in } \Omega\backslash \Gamma_0,
\quad u_0=0 \mbox{ on } \partial\Omega,\\
&&[\![\cP_{\Gamma_0}\mu(\nabla u_0+[\nabla u_0]^{\sf{T}}) \nu_{\Gamma_0}]\!]=0,\; [\![u_0]\!]=0\; \mbox{
on } \Gamma_0,\end{eqnarray*}
where
$\cP_{\Gamma_0}=I-\nu_{\Gamma_0}\times\nu_{\Gamma_0} $.

Then there exists  $a=a(u_0,\Gamma_0)>0$
and  a unique classical solution $(u,\pi,\Gamma)$
of (\ref{NS-2phase}) on $(0,a)$.  The set
$$\Upsilon=\bigcup_{t\in(0,a)}\{t\}\times\Gamma(t)$$
is a real analytic manifold, and with
$$\mho:=\{(t,x)\in(0,a)\times\Omega,\; x\not\in\Gamma(t)\},$$
the function $(u,\pi):\mho\rightarrow\R^{n+1}$ is real analytic. The transformed solution $(\bar{u},\bar{\pi},\bar{q},h)$ belongs to the space $\EE(a)$.
\end{theorem}

\begin{proof} We consider the transformed problem.\\[\baselineskip]
{\bfseries Step 1.}
Let $h_0\in W^{3-2/p}_p(\Omega)$ and $u_0\in W^{2-2/p}_p(\Omega\backslash\Sigma)^n$ be given, such that the compatibility conditions are satisfied,
and $|h_0|_{BUC^2(\Sigma)}\leq\eta$.
Let  $J_0=[0,a_0]$ and
\begin{align*}
f^\ast_d & \in H^1_p(J_0;\dot{H}^{-1}_p(\Omega\backslash\Sigma)) \cap L_p(J_0;H^1_p(\Omega\backslash\Sigma)), \\
g^\ast & \in W^{1/2-1/2p}_p(J_0;L_p(\Sigma)) \cap L_p(J_0;W^{1-1/p}_p(\Sigma))
\end{align*}
be extensions of
\begin{equation*}
(\nabla|u_0) \in \dot{H}^{-1}_p(\Omega)\cap W^{1-2/p}_p(\Omega\backslash\Sigma)
\quad \textrm{and} \quad
\cP_\Sigma[\![\mu(\nabla u_0 + {[\nabla u_0]}^{\sf{T}})\nu_\Sigma]\!] \in W^{1-3/p}_p(\Sigma),
\end{equation*}
which exist due to Proposition~\ref{lwp-extensions}.
Further choose an extension $\tilde{u}\in \EE_1(a_0)$ of $u_0$ and set $b=\tilde{u}$ restricted to $[0,a_0]\times\Sigma$.
With these extensions we may solve the problem
\begin{equation*}
Lz^\ast = (0,f^\ast_d,g^\ast,0), \quad (u^\ast(0),h^\ast(0)) = (u_0,h_0),
\end{equation*}
since all regularity and compatibility conditions of Theorem~\ref{linearproblem} are satisfied. \\[\baselineskip]
{\bfseries Step 2.}
We rewrite problem (\ref{lwp-tfp}) as
$$Lz = N(z + z^\ast) - Lz^\ast =: K(z), \quad \quad z \in {}_0\mathbb{E}(a)$$
and observe, that the solution is given as
$z = L^{-1} K(z)$, since Theorem~\ref{linearproblem} implies that $L:{}_0\mathbb{E}(a)\to{}_0\mathbb{F}(a)$
is an isomorphism with
\begin{equation*}
{|L^{-1}|}_{\cL({}_0\mathbb{E}(a),{}_0\mathbb{F}(a))} \leq M, \quad \quad a \in (0,a_0],
\end{equation*}
where $M$ is independent of  $a\leq a_0$.
Thanks to Proposition \ref{pro:estimates-K} and due to $K(0)=N(z^\ast)-Lz^\ast$, we may choose $a \in (0,a_0]$ and $r>0$ sufficiently small such that
\begin{equation*}
{|K(0)|}_{{\mathbb{F}}(a)} \leq \frac{r}{2M},\quad
{|DK(z)|}_{\cL({_0\mathbb{E}}(a);{_0\mathbb{F}}(a))} \leq \frac{1}{2M},
\quad z \in {_0\mathbb{E}(a)},\ {|z|}_{{\mathbb{E}}(a)} \leq r,
\end{equation*}
hence $$|K(z)|_{{\FF}(a)}\leq\frac{r}{M},$$
which ensures, that $L^{-1} K(z): \overline{\mathbb{B}}_r^{{_0\EE}(a)}(0) \to \overline{\mathbb{B}}^{{_0\EE}(a)}_r(0)$
is a strict contraction; see also \cite{PrSi09b}.
Thus, we may employ the contraction mapping principle to obtain a unique solution on the time interval $[0,a]$.
\\[\baselineskip]
{\bfseries Step 3.}
By Proposition \ref{pro:estimates-K} the right-hand side $N$ is real analytic, and hence,
we obtain analyticity of $(u,\pi,q,h)$ in space and time by the parameter trick  as shown in  \cite[Theorem 6.3]{PrSi09a}; cf.\ also \cite{EPS03a,EPS03b}.
\end{proof}

\noindent
At the end of this section, we want to mention an extension of Theorem~\ref{wellposed} to the case of time-weighted $L_p$-spaces.
For this purpose, for $\mu\in(1/p,1]$ we define $\EE_\mu(a)$ by means of
$$ z\in \EE_\mu(a) \quad \Leftrightarrow \quad t^{1-\mu} z\in \EE(a),$$
and similarly we define  $\FF_\mu(a)$. Thus $\EE_1(a)=\EE(a)$, $\FF_1(a)=\FF(a)$. Such time-weights are useful to relax the regularity of the initial values, but maintaining the regularity of the solution for $t\geq\delta>0$ for arbitrary small positive $\delta$. More precisely, we have the following result.
\begin{corollary} \label{time-weight}
Fix $p>n+2$, $\mu\in(\frac{1}{2}+\frac{n+2}{2p},1)$, let $\partial\Omega\in C^3$,
and suppose
$$\Gamma_0\in
W^{2+\mu-2/p}_p,\quad u_0\in W^{2\mu-2/p}_p(\Omega\backslash\Gamma_0)^n$$
are subject to the compatibility conditions \eqref{compatibilities}.

Then there exists  $a=a(u_0,\Gamma_0)>0$
and  a unique solution $(\bar{u},\bar{\pi},\bar{q},h)\in\EE_\mu(a)$ of the transformed problem \eqref{tfbns2}, which depends continuously on the initial data.
\end{corollary}
\noindent
This result is proved in the same way as Theorem \ref{wellposed}, taking into account that the restriction $\mu>\frac{1}{2}+\frac{n+2}{2p}$
still ensures the embeddings
$$ W^{2\mu-2/p}_p(\Omega\backslash\Sigma)\hookrightarrow BUC^{1}(\Omega\backslash\Sigma),\quad
W^{2+\mu-2/p}_p(\Sigma)\hookrightarrow BUC^{2}(\Sigma),$$ which are crucial for Proposition \ref{pro:estimates-K} to be valid also in
the corresponding time-weighted spaces. Further, the result about the linear problem remains valid in these time-weighted spaces as well.
This can be shown in the same way as in the case $\mu=1$, taking into account that the operator $d/dt$ admits an $H^\infty$-calculus in $L_{p,\mu}(J;X))$ with angle $\pi/2$, provided $\mu>1/p$ and $X$ is UMD-Banach space; cf.\ \cite{PrSi04}.
We refrain from giving more details, here; see also \cite{PrWi09}.

%%%%%%%%%%%%%%%%%%%%%%%%%%%%%%%%%%%%%%%%%%%%%%%%%%%%%
\section{Semiflow, Energy Functional and Equilibria}\label{sect-sfefeq}
%%%%%%%%%%%%%%%%%%%%%%%%%%%%%%%%%%%%%%%%%%%%%%%%%%%%%

\subsection{ The Induced Semiflow }
Recall that the closed $C^2$-hyper-surfaces contained in $\Omega$
form a $C^2$-manifold, which we denote by $\cMH^2(\Omega)$.  The metric on $\cMH^2(\Omega)$ is defined by
$$d(\Sigma_1,\Sigma_2):= d_H(\cN^2\Sigma_1,\cN^2\Sigma_2),\quad \Sigma_1,\Sigma_2\in\cMH^2(\Omega).$$
The charts are the parameterizations over a given real analytic hyper-surface $\Sigma$,
as described in Section 2,
and the tangent space at $\Sigma$ consist of the normal vector fields on $\Sigma$ of class $C^2$.
This way $\cMH^2(\Omega)$ becomes a Banach manifold.

Let $d_\Sigma(x)$ denote the signed distance for $\Sigma$ as introduced in Section 2. We may then define a {\em level function} $\varphi_\Sigma$ by means of
$$\varphi_\Sigma(x) = g(d_\Sigma(x)),\quad x\in\R^n,$$
where
$$g(s)=s(1-\chi(s/a))+ \chi(s/a){\rm sgn}\, s,\quad s\in \R,$$
and $\chi$ denotes the cut-off function defined in Section 2.
Then it is easy to see that $\Sigma=\varphi_\Sigma^{-1}(0)$, and $\nabla \varphi_\Sigma(x)=\nu_\Sigma(x)$, for each $x\in \Sigma$. Moreover, $\mu=0$ is an eigenvalue of $\nabla^2\varphi_\Sigma(x)$ with eigenfunction $\nu_\Sigma(x)$, the remaining eigenvalues of $\nabla^2\varphi_\Sigma(x)$ are the principal curvatures $\kappa_j(x)$ of $\Sigma$ at $x\in\Sigma$.

Consider the subset $\cMH^2(\Omega,r)$ of $\cMH^2(\Omega)$ which consists of all $\Gamma\in \cMH^2(\Omega)$ such that $\Gamma\subset \Omega$ satisfies the ball condition with fixed radius $r>0$. This implies in particular that ${\rm dist}(\Gamma,\partial\Omega)\geq r$  and all principal curvatures of $\Gamma\in\cMH^2(\Omega,r)$ are bounded by $r$. Further, the level functions $\varphi_\Gamma = g\circ d_{\Gamma}$ are well defined for $\Gamma\in\cMH^2(\Omega,r)$, and form a bounded subset of $C^2(\bar{\Omega})$. The map $\Phi:\cMH^2(\Omega,r)\to C^2(\bar{\Omega})$ defined by $\Phi(\Gamma)=\varphi_\Gamma$ is an isomorphism of the metric space $\cMH^2(\Omega,r)$ onto $\Phi(\cMH^2(\Omega,r))\subset C^2(\bar{\Omega})$.

Let $s-(n-1)/p>2$; for $\Gamma\in\cMH^2(\Omega,r)$, we define $\Gamma\in W^s_p(\Omega,r)$ if $\varphi_\Gamma\in W^s_p(\Omega)$. In this case the local charts for $\Gamma$ can be chosen of class $W^s_p$ as well. A subset $A\subset W^s_p(\Omega,r)$ is said to be (relatively) compact, if $\Phi(A)\subset W^s_p(\Omega)$ is (relatively) compact. Finally, we define
$${\rm dist}_{W^s_p}(\Gamma,\Sigma):= |\varphi_\Gamma-\varphi_\Sigma|_{W^s_p(\Omega)}$$
for $\Gamma,\Sigma\in \cMH^2(\Omega,r)$.

As an ambient space for the phase-manifold $\cPM$ of the two-phase Navier-Stokes
problem with surface tension we consider the product space $C(\bar{\Omega})^n\times \cMH^2(\Omega)$.

We define $\cPM$ as follows.
%\begin{eqnarray}\label{phasemanif}
%\cPM:=\hspace{-0.2cm}&&\hspace{-0.2cm}\{(u,\Gamma)\in C(\bar{\Omega})^n\times \cMH^2(\Omega):
%u\in W^{2-2/p}_p(G\backslash\Gamma)^n,
%\, \Gamma\in W^{3-2/p}_p,\\
%&&\hspace{-0.2cm}{\rm div}\, u =0 \mbox{ in }\Omega\backslash\Gamma,\; u=0
%\mbox{ on }
%\partial\Omega,\,
% \cP_\Gamma[\![\mu(\nabla u+[\nabla u]^{\sf T})]\!]\nu_\Gamma=0 \mbox{ on } \Gamma\}.\nonumber
%\end{eqnarray}
\begin{equation}\label{phasemanif}
\begin{array}{rcl}
\cPM & := &
\{(u,\Gamma)\in C(\bar{\Omega})^n\times \cMH^2(\Omega):
u\in W^{2-2/p}_p(G\backslash\Gamma)^n,
\, \Gamma\in W^{3-2/p}_p,\\[0.25em]
& & {\rm div}\, u =0 \mbox{ in }\Omega\backslash\Gamma,\; u=0
\mbox{ on }
\partial\Omega,\,
 \cP_\Gamma[\![\mu(\nabla u+[\nabla u]^{\sf T})]\!]\nu_\Gamma=0 \mbox{ on } \Gamma\}.
\end{array}
\end{equation}
The charts for this manifold are obtained by the charts induced by $\cMH^2(\Omega)$,
followed by a Hanzawa transformation; see Section 2.

Observe that the compatibility conditions
\begin{align*}
{\rm div}\,  u =0 \; \mbox{ in } \Omega\backslash\Gamma, \quad
u=0 \mbox{ on } \partial\Omega,\\
\cP_\Gamma[\![\mu (\nabla u+[\nabla u]^{\sf T})]\!]\nu_\Gamma=0, \; [\![u]\!]=0 \quad \mbox{ on }
\Gamma,
\end{align*}
as well as regularity are preserved by the solutions.

Applying Theorem \ref{wellposed} and re-parameterizing repeatedly, we obtain
 a {\em local semiflow} on $\cPM$.

\begin{theorem}\label{semiflow} Let $p>n+2$. Then the two-phase Navier-Stokes
problem with surface tension generates a local semiflow
on the phase-manifold $\cPM$. Each solution $(u,\Gamma)$
exists on a maximal time interval $[0,t_*)$.
\end{theorem}

\subsection{The Pressure}
The pressure does not occur explicitly as a variable in the local semiflow, the latter is
only formulated in terms of the velocity field $u$ and the free boundary $\Gamma$.
Actually, at every instant $t$ the pressure $\pi$ can be reconstructed from the semiflow.
In fact, fix any $t\in(0,t_*)$ and consider $\phi\in H^{1}_{p^\prime}(\Omega)$.
Then we have by the divergence theorem
$$(u(t)|\nabla\phi)_{L_2(\Omega)}= -({\rm div}\, u|\phi)_{L_2(\Omega)}
-([\![(u|\nu )]\!]|\phi)_{L_2(\Gamma)}=0,$$
hence also $(\partial_tu|\nabla\phi)=0$. This implies, multiplying the momentum balance
divided by $\rho$ with $\nabla\phi$ in $L_2(\Omega)$
$$ \left.\left(\nabla\frac{\pi}{\rho}\,\right|\nabla\phi\right)_{L_2(\Omega)}=
\left.\left(\frac{\mu}{\rho}\Delta u-u\cdot\nabla u\,\right|\nabla\phi\right)_{L_2(\Omega)}.$$
On the other hand, multiplying the stress boundary condition by $\nu_\Gamma$ yields
$$[\![\pi]\!]= \sigma H_\Gamma +([\![\mu(\nabla u+[\nabla u]^{\sf T})]\!]\nu_\Gamma|\nu_\Gamma)$$
on $\Gamma$. Thus $\pi$ must satisfy the following problem.
\begin{align}\label{pressure-reconstr}
\left.\left(\nabla\frac{\pi}{\rho}\,\right|\nabla\phi\right)_{L_2(\Omega)}=
\left.\left(\frac{\mu}{\rho}\Delta u-u\cdot\nabla u\,\right|\nabla\phi\right)_{L_2(\Omega)},&
\quad \phi\in H^{1}_{p^\prime}(\Omega)\\
[\![\pi]\!]= \sigma H_\Gamma +([\![\mu(\nabla u+[\nabla u]^{\sf T})]\!]\nu_\Gamma|\nu_\Gamma),&
\quad x\in\Gamma.\nonumber
\end{align}
Theorem \ref{thmtrans} implies that this problem has a unique solution
$\pi\in \dot{H}^1_p( \Omega\backslash\Gamma)$. Thus the pressure is uniquely
defined (up to a constant) by the semiflow and can be obtained by solving the transmission problem
 \eqref{pressure-reconstr}.

\subsection{The Energy Functional}
Define the {\em energy functional} by means of
$$\Phi(u,\Gamma):=\frac{1}{2} |\rho^{1/2}u|_{L_2(\Omega)}^2
+\sigma|\Gamma(t)|. $$
Then
\begin{equation}\label{LF}
\partial_t \Phi(u,\Gamma) + 2 |\mu^{1/2}E|^2_{L_2(\Omega)}=0,\nonumber
\end{equation}
hence the energy functional is a Ljapunov functional, in fact, even a strict one.
We have the following result.

\begin{proposition}\label{energy}
Let $\rho_i,\mu_i,\sigma>0$ be constants. Then

\noindent
(a) The {\em energy equality} is valid for smooth solutions.\\
(b) The {\em equilibria} are zero velocities, constant pressures in the components
of the phases, the dispersed phase is a union of  nonintersecting open balls. \\
(c) The {\em energy functional} is a strict Ljapunov-functional.\\
(d) The {\em critical points} of the energy functional for constant
phase volumes are precisely the equilibria.
\end{proposition}

\noindent
This result is a special case of \cite[Theorem 3.1]{BP07}.

\begin{remark}
(i) Let us point out that in equilibrium the dispersed phase consists of at most
countably many disjoint balls $B_{R_i}(x_i)$. If there are infinitely many of them,
then $R_i\to0$ as $i\to\infty$, hence the corresponding curvatures $H_i= -(n-1)/R_i$
tend to infinity, as well as the pressures inside these balls. This is due to the model
assumption that there is no phase transition. On the other hand, phase transition will
occur at very high pressure levels. To avoid this contradiction,
in the sequel we consider only equilibria in which the dispersed phase consists
of only finitely many balls. Note also that the free boundary will not
be of class $C^2$ if $\Omega_1$ has infinitely many components.\\
(ii) There is another pathological case which we exclude in the sequel, namely if
the dispersed phase contains balls touching each other. This can only happen if the radii
of these balls are equal, otherwise the pressure jump would not be constant on $\Gamma$.
Physically one would expect such an equilibrium  to be unstable,
but at present we are not able to handle this case. Observe that also in such a situation
the free boundary $\Gamma$ is not a manifold of class $C^2$.\\
(iii) Of course neither (i) or (ii) occurs if we assume that $\Omega_1$ is connected,
the continuous phase enjoys this property anyway.
\end{remark}

%\newpage

%%%%%%%%%%%%%%%%%%%%%%%%%%%%%%
\section{The Stability Result} \label{sect-sr}
%%%%%%%%%%%%%%%%%%%%%%%%%%%%%%

Assuming, for simplicity, that the phases are connected, we denote by
$$\cE:=\{(0,S_R(x_0)):\, x_0\in\Omega, \, R>0,\, \bar{B}_R(x_0)\subset\Omega\}$$
the set of equilibria without boundary contact. Note that $\cE$ forms a
real analytic manifold of dimension $n+1$. Here $n$ dimensions come from the coordinates
of the center $x_0$ and one from the radius $R$ of the sphere $S_R(x_0)$.

Fix any such equilibrium $(0,\Sigma)\in\cE$. We consider the behaviour of the
solutions near this steady  state.
Suppose $p>n+2$, let $\partial\Omega\in C^3$,
and consider initial data $(u_0,\Gamma_0)\in\cPM$.

Here we have to use the full linearization of the problem at
an equilibrium $(0,\Sigma)$ i.e.\ at $(u,h)=(0,0)$, and for this reason
we have to replace $\Delta_\Sigma$ in the linear problem \eqref{linFB} by
$$ \cA_\Sigma = H^\prime_\Gamma(0) = \frac{n-1}{R^2} +\Delta_\Sigma.$$
This results in the problem
\begin{equation}\label{linStokes}
\begin{array}{rcll}
\rho\partial_tv -\mu\Delta v+\nabla q & = & \rho f_v & \mbox{ in }\Omega\backslash\Sigma,\\[0.25em]
\mbox{div}\,v & = & f_d & \mbox{ in }\Omega\backslash\Sigma,\\[0.25em]
[\![-\mu([\nabla v]+[\nabla v]^{\sf T})+q]\!]\nu_\Sigma - \sigma(\cA_\Sigma h) \nu_\Sigma & = & g & \mbox{ on } \Sigma,\\[0.25em]
[\![v]\!] & = & 0 & \mbox{ on } \Sigma,\\[0.25em]
v & = & 0 & \mbox{ on } \partial\Omega,\\[0.25em]
\partial_th - (v|\nu_\Sigma) & = & g_h & \mbox{ on } \Sigma,\\[0.25em]
v(0) = v_0 \; \mbox { in } \Omega\backslash\Sigma, \quad h(0) & = & h_0 & \mbox{ on } \Sigma.
\end{array}
\end{equation}
It is well-known that $\cA_\Sigma$ is selfadjoint, negative semidefinite
on functions with zero mean, and has compact resolvent in $L_2(\Sigma)$.
$\lambda_0=0$ is an eigenvalue with eigenspace of dimension $n$, spanned by the
spherical harmonics of degree one.
$\lambda_{-1}=(n-1)/R^2$ is also an eigenvalue, its eigenspace is one-dimensional
and consists of the constants.

As a base space for our analysis we use
$$X_0=L_{p,\sigma}(\Omega)^n\times W^{2-1/p}_p(\Sigma),$$
where the subscript $\sigma$ means solenoidal, and we set
$$\bar{X}_1=H^2_{p}(\Omega\backslash\Sigma)^n
\times W^{3-1/p}_p(\Sigma).$$
Define a closed linear operator in $X_0$ by means of
$$A(v,h)=
( -(\mu/\rho) \Delta v +\rho^{-1}\nabla q,-(v|\nu_\Sigma)),$$
with domain $X_1:=D(A)\subset \bar{X}_1$ defined by
\begin{eqnarray*}
D(A) \hspace{-0.3cm}&&=\{(v,h)\in \bar{X}_1\cap X_0:
v=0
\mbox{ on }\partial\Omega,\;
[\![v]\!]=0  \mbox{ and } \\
&&\qquad[\![\cP_\Sigma\mu(\nabla v+\nabla v^{\sf T})\nu_\Sigma]\!]
=0 \mbox{ on } \Sigma\},
\end{eqnarray*}
where as before $\cP_\Sigma$ means the projection onto the tangent space of
$\Sigma$. \\
Here $q\in \dot{H}^1_p(\Omega\backslash \Sigma)$
 is determined as the solution of the transmission problem
\begin{eqnarray*}
&&\left(\nabla \frac{q}{\rho}|\nabla\phi\right)_{L_2} = \left(\frac{\mu}{\rho}\Delta v|\nabla \phi\right)_{L_2},
\quad \phi\in W^1_{p^\prime}(\Omega),\\[0.5em]
&& \mbox{}[\![ q]\!] = [\![\mu((\nabla v + [\nabla v]^{\sf T})\nu_\Sigma|\nu_\Sigma)]\!]
+\sigma\cA_\Sigma h
\quad \mbox{ on } \Sigma,
\end{eqnarray*}
which is well-defined (up to a constant) by Theorem \ref{thmtrans}.
This implies
$$ \frac{1}{\rho}\nabla q= T_1\left(\frac{\mu}{\rho}\Delta v\right) + T_2\big(([\![\mu(\nabla v +[\nabla v]^{\sf T})]\!]\nu_\Sigma|\nu_\Sigma)+\sigma\cA_\Sigma h\big).$$
Then with $z=(v,h)$ and $f=(f_v,g_h)$ as well as
$z_0=(v_0,h_0)$, system (\ref{linStokes}) can be
rewritten as the abstract evolution equation
\begin{equation} \label{ACP}\dot{z}+Az=f,\quad t>0,\quad z(0)=z_0,\end{equation}
in $X_0$, provided $f_d\equiv g\equiv 0$; see also \cite{ShSh09}.

Since by Theorem \ref{linearproblem} problem (\ref{linStokes}) has maximal $L_p$-regularity,
the abstract problem (\ref{ACP}) has maximal $L_p$-regularity, as well.
In particular, $-A$ generates an analytic $C_0$-semigroup in $X_0$; see e.g.\ \cite{Pru03}, Proposition 1.2.
In addition we have the following result.

\begin{proposition}
Let $\rho_i,\mu_i >0, \sigma>0$ be constants, $p\in(1,\infty)$,
and let $X_0$, $A$, $X_1:=D(A)$ be defined as above.
Then the following holds.\\
(a) The linear operator $-A$ generates a compact analytic
$C_0$-semigroup in $X_0$ which has the property of maximal
$L_p$-regularity.\\
(b) The spectrum of $A$ consists of countably many eigenvalues
with finite algebraic multiplicity and is independent of $p$.\\
(c) $A$ has no eigenvalues $\lambda$
with nonnegative real part other than $\lambda=0$.\\
(d) $\lambda=0$ is a semisimple eigenvalue with multiplicity $n+1$.\\
(e) The eigenspace $N(A)$ is isomorphic to the tangent space
$T_{z^*}\cE$ of $\cE$ at the given equilibrium
 $z^*=(0, \Sigma)$,
where $\Sigma = S_R(x_0)$.\\
(f) The restriction of $e^{-At}$ to $R(A)$ is exponentially stable.
\end{proposition}

This result is a special case of  \cite[Theorem 4.1]{BP08}. It shows that equilibria $(0,\Sigma)\in\cE$ are normally stable,
 hence allows for the use of the {\em generalized principle of
linearized stability} to obtain our main result on stability and convergence.

The following result concerns the stationary Stokes problem
\begin{equation}
\label{aux-problem}
\begin{aligned}
\rho\omega u -\mu\Delta u+\nabla \pi&=0
    &\ \hbox{in}\quad &\Omega\backslash\Sigma,\\
{\rm div}\,u&= f_d& \hbox{in}\quad &\Omega\backslash\Sigma,\\
-\cP_\Sigma[\![\mu(\nabla u+[\nabla u]^{\sf T})]\!]\nu_\Sigma &=g_\tau
	&\ \hbox{on}\quad &\Sigma,\\
-([\![\mu(\nabla u+[\nabla u]^{\sf T})]\!]\nu_\Sigma|\nu_\Sigma) + [\![\pi]\!]&= g_\nu
	&\ \hbox{on}\quad &\Sigma,\\
[\![u]\!]&=0 &\ \hbox{on}\quad &\Sigma,\\
u&=0 &\ \hbox{on} \quad&\partial\Omega.
\end{aligned}
\end{equation}
It is needed in the proof of the main result of this section.
\begin{proposition}\label{proellprb}
Let $p>3$ be fixed, and assume that $\rho_i$ and $\mu_i$ are positive
constants for $i=1,2$, and that $\omega>0$ is large enough.
Then the stationary Stokes problem with free boundary \eqref{aux-problem} admits a unique solution $(u,\pi,[\![\pi]\!])$ with regularity
\begin{equation}
\begin{aligned}
&u\in W^{2-2/p}_p(\Omega\backslash\Sigma)^n,\quad \pi\in \dot{W}^{1-2/p}_p(\Omega\backslash\Sigma),\quad
[\![\pi]\!]\in W^{1-3/p}_p(\Sigma),
\end{aligned}
\end{equation}
if and only if the data
$(f_d,g)$
satisfy the following regularity  conditions:
\begin{itemize}
\item[(a)]
$f_d\in  W^{1-2/p}_p(\Omega\backslash\Sigma)\cap \dot{H}^{-1}_p(\Omega)$,
\vspace{1mm}
\item[(b)]
$g=(g_\tau,g_\nu)\in W^{1-3/p}_p(\Sigma)^n$.
\end{itemize}
The solution map $[(f_d,g)\mapsto (u,\pi, [\![\pi]\!])]$
is continuous between the corresponding spaces.
\end{proposition}

The proof of this elliptic problem is similar to that of Theorem \ref{linearproblem}; we refer to \cite{PrSi09c} for the case of a flat interface.

The main result of this section is the following.

\begin{theorem}\label{stability} The equilibrium $(0,\Sigma)\in\cE$ is stable in the sense
that for each $\epsilon\in(0,\epsilon_0]$
there exists $\delta(\epsilon)>0$ such that for all initial values
$(u_0,\Gamma_0)$ subject to
$${\rm dist}_{W^{3-2/p}_p}(\Gamma_0,\Sigma)\leq \delta(\epsilon)\quad \mbox{ and }
\quad
\|u_0\|_{W^{2-2/p}_p(\Omega\backslash\Gamma_0)}\leq \delta(\epsilon)$$ there exists
a unique global solution $(u(t),\Gamma(t))$
of the problem, and it satisfies
$${\rm dist}_{W^{3-2/p}_p}(\Gamma(t),\Sigma)\leq \epsilon\quad \mbox{ and } \quad
\|u(t)\|_{W^{2-2/p}_p(\Omega\backslash\Gamma(t))}\leq \epsilon,\; t\geq0.$$
Moreover, as $t\to\infty$ the solutions $(u(t),\Gamma(t))$
converges to an equilibrium $(0,\Sigma_\infty)$
in the same topology, i.e.
$$\lim_{t\to\infty} {\rm dist}_{W^{3-2/p}_p}(\Gamma(t),\Sigma_\infty)+
\|u(t)\|_{W^{2-2/p}_p(\Omega\backslash\Gamma(t))}=0.$$
The convergence is at exponential rate.
\end{theorem}

\begin{proof}
1.\,  To prepare, as in \cite{PrSi09c} we first parameterize the nonlinear phase manifold locally near $(0,\Sigma)$ over
\begin{align*}
X_\gamma&:=\{(u,h)\in [W^{2-2/p}_p(\Omega\backslash\Sigma)\times W^{3-2/p}_p(\Sigma)]\cap X_0:\, u=0 \mbox{ on }\partial\Omega,\\
&\qquad [\![u]\!]=0,\, \cP_\Sigma[\![\mu(\nabla u +[\nabla u]^{\sf T})]\!]\nu_\Sigma=0 \mbox{ on } \Sigma\}.
 \end{align*}
In particular, this will show that $X_\gamma$ is isomorphic to the tangent space $T_{(0,\Sigma)}\cP M$.

 For this purpose fix $\omega>0$ and solve for given $\tilde{z}:=(\tilde{u},\tilde{h})\in B_r^{X_\gamma}(0)$ the problem
\begin{equation}
\label{nonlinstokes1}
\begin{aligned}
\rho\omega \bar{u} -\mu\Delta \bar{u}+\nabla \bar{\pi}&=0
    &\ \hbox{in}\quad &\Omega\backslash\Sigma,\\
{\rm div}\,\bar{u}&= M_1(h):\nabla u& \hbox{in}\quad &\Omega\backslash\Sigma\\
-\cP_\Sigma[\![\mu(\nabla\bar{u}+[\nabla \bar{u}]^{\sf T})]\!]\nu_\Sigma &=G_\tau(h)\nabla u
	&\ \hbox{on}\quad &\Sigma,\\
-([\![\mu(\nabla \bar{u}
+[\nabla \bar{u}]^{\sf T})]\!]\nu_\Sigma|\nu_\Sigma) + [\![\bar{\pi}]\!]&= G_\nu(h)\nabla u +G_\gamma(h)
	&\ \hbox{on}\quad &\Sigma,\\
[\![\bar{u}]\!] &=0 &\ \hbox{on}\quad &\Sigma,\\
\bar{u}&=0 &\ \hbox{on} \quad&\partial\Omega,
\end{aligned}
\end{equation}
where $u=\bar{u}+\tilde{u}$ and $h=\tilde{h}$, i.e.\ $\bar{h}=0$.
We write this equation in short hand notation as  $L_\omega \bar{u}=N(\bar{z}+\tilde{z})$ in $ \bar{X}_\gamma=W^{2-2/p}_p(\Omega\backslash\Sigma)\cap H^1_p(\Omega)$. It is easily shown that $N$ is real analytic and  $N^\prime(0)=0$; see Section \ref{sect-lwp}. $L_\omega$ is invertible by Proposition \ref{proellprb}, hence the implicit function theorem yields a unique solution $\bar{u}=\phi(\tilde{u},\tilde{h})\in \bar{X}_\gamma$ near 0. $\phi$ is real analytic as well, and satisfies $\phi^\prime(0)=0$. Then we define
$$ \Phi(\tilde{u},\tilde{h})= (\tilde{u},\tilde{h}) +(\phi(\tilde{u},\tilde{h}),0).$$
Obviously, $\Phi$ is real analytic, $\Phi^\prime(0)=I$, $\Phi(B^{X_\gamma}_\rho(0))\subset \cPM$, and  $\Phi$ is injective.

Hence it remains to show local surjectivity near 0. So suppose that $\bar{z}:=(\bar{u},\bar{h})\in\cPM$ has sufficiently small norm. Solving the problem
\begin{equation}
\label{nonlinstokes2}
\begin{aligned}
\rho\omega u -\mu\Delta u+\nabla \pi&=0
    &\ \hbox{in}\quad &\Omega\backslash\Sigma,\\
{\rm div}\,u&= M_1(\bar{h}):\nabla \bar {u}& \hbox{in}\quad &\Omega\backslash\Sigma,\\
-\cP_\Sigma[\![\mu(\nabla u+[\nabla u]^{\sf T})]\!]\nu_\Sigma &=G_\tau(\bar{h})\nabla \bar{u}
	&\ \hbox{on}\quad &\Sigma,\\
-([\![\mu(\nabla u+[\nabla u]^{\sf T})]\!]\nu_\Sigma|\nu_\Sigma) + [\![\pi]\!]&= G_\nu(\bar{h})\nabla \bar{u} +G_\gamma(\bar{h})
	&\ \hbox{on}\quad &\Sigma,\\
[\![u]\!] &=0 &\ \hbox{on}\quad &\Sigma,\\
u&=0 &\ \hbox{on} \quad&\partial\Omega,
\end{aligned}
\end{equation}
by means of Proposition \ref{proellprb},
$(\tilde{u},\tilde{h}):= (\bar{u}-u,\bar{h})$ belongs to $X_\gamma$ and $\phi(\tilde{u},\tilde{h})=u$, showing surjectivity of $\Phi$ near 0. In particular, $\cPM$ is a real analytic manifold near $(0,\Sigma)$.

\bigskip

\noindent
2.\, Let $(u,\pi,h)$ be a solution on its maximal time interval $[0,t_*)$. In this step we decompose  $u=\bar{u}+\tilde{u}$, $\pi=\bar{\pi}+\tilde{\pi}$, $h=\bar{h}+\tilde{h}$, where $(\bar{u},\bar{\pi},\bar{h})$ solves the problem
\begin{equation}
\label{nonlinstokes3}
\begin{aligned}
\rho\omega \bar{u} +\rho\partial_t \bar{u}-\mu\Delta \bar{u}+\nabla \bar{\pi}&=F_u(u,\pi,h)
    &\hbox{in}\quad &\Omega\backslash\Sigma,\; t>0,\\
{\rm div}\,\bar{u}&=M_1(h):\nabla u&\hbox{in}\quad &\Omega\backslash\Sigma,\; t>0,\\
-\cP_\Sigma[\![\mu E(\bar{u})]\!]\nu_\Sigma &= G_\tau(h)\nabla u
	&\ \hbox{on}\quad &\Sigma,\; t>0,\\
-([\![\mu E(\bar{u})]\!]\nu_\Sigma|\nu_\Sigma)+[\![\bar{\pi}]\!]-\sigma\cA_\Sigma\bar{h}&=G_\nu(h)\nabla u\\ &+G_\gamma(h)-G_\gamma(h-\bar{h})
	&\ \hbox{on}\quad &\Sigma,\; t>0,\\
[\![\bar{u}]\!] &=0 &\ \hbox{on}\quad &\Sigma,\; t>0,\\
\bar{u}&=0&\hbox{on}\quad &\partial\Omega,\; t>0,\\
\omega \bar{h}+\partial_t\bar{h}-(\bar{u}|\nu_\Sigma)&=(M_0(h)\nabla_\Sigma h|u) &\ \hbox{on}\quad &\Sigma,\; t>0,\\
\bar{u}(0)=\bar{u}_0 \quad \hbox{in}\quad \Omega\backslash\Sigma,\quad \bar{h}(0)&=\bar{h}_0 & \hbox{on}\quad & \Sigma,
\end{aligned}
\end{equation}
with $\bar{z}_0=(\bar{u}_0,\bar{h}_0)=(\phi(\tilde{u}_0,\tilde{h}_0),0)$ and $E(\bar{u}):=\nabla \bar{u}+[\nabla \bar{u}]^{\sf T}$. Writing this problem
abstractly as $\LL_\omega \bar{z} =\NN(\bar{z}+\tilde{z})$ , by the implicit function theorem we obtain a unique solution $\bar{z}=\bar{z}(\tilde{z})$
in the function space $\EE(a)$ for each $a<t_*$.
Then $\tilde{z}$ is determined by the problem
\begin{equation}
\label{nonlinstokes4}
\begin{aligned}
\rho\partial_t \tilde{u}-\mu\Delta \tilde{u}+\nabla \tilde{\pi}&=\rho\omega \bar{u}
    &\ \hbox{in}\quad &\Omega\backslash\Sigma,\; t>0,\\
{\rm div}\,\tilde{u}&= 0& \hbox{in}\quad &\Omega\backslash\Sigma,\;t>0,\\
-\cP_\Sigma[\![\mu(\nabla\tilde{u}+[\nabla\tilde{u}]^{\sf T})]\!]\nu_\Sigma &=0
	&\ \hbox{on}\quad &\Sigma,\; t>0,\\
-([\![\mu(\nabla\tilde{u}+[\nabla\tilde{u}]^{\sf T})]\!]\nu_\Sigma|\nu_\Sigma)+[\![\tilde{\pi}]\!]-\sigma\cA_\Sigma\tilde{h}&=G_\gamma(\tilde{h})
	&\ \hbox{on}\quad &\Sigma,\; t>0,\\
[\![\tilde{u}]\!] &=0 &\ \hbox{on}\quad &\Sigma,\; t>0,\\
\tilde{u}&=0&\hbox{on}\quad &\partial\Omega,\; t>0,\\
\partial_t\tilde{h}-(\tilde{u}|\nu_\Sigma)&=\omega\bar{h} &\ \hbox{on}\quad &\Sigma,\; t>0,\\
\tilde{u}(0)=\tilde{u}_0 \quad \hbox{in}\quad \Omega\backslash\Sigma,\quad \tilde{h}(0)&=\tilde{h}_0 & \hbox{on}\quad & \Sigma.
\end{aligned}
\end{equation}
The last equation can be rewritten abstractly in $X_0$ employing the operator $A$ introduced above as
\begin{equation}\label{redns}\dot{\tilde{z}}+A\tilde{z}=R( \tilde{z}),\;t>0,\quad \tilde{z}(0)=\tilde{z}_0,\end{equation}
where
$$R(\tilde{z})=( \omega(I-T_1)\bar{u}(\tilde{z}) -T_2G_\gamma(\tilde{z}),\omega \bar{h}(\tilde{z})).$$
Note that $\bar{z}$ is a causal functional of $\tilde{z}$.

\bigskip

\noindent
3.\, Problem \eqref{redns} is of the form studied in \cite{PSZ09}, where the generalized principle of linearized stability is proved for abstract parabolic quasilinear problems of the form \eqref{redns}. The only difference is that here a part of $R$ is nonlocal, but causal in time. Therefore we only comment on the required modifications in the proof of Theorem 2.1 in \cite{PSZ09}. For this purpose we decompose
$$ R(\tilde{z})= R_{nloc}(\tilde{z})+R_{loc}(\tilde{z}):= ( \omega(I-T_1)\bar{u}(\tilde{z}),\omega \bar{h}(\tilde{z}))+(-T_2G_\gamma(\tilde{z}),0).$$
Observe that by construction, if $z$ is an equilibrium then $z=\tilde{z}$, hence $\bar{z}=0$. Therefore the equilibria are determined by the equation
$Az_*= R_{loc}(z_*)$. Further, we have an estimate of the form
$$ |R_{nloc}(\tilde{z})|_{\FF(t_0)}\leq \ve |\tilde{u}|_{\EE(t_0)},$$
provided $\tilde{z}_0$ is small in the norm of $X_\gamma$.

Let $P^c$ denote the projection in $X_0$ onto the kernel $N(A)$ along the range $R(A)$ of $A$ and let $P^s=I-P^c$ the complementary projection onto $R(A)$.
 As in the proof of Theorem 2.1 of \cite{PSZ09} we parameterize the set of equilibria $\cE$ near $0$ over $N(A)$ via a $C^1$-map ${\cxx}\mapsto {\cxx} +\Psi({\cxx})$ such that $\Psi(0)=\Psi^\prime(0)=0$. By ${\cxx}:=P^c\tilde{z}$ and ${\cy}:= P^s\tilde{z}-\Psi(P^c\tilde{z})$ we introduce the normal form of the problem which reads as follows.
\begin{align}\label{normalform}
\dot{\cxx}&= T(\cxx,\cy),\quad \cxx(0)=\cxx_0,\\
\dot{\cy} + A^s\cy &= S(\cxx,\cy),\quad \cy(0)=\cy_0,\nonumber
\end{align}
where \begin{align*}
T(\cxx,\cy)&=P^c[R({\cxx}+\Psi({\cxx})+{\cy})-R({\cxx}+\Psi({\cxx}))],\\
S(\cxx,\cy)&=P^s[R({\cxx}+\Psi({\cxx})+{\cy})-R({\cxx}+\Psi({\cxx}))]-\Psi^\prime({\cxx})T({\cxx},{\cy}),
\end{align*}
since $P^cR({\cxx}+\Psi({\cxx}))=0$ and $P^sR({\cxx}+\Psi({\cxx}))= A^s\Psi({\cxx})= A^s[{\cxx}+\Psi({\cxx})]$.
The first component of $P^c\tilde{z}$ equals zero, since the eigenfunctions of $A$ for eigenvalue 0 have vanishing velocity part. This implies
$|\tilde{u}|_{\EE_1(t_0)}\leq |{\cy}|_{\EE(t_0)}$, hence the nonlocal part is estimated as
$$ |R_{nloc}(\tilde{z})|_{\FF(t_0)}\leq \ve |\tilde{u}|_{\EE_1(t_0)}\leq\ve|{\cy}|_{\EE(t_0)},$$ for any $t_0>0$.
The local parts of $T$ and $R$ can be estimated by $\ve|{\cy}|_{\EE(t_0)}$, as in the proof of Theorem 2.1 in \cite{PSZ09}.
Taking these observations into account, we may proceed as in \cite{PSZ09} to prove global existence of $\tilde{z}$, its stability and convergence to another equilibrium, provided $\tilde{z}_0$ is small in $X_\gamma$.
\end{proof}

%%%%%%%%%%%%%%%%%%%%%%%%%%%%%%%%%%%%%%%%%%%%%%
\section{ Global Existence and Convergence}\label{sect-gec}
%%%%%%%%%%%%%%%%%%%%%%%%%%%%%%%%%%%%%%%%%%%%%

\noindent
Again we assume for simplicity that the phases are connected.
There are basically two obstructions against global existence:

\noindent
- {\bf regularity}: the norms of either $u(t)$ or $\Gamma(t)$ become unbounded;\\
- {\bf geometry}: the topology of the interface changes,
or the interface touches the boundary of $\Omega$.

Note that the {\em phase volumes} are  preserved by the semiflow.

We say that a solution $(u,\Gamma)$ satisfies a {\bf uniform ball condition},
if there is a radius $r>0$ such that $\Gamma([0,t_*))\subset \cMH^2(\Omega,r)$.
Note that this condition
bounds the curvature of $\Gamma(t)$, and prevents it to touch the outer
boundary $\partial G$, or to undergo topological changes.

Combining the above results, we obtain the following result
on the asymptotic behavior of solutions.

\begin{theorem}\label{convergence} Let $p>n+2$. Suppose that $(u,\Gamma)$
is a solution of the two-phase Navier-Stokes problem with surface tension
on the maximal time interval $[0,t_*)$.
Assume the following on $[0,t_*)$:\\
(i)  \, $|u(t)|_{W^{2-2/p}_p}+|\Gamma(t)|_{W^{3-2/p}_p}\leq M<\infty$;\\
(ii) \, $(u,\Gamma)$ satisfies a uniform ball condition. \\
Then $t_*=\infty$, i.e.\ the solution exists globally, and it
converges in $\cPM$ to an equilibrium $(0,\Gamma_\infty)\in\cE$.
Conversely, if $(u,\Gamma)$ is a global solution in $\cPM$ which converges to an equilibrium $(u_\infty,\Gamma_\infty)\in\cE$ in $\cPM$,
then (i) and (ii) are valid.
\end{theorem}

\begin{proof} Assume that  (i) and (ii) are valid. Then $\Gamma([0,t_*))\subset W^{3-2/p}_p(\Omega,r)$ is bounded, hence relatively compact in
$W^{3-2/p-\ve}_p(\Omega,r)$. Thus we may cover $\Gamma([0,t_*))$ by finitely many balls with centers $\Sigma_k$  such that
${\rm dist}_{W^{3-2/p-\ve}}(\Gamma(t),\Sigma_j)\leq \delta$ for some $j=j(t)$, $t\in[0,t_*)$. Let $J_k=\{t\in[0,t_*):\, j(t)=k\}$; using for each $k$ a Hanzawa-transformation $\Theta_k$, we see that the pull backs $\{u(t,\cdot)\circ\Theta_k:\, t\in J_k\}$ are bounded in $W^{2-2/p}_p(\Omega\backslash \Sigma_k)$, hence relatively compact in $W^{2-2/p-\ve}_p(\Omega\backslash\Sigma_k)$. Employing now Corollary \ref{time-weight} with $\mu=1-\varepsilon$ we obtain solutions
 $(u^1,\Gamma^1)$ with initial configurations $(u(t),\Gamma(t))$ in the phase manifold on a common time interval say $(0,a]$, and by uniqueness we have
 $(u^1(a),\Gamma^1(a))=(u(t+a),\Gamma(t+a))$. Continuous dependence implies then relative compactness of $\{(u(\cdot),\Gamma(\cdot)):\, 0\leq t<t_*\}$ in $\cPM$, in particular $t_*=\infty$ and the orbit $(u,\Gamma)(\R_+)\subset\cPM$ is relatively compact.
The energy is a strict Ljapunov functional, hence the limit set $\omega(u,\Gamma)$
of a solution is contained in the set $\cE$ of equilibria.
By compactness $\omega(u,\Gamma)\subset \cPM$ is non-empty, hence the solution comes close to $\cE$. Finally, we apply the convergence result Theorem \ref{stability} to complete the sufficiency part of the proof. Necessity follows by a compactness argument.
\end{proof}

%%%%%%%%%%%%%%%%%%%%%%%%%%%%%%%%%%
\section{Appendix: Transmission Problems} \label{sect-tp}
%%%%%%%%%%%%%%%%%%%%%%%%%%%%%%%%%%
\noindent
In this section we provide some results, concerning the existence and uniqueness of solutions to the transmission problem
    \begin{align}
    \begin{split}\label{trans0}
    \lambda q-\Delta q&=f,\quad x\in \Omega\backslash\Gamma\\
    [\![\rho q]\!]&=g,\quad x\in\Gamma,\\
    [\![\partial_{\nu_\Gamma} q]\!]&=h_1,\quad x\in\Gamma,\\
    \delta\partial_{\nu_\Omega}q+(1-\delta)q&=h_{2,\delta},\quad x\in\partial\Omega,\ \delta\in\{0,1\},
    \end{split}
    \end{align}
where $\lambda\ge 0$,
    $$\rho(x):=\rho_1\chi_{\Omega_1}(x)+\rho_2\chi_{\Omega_2}(x),\quad x\in\Omega\backslash\Gamma,$$
and $\rho_j>0$. To be precise, we will study \eqref{trans0} in different functional analytic settings. We begin by stating the result for the 'classical' case, i.e.\ if the basic space is given by $L_p(\Omega)$.
\begin{theorem}\label{thmtrans0}
Let $\Omega\subset\R^n$ open, $1<p<\infty$, $f\in L_p(\Omega)$, $g\in W_p^{2-1/p}(\Gamma)$, $h_1\in W_p^{1-1/p}(\Gamma)$ and $h_{2,\delta}\in W_p^{2-\delta-1/p}(\partial\Omega)$, $\delta\in\{0,1\}$ be given. Then, for each $\lambda>0$, there exists a unique solution $q\in H_p^2(\Omega\backslash\Gamma)$ of \eqref{trans0} and a constant $C_1>0$ such that
    $$|q|_{H_p^2(\Omega\backslash\Gamma)}\le C_1\left(|f|_{L_p(\Omega)}+|g|_{W_p^{2-1/p}(\Gamma)}+|h_1|_{W_p^{1-1/p}(\Gamma)}+|h_{2,\delta}|_{W_p^{2-\delta-1/p}(\partial\Omega)}\right).$$
If in addition $J=[0,a]$, $f=f(t,x)$, $f\in H_p^1(J;L_p(\Omega))$ and $g=h_1=h_{2,\delta}=0$, then for each $\lambda>0$, there exists a unique solution $q\in H_p^1(J;H_p^2(\Omega\backslash\Gamma))$, and the estimate
    $$||q||_{H_p^1(J;H_p^2(\Omega\backslash\Gamma))}\le C_2||f||_{H_p^1(J;L_p(\Omega))}$$
holds with some constant $C_2>0$.
\end{theorem}
\begin{proof}
The first assertion basically follows from \cite{DHP07}, since the Lopatinskii-Shapiro condition is satisfied at $\Gamma$ and $\partial\Omega$. The second assertion follows from the first one by differentiating \eqref{trans0} w.r.t.\ $t$ and by employing the uniqueness of the solution of \eqref{trans0}.
\end{proof}
We will also need a result for the case $\lambda=0$. To this end, let $\Omega\subset\R^n$ be a bounded domain, $g=h_1=h_{2,\delta}=0$ and $f\in L_p(\Omega)$. Define $A_\delta$ by $A_\delta q=-\Delta q$, with domain
    \begin{multline*}
    D(A_\delta)=\{q\in H_p^2(\Omega\backslash\Gamma):[\![\rho q]\!]=[\![\partial_{\nu_{\Gamma}}q]\!]=0\ \mbox{on}\ \Gamma,\\
    \delta\partial_{\nu_\Omega}q+(1-\delta)q=0,\ \mbox{on}\ \partial\Omega\},\ \delta\in\{0,1\}.
    \end{multline*}
Since
    $$D(A_\delta)\compemb L_p(\Omega),$$
the resolvent of $A_\delta$ is compact and therefore the spectral set $\sigma(A_\delta)$
consists solely of a countably infinite sequence of isolated
eigenvalues. In case $\delta=1$ it can be readily checked that $0$ is a simple eigenvalue of $A_1$, hence $L_p(\Omega)=N(A_1)\oplus R(A_1)$. The kernel $N(A_1)$ of $A_1$ is given by $N(A_1)=\mathbb{K}\mathds{1}_\rho$, where
    $$\mathds{1}_\rho(x):=\chi_{\Omega_1}(x)+\frac{\rho_1}{\rho_2}\chi_{\Omega_2}(x),\quad x\in\Omega\backslash\Gamma.$$
and $R(A_1)=\{f\in L_p(\Omega):(f|\mathds{1}_\rho)=0\}$. Therefore \eqref{trans0} has a unique solution $q\in H_p^2(\Omega\backslash\Gamma)\ominus \mathbb{K}\mathds{1}_\rho$, provided $(f|\mathds{1}_\rho)=0$. In case of Dirichlet boundary conditions, i.e.\ $\delta=0$, it holds that $N(A_0)=\{0\}$, hence or each $f\in L_p(\Omega)$, the system \eqref{trans0} admits a unique solution $q\in H_p^2(\Omega\backslash\Gamma)$.
\begin{theorem}\label{thmtrans0.1}
Let $\Omega\subset\R^n$ a bounded domain, $1<p<\infty$, $f\in L_p(\Omega)$, $g=h_1=h_2=0$ and $\lambda=0$. Then the following assertions hold
    \begin{enumerate}
    \item If $\delta=0$, then there exists a unique solution $q\in H_p^2(\Omega\backslash\Gamma)$ of \eqref{trans0}.
    \item If $\delta=1$ and $(f|\mathds{1}_\rho)=0$, then there exists a unique solution $q\in H_p^2(\Omega\backslash\Gamma)\ominus \mathbb{K}\mathds{1}_\rho$.
    \end{enumerate}
If in addition $J=[0,a]$, $f=f(t,x)$ and $f\in H_p^1(J;L_p(\Omega))$ s.t.\ $f(t,\cdot)\in R(A_\delta)$ for a.e.\ $t\in J$, then $q\in H_p^1(J;H_p^2(\Omega\backslash\Gamma)\ominus N(A_\delta))$.
\end{theorem}

\subsection{A weak transmission problem}

Here we study the (weak) transmission problem
    \begin{align*}
    (\nabla q|\nabla\phi)_{L_2(\Omega)}&=(f|\nabla\phi)_{L_2(\Omega)},\quad
\phi\in H_{p'}^1(\Omega),\\
    [\![\rho q]\!]&=g,\quad x\in\Gamma,
    \end{align*}
where $\Omega\subset\R^n$ is open and bounded with $\partial\Omega\in C^2$. We want to show that this problem admits a unique solution $q\in
\dot{H}_p^1(\Omega\backslash\Gamma)$, that satisfies the estimate
    $$|\nabla q|_{L_p(\Omega)}\le
C\left(|f|_{L_p(\Omega;\rr^n)}+|g|_{W_p^{1-1/p}(\Gamma)}\right),$$
provided $f\in L_p(\Omega;\rr^n)$ and $g\in W_p^{1-1/p}(\Gamma)$. We will first treat the case $f=0$, $g\in W_p^{2-1/p}(\Gamma)$ and consider the problem
    \begin{align}
    \begin{split}\label{trans0.1}
    \lambda(q|\phi)+(\nabla q|\nabla\phi)_{L_2(\Omega)}&=0,\quad
\phi\in H_{p'}^1(\Omega),\\
    [\![\rho q]\!]&=g,\quad x\in\Gamma.
    \end{split}
    \end{align}
with $\lambda>0$. Theorem \ref{thmtrans0} then yields a strong unique solution $q\in H_p^2(\Omega\backslash\Gamma)$ of \eqref{trans0} with $f=h_1=h_2=0$ which is also the unique solution of \eqref{trans0.1}. This follows from integration by parts. Our aim is to derive an estimate which is of the form
    $$|q|_{H_p^1(\Omega\backslash\Gamma)}\le C|g|_{W_p^{1-1/p}(\Gamma)},$$
which will be done by a localization argument. For this purpose we consider first the following auxiliary
transmission problem
    \begin{align}
    \begin{split}\label{trans1}
    \lambda q-\Delta q&=f,\quad x\in \dot{\mathbb{R}}^{n},\\
    [\![\rho q]\!]&=g,\quad x\in\rr^{n-1},\\
    [\![\partial_\nu q]\!]&=h,\quad x\in\rr^{n-1},
    \end{split}
    \end{align}
with data $f\in L_p(\mathbb{R}^{n})$, $g\in
W_p^{2-1/p}(\rr^{n-1})$ and $h\in W_{p}^{1-1/p}(\rr^n)$, which will play an important role in the forthcoming localization procedure. Solve the full space problem
    $$\lambda q-\Delta q=f,\quad x\in \mathbb{R}^{n},$$
to obtain a unique solution $q_1=(\lambda-\Delta)^{-1}f\in
H_p^2(\rr^{n})$, provided ${\rm Re}\lambda>0$. In the sequel we
will always assume that $\lambda$ is real and $\lambda\ge 1$. In
particular, it follows that
    \begin{equation}\label{transq1}
    \lambda^{1/2}|q_1|_{L_p(\rr^{n})}+|\nabla q_1|_{L_p(\rr^{n})}\le C|f|_{H_p^{-1}(\rr^{n})},
    \end{equation}
with some constant $C>0$ being independent of $\lambda\ge 1$,
since
    \begin{align*}
    \lambda^{1/2}|(\lambda-\Delta)^{-1}f|_{L_p(\rr^{n})}&\le C\lambda^{1/2}||(I-\Delta)^{1/2}(\lambda-\Delta)^{-1}||_{\cB(L_p;L_p)}|f|_{H_p^{-1}(\rr^n)}\\
    &\le C||(I-\Delta)^{1/2}(\lambda-\Delta)^{-1/2}||_{\cB(L_p;L_p)}|f|_{H_p^{-1}(\rr^n)}\\
    &\le C|f|_{H_p^{-1}(\rr^n)},
    \end{align*}
and
    \begin{align*}
    |\nabla(\lambda-\Delta)^{-1}f|_{L_p(\rr^{n})}&\le C ||(I-\Delta)(\lambda-\Delta)^{-1}||_{\cB(L_p;L_p)}|f|_{H_p^{-1}(\rr^n)}\\
    &\le C|f|_{H_p^{-1}(\rr^n)},
    \end{align*}
since the norm
    $$||(I-\Delta)^{\alpha}(\lambda-\Delta)^{-\alpha}||_{\cB(L_p;L_p)},\quad \alpha\in\{1/2,1\},$$
is independent of $\lambda\ge 1$, which follows e.g.\ from
functional calculus. The shifted function $q_2=q-q_1$ should now
solve the reduced problem
    \begin{align}
    \begin{split}\label{trans2}
    \lambda q_2-\Delta q_2&=0,\quad x\in \dot{\mathbb{R}}^{n},\\
    [\![\rho q_2]\!]&=\tilde{g},\quad x\in\rr^{n-1},\\
    [\![\partial_\nu q_2]\!]&=h,\quad x\in\rr^{n-1},
    \end{split}
    \end{align}
with a modified function $\tilde{g}\in W_p^{2-1/p}(\rr^{n-1})$.
Let $x=(x',y)\in\rr^n\times\rr$ and define
$L:=(\lambda-\Delta_n)^{1/2}$, where $\Delta_n$ denotes the
Laplacian with respect to the first $n-1$ variables $x'$ and with
domain $D(L)=H_p^1(\rr^{n-1})$. Let furthermore
    $$\rho(x',y)=\rho_2\chi_{\{y>0\}}(x',y)+\rho_1\chi_{\{y<0\}}(x',y),\ (x',y)\in\rr^{n-1}\times\rr.$$
We make the following ansatz to
find a solution of \eqref{trans2}
    \begin{equation}\label{trans2.1}
    q_2(y):=
    \begin{cases}
    e^{-Ly}a_+,\ &y>0,\\
    e^{Ly}a_-,\ &y<0,
    \end{cases}
    \end{equation}
where $a_-,a_+$ have to be determined. The first transmission
condition in \eqref{trans2} yields
$\rho_2a_+-\rho_1a_-=\tilde{g}$, whereas the second condition
implies $-L(a_+ + a_-)=h$, hence $a_+ +a_-=-L^{-1}h$. Observe
that $\tilde{g},L^{-1}h\in W_p^{2-1/p}(\rr^{n-1})$. Therefore we
may solve this linear system of equations to the result
    \begin{equation}\label{trans2.2}
    a_-=-\frac{1}{\rho_1+\rho_2}\left(\tilde{g}+\rho_2
    L^{-1}h\right),\ a_+=\frac{1}{\rho_1+\rho_2}\left(\tilde{g}+\rho_2
    L^{-1}h\right)-L^{-1}h.
    \end{equation}
In other words, the solution of \eqref{trans2} (hence of
\eqref{trans1}) is uniquely determined and $a_-,a_+\in
W_p^{2-1/p}(\rr^{n-1})$. Since $|Le^{\pm
L\cdot}\cdot|_{L_p(\rr^{n-1}\times\rr_\mp)}$ is an equivalent
norm in $W_p^{1-1/p}(\rr^{n-1})$ and the corresponding constants
are independent of $\lambda\ge 1$, we obtain first
    $$\lambda^{1/2}|q_2|_{L_p(\rr^{n})}=\lambda^{1/2}|L^{-1}Lq_2|_{L_p(\rr^{n})}
    \le C\left(|a_+|_{W_p^{1-1/p}(\rr^{n-1})}+|a_-|_{W_p^{1-1/p}(\rr^{n-1})}\right).$$
Concerning $\nabla q_2$ in $L_p(\rr^{n})$, we estimate as follows
    \begin{align*}
    |\nabla_{x'} q_2|_{L_p(\rr^n)}&\le C|L_0 q_2|_{L_p(\rr^n)}=C|L_0L^{-1}L q_2|_{L_p(\rr^n)}\\
    &\le C||L_0L^{-1}||_{\cB(L_p,L_p)}|Lq_2|_{L_p(\rr^n)}\\
    &\le C\left(|a_+|_{W_p^{1-1/p}(\rr^{n-1})}+|a_-|_{W_p^{1-1/p}(\rr^{n-1})}\right),
    \end{align*}
with $L_0:=(I-\Delta_{x'})^{1/2}$. Here the norm
$||L_0L^{-1}||_{\cB(L_p,L_p)}$ does not depend on $\lambda\ge 1$,
which is a consequence of the functional calculus. The estimate
for $\partial_yq_2$ in $L_p(\rr^{n})$ is even simpler, since
    $$|\partial_y q_2|_{L_p(\rr^n)}=|Lq_2|_{L_p(\rr^n)}\le C\left(|a_+|_{W_p^{1-1/p}(\rr^{n-1})}+|a_-|_{W_p^{1-1/p}(\rr^{n-1})}\right).$$
This yields the estimate
    $$\lambda^{1/2}|q_2|_{L_p(\rr^{n})}+|\nabla q_2|_{L_p(\rr^n)}\le C\left(|\tilde{g}|_{W_p^{1-1/p}(\rr^{n-1})}+|L^{-1}h|_{W_p^{1-1/p}(\rr^{n-1})}\right).$$
For each fixed $\lambda\ge 1$ the operator $L^{-1}$ is bounded
and linear from $W_p^{-1/p}(\rr^{n-1})$ to
$W_p^{1-1/p}(\rr^{n-1})$, where $W_p^{-1/p}(\rr^{n-1})$ is the
topological dual space of $W_{p'}^{1/p}(\rr^{n-1})$, and
$1/p+1/p'=1$. We want to show that the bound of $L^{-1}$ is
independent of $\lambda\ge 1$. This can be seen as follows. We
have
    \begin{align*}
    |L^{-1}h|_{W_p^{1}(\rr^{n-1})}\le C|L_0L^{-1}h|_{L_p(\rr^{n-1})}\le C||L_0L^{-1}||_{\cB(L_p,L_p)}|h|_{L_p(\rr^{n-1})}
    \end{align*}
which holds for all $h\in L_p(\rr^{n-1})$, since
$|L_0\cdot|_{L_p(\rr^{n-1})}$ is an equivalent norm in
$W_p^1(\rr^{n-1})$. On the other hand we have
    \begin{align*}
    |L^{-1}h|_{L_p(\rr^{n-1})}&=|L_0L_0^{-1}L^{-1}h|_{L_p(\rr^{n-1})}=|L_0L^{-1}L_0^{-1}h|_{L_p(\rr^{n-1})}\\
    &\le ||L_0L^{-1}||_{\cB(L_p,L_p)}|L_0^{-1}h|_{L_p(\rr^{n-1})}\\
    &\le C||L_0L^{-1}||_{\cB(L_p,L_p)}|h|_{W_p^{-1}(\rr^{n-1})}
    \end{align*}
for all $h\in W_p^{-1}(\rr^{n-1})$, since
$|L_0^{-1}\cdot|_{L_p(\rr^{n-1})}$ is an equivalent norm in
$W_p^{-1}(\rr^{n-1})$ and since $L^{-1}$ and $L_0^{-1}$ are
commuting operators. Finally we apply the real interpolation
method to obtain
    \begin{align*}
    |L^{-1}h|_{W_p^{1-1/p}(\rr^{n-1})}\le C|h|_{W_p^{-1/p}(\rr^{n-1})},
    \end{align*}
for all $h\in W_p^{-1/p}(\rr^{n-1})$, where the constant $C>0$ is
independent of $\lambda\ge 1$. In summary we derived the a priori
estimate
    $$\lambda^{1/2}|q_2|_{L_p(\rr^n)}+|\nabla q_2|_{L_p(\mathbb{R}^{n})}\le
C\left(|\tilde{g}|_{W_p^{1-1/p}(\rr^{n-1})}+|h|_{W_p^{-1/p}(\rr^{n-1})}\right),$$
for the solution of \eqref{trans2}, hence
    \begin{multline}\label{trans2a}
    \lambda^{1/2}|q|_{L_p(\rr^n)}+|\nabla q|_{L_p(\mathbb{R}^{n})}\\
    \le C\left(|f|_{H_p^{-1}(\rr^{n})}+|g|_{W_p^{1-1/p}(\rr^{n-1})}+|h|_{W_p^{-1/p}(\rr^{n-1})}\right)
    \end{multline}
for the solution of \eqref{trans1}, since
    \begin{align*}
    |\tilde{g}|_{W_p^{1-1/p}(\rr^{n-1})}&\le |g|_{W_p^{1-1/p}(\rr^{n-1})}+|[\![\rho q_1]\!]|_{W_p^{1-1/p}(\rr^{n-1})}\\
    &\le |g|_{W_p^{1-1/p}(\rr^{n-1})}+C|f|_{H_p^{-1}(\rr^n)},
    \end{align*}
by \eqref{transq1}. Consider now a bounded domain
$\Omega\subset\rr^{n}$ with $\partial\Omega\in C^2$ and let
$\Gamma\subset\Omega$ be a hypersurface such that $\Gamma\in
C^2$, $\Gamma\cap\partial\Omega=\emptyset$ and such that $\Gamma$
divides the set $\Omega$ into two disjoint regions
$\Omega_1,\Omega_2$, where $\partial\Omega_1=\Gamma$ and
$\partial\Omega_2=\partial\Omega\cup\Gamma$. Since $\bar{\Omega}$
is compact, we may cover it by a union of finitely many open sets
$U_k,\ k=0,\ldots,N$ which are subject to the following conditions
    \begin{itemize}
        \item $\partial\Omega\subset U_0$ and $U_0\cap\Gamma=\emptyset$;
        \item $U_1\subset\Omega_1$ and $U_1\cap\Gamma=\emptyset$;
        \item $U_k\cap\Gamma\neq\emptyset$,
$U_k\cap\partial\Omega=\emptyset\ k=2,\ldots,N$ and
            $$\bigcup_{k=2}^N U_k\supset \Gamma.$$
    \end{itemize}
For $k\ge 2$, the sets $U_k$ may be balls with a fixed but
arbitrarily small radius $r>0$. Let $\{\varphi_k\}_{k=0}^N$ be a
partition of unity, such that ${\rm supp}\,\varphi_k\subset U_k$
and $0\le\varphi_k(x)\le 1$ for all $x\in\bar{\Omega}$. Consider
the transmission problem
    \begin{align}
    \begin{split}\label{trans3}
    \lambda q-\Delta q&=0,\quad x\in \Omega\backslash\Gamma\\
    [\![\rho q]\!]&=g,\quad x\in\Gamma,\\
    [\![\partial_{\nu_\Gamma} q]\!]&=0,\quad x\in\Gamma,\\
    \partial_\nu q&=0,\quad x\in\partial\Omega,
    \end{split}
    \end{align}
where $g\in W_p^{2-1/p}(\Gamma)$. Set $q_k=q\varphi_k$ and
$g_k=g\varphi_k$. By Theorem \ref{thmtrans0} there exists a unique
solution $q\in H_p^2(\Omega\backslash\Gamma)$ of \eqref{trans3},
if e.g.\ $\lambda\ge 1$. Multiplying
\eqref{trans3} by $\varphi_0$ yields
    \begin{align}
    \begin{split}\label{trans4}
    \lambda q_0-\Delta q_0&=-2(\nabla
q|\nabla\varphi_0)-q\Delta\varphi_0,\quad x\in \Omega,\\
    \partial_\nu q_0&=q\partial_\nu\varphi_0,\quad x\in\partial\Omega,
    \end{split}
    \end{align}
which is an elliptic boundary value problem in $\Omega$. Denote by
$(F_0,G_0)$ the right hand side of \eqref{trans4}. By
\cite[Theorem 3.3.4]{Tri83}, there exists a common bounded
extension operator $E$ from $L_p(\Omega)$ resp.\
$H_p^{-1}(\Omega)$ to $L_p(\rr^{n})$ resp.\ $H_p^{-1}(\rr^{n})$.
Solve the equation
    $$\lambda q_0^1-\Delta q_0^1=E F_0,\quad x\in\rr^{n}.$$
The solution is given by $q_0^1=(\lambda-\Delta)^{-1}EF_0$ and we
have the estimate
    $$\lambda^{1/2}|q_0^1|_{L_p(\rr^n)}+|\nabla q_0^1|_{L_p(\rr^{n})}\le C|EF_0|_{H_p^{-1}(\rr^{n})}\le
C|F_0|_{H_p^{-1}(\Omega)}\le C|q|_{L_p(\Omega)},$$ as we have already shown. Note that since $F_0\in L_p(\Omega)$, it holds that
    $$|q_0^1|_{H_p^2(\rr^n)}=|(\lambda-\Delta)^{-1}EF_0|_{H_p^2(\rr^n)}\le C|F_0|_{L_p(\Omega)}\le C|q|_{H_p^1(\Omega)},$$
and $C>0$ does not depend on $\lambda\ge 1$. In particular, the real interpolation method yields
    $$|q_0^1|_{W_p^{1+s}(\rr^n)}\le C|q|_{W_p^s(\Omega)},\ s\in[0,1].$$
The shifted function $q_0^2=q_0-q_0^1$ solves the problem
    \begin{align}
    \begin{split}\label{trans4a}
    \lambda q_0^2-\Delta q_0^2&=0,\quad x\in \Omega,\\
    \partial_\nu q_0^2&=G_0^2,\quad x\in\partial\Omega,
    \end{split}
    \end{align}
with some modified function $G_0^2\in
W_p^{1-1/p}(\partial\Omega)$. By \cite[Theorem 9.2]{Ama93}, there
exists a bounded solution operator
$S_0^2:W_p^{-1/p}(\partial\Omega)\to H_p^1(\Omega)$ such that
$q_0^2=S_0^2G_0^2$ and there exists a constant $C>0$ being
independent of $\lambda\ge 1$ such that
    $$\lambda^{1/2}|q_0^2|_{L_p(\Omega)}+|\nabla q_0^2|_{L_p(\Omega)}\le C|G_0^2|_{W_p^{-1/p}(\partial\Omega)}.$$
This yields
    \begin{multline*}
    \lambda^{1/2}|q_0|_{L_p(\Omega)}+|\nabla q_0|_{L_p(\Omega)}\le C(|(\nabla
q|\nabla\varphi_0)|_{H_p^{-1}(\Omega)}+|q\Delta\varphi_0|_{H_p^{-1}(\Omega)}\\
    +|q\partial_\nu\varphi_0|_{W_p^{-1/p}(\partial\Omega)}+|\partial_\nu
q_0^1|_{W_p^{-1/p}(\partial\Omega)}).
    \end{multline*}
Since $\varphi_0$ is smooth and compactly supported and since
$\nu\in C^1(\partial\Omega)$, we have
    \begin{equation}\label{trans5}
    \lambda^{1/2}|q_0|_{L_p(\Omega)}+|\nabla q_0|_{L_p(\Omega)}\le C|q|_{W_p^s(\Omega)},
    \end{equation}
for some $s\in (1/p,1)$, since
    $$|q|_{W_p^{-1/p}(\partial\Omega)}\le C|q|_{L_p(\partial\Omega)}\le
C|q|_{W_p^s(\Omega)},\quad s\in (1/p,1),$$
and
    $$|\partial_\nu q_0^1|_{W_p^{-1/p}(\partial\Omega)}\le C|\partial_\nu q_0^1|_{L_p(\partial\Omega)}\le C|q_0^1|_{W_p^{1+s}(\Omega)}\le C|q|_{W_p^s(\Omega)}.$$
In a next step we multiply \eqref{trans3} by $\varphi_1$ to obtain the full space
problem
    \begin{equation}\label{trans6}
    \lambda q_1-\Delta q_1=-2(\nabla
q|\nabla\varphi_1)-q\Delta\varphi_1,\quad x\in \rr^{n}.
    \end{equation}
This problem admits a unique solution
$q_1=(\lambda-\Delta)^{-1}F_1$, provided $\lambda\ge 1$, where
$S_1=(\lambda-\Delta)^{-1}:H_p^{-1}(\rr^{n})\to H_p^1(\rr^{n})$ is
bounded and $F_1$ denotes the right hand side of \eqref{trans6}.
As before we obtain the estimate
    \begin{equation}\label{trans7}
    \lambda^{1/2}|q_1|_{L_p(\rr^n)}+|\nabla q_1|_{L_p(\rr^n)}\le C|q|_{L_p(\Omega)},
    \end{equation}
with $C>0$ being independent of $\lambda\ge 1$.

We turn now to the charts $U_k$, $k=2,\ldots,N$. Multiplying
\eqref{trans3} by $\varphi_k$, $k=2,\ldots,N$, we obtain the pure
transmission problem
    \begin{align}
    \begin{split}\label{trans8}
    \lambda q_k-\Delta q_k&=-2(\nabla
q|\nabla\varphi_k)-q\Delta\varphi_k,\quad x\in \rr^n\backslash\Gamma,\\
    [\![\rho q_k]\!]&=g_k,\quad x\in\Gamma,\\
    [\![\partial_\nu q_k]\!]&=[\![q]\!]\partial_\nu\varphi_k,\quad
x\in\Gamma.
    \end{split}
    \end{align}
Let $x_0\in\Gamma$. Then there exists $k\in \{2,\ldots,N\}$ such
that $x_0\in U_{k}$. After a translation and a rotation of
coordinates we may assume that $x_0=0$ and that the normal
$\nu(x_0)$ at $x_0$ which points from $\Omega_1$ to $\Omega_2$ is
given by $\nu(x_0)=[0,\ldots,0,-1]^{\sf T}$. Consider a graph
$\eta\in C^2(\rr^{n-1})$ with compact support such that
    $$\{(x',x_{n})\in
U_{k}\subset\rr^{n-1}\times\rr:x_{n}=\eta(x')\}=\Gamma\cap
U_{k}.$$ Note that, since $\nabla_{x'}\eta(0)=0$, we may choose
$|\nabla_{x'}\eta|_\infty$ as small as we wish, by decreasing the
size of the chart $U_{k}$. Let
$q(x',x_{n})=v(x',x_{n}-\eta(x'))$, where $(x',x_{n})\in U_{k}$.
We define a new coordinate by $y=x_{n}-\eta(x')$, $(x',x_{n})\in
U_{k}$. Then we obtain
    \begin{multline*}
    \Delta q(x',x_{n})=\Delta_y
v(x',y)-2\partial_y\left(\nabla_{x'}v(x',y)|\nabla_{x'}\eta(x')\right)\\
    +\partial_y^2v(x',y)|\nabla_{x'}\eta|^2-\partial_y
v(x',y)\Delta_{x'}\eta(x')
    \end{multline*}
and
    $$[\![\partial_\nu q]\!]=-\sqrt{1+|\nabla_{x'}\eta|^2}[\![\partial_y
v]\!]+\frac{1}{\sqrt{1+|\nabla_{x'}\eta|^2}}\left([\![\nabla_{x'}v]\!]|\nabla_{x'}\eta\right),$$
since the normal at $x\in U_{k}\cap \Gamma$ is given by

$$\nu(x',\eta(x'))=\frac{1}{\sqrt{1+|\nabla_{x'}\eta|^2}}[(\nabla_{x'}\eta)^{\sf
T},-1]^{\sf T}.$$ Let $(\Theta
u)(x',y):=q(x',y+\eta(x'))=v(x',y)$ with inverse $(\Theta^{-1}
v)(x',x_{n+1})=v(x',x_{n+1}-\eta(x'))=q(x',x_{n+1})$. Applying
the $C^2$-diffeomorphism $\Theta$ to \eqref{trans8} and
considering the terms on the right hand side of \eqref{trans8}
which depend on $u$ as given functions $(f_k,g_k,h_k)$ yields the
problem
    \begin{align}
    \begin{split}\label{trans9}
    \lambda v_k-\Delta_y v_k&=F(f_k,v_k,\varphi_k,\eta),\quad
(x',y)\in\dot{\mathbb{R}}^{n},\\
    [\![\rho v_k]\!]&=G(g_k),\quad x'\in\mathbb{R}^{n-1},\ y=0,\\
    [\![\partial_y v_k]\!]&=H(v_k,\varphi_k,\eta),\quad
x'\in\mathbb{R}^{n-1},\ y=0.
    \end{split}
    \end{align}
which is of the form \eqref{trans1}. Here
    \begin{equation*}
    F(f_k,v_k,\varphi_k,\eta):=-2\partial_y\left(\nabla_{x'}v_k|\nabla_{x'}\eta\right)
    +\partial_y^2v_k|\nabla_{x'}\eta|^2-\partial_y v_k\Delta_{x'}\eta,
    \end{equation*}
$G(g_k):=\Theta g_k$ and
    \begin{equation*}
    H(h_k,v_k,\varphi_k,\eta):=\frac{1}{1+|\nabla_{x'}\eta|^2}
    \left([\![\nabla_{x'}v_k]\!]|\nabla_{x'}\eta\right)
    \end{equation*}
We want to apply \eqref{trans2a} to \eqref{trans9} and estimate
as follows.
    \begin{align*}
    |\partial_y\left(\nabla_{x'}v_k|\nabla_{x'}\eta\right)|_{W_p^{-1}(\rr^{n})}&\le
    C|(I-\Delta_y)^{-1/2}\partial_y\left(\nabla_{x'}v_k|\nabla_{x'}\eta\right)|_{L_p(\rr^{n})}\\
    &\le C|\left(\nabla_{x'}v_k|\nabla_{x'}\eta\right)|_{L_p(\rr^{n})}\\
    &\le C|\nabla_{x'}\eta|_\infty|v_k|_{W_p^1(\dot{\rr}^{n})}.
    \end{align*}
In the same way we obtain
    $$|\partial_y^2v_k|\nabla_{x'}\eta|^2|_{W_p^{-1}(\rr^{n})}\le
C|\nabla_{x'}\eta|_\infty^2|v_k|_{W_p^1(\dot{\rr}^{n})},$$ whereas
    $$|\partial_y v_k\Delta_{x'}\eta|_{W_p^{-1}(\rr^{n})}\le C
|v_k|_{L_p(\rr^{n})},$$ since $\eta$ is smooth. Concerning the
terms in the Neumann transmission condition, we obtain by trace
theory
    \begin{align*}
    |\left([\![\nabla_{x'}v_k]\!]|\nabla_{x'}\eta\right)|_{W_p^{-1/p}(\rr^{n-1})}&\le
    C|\nabla_{x'}\eta|_{C^\alpha(\rr^{n-1})}|\nabla_{x'}v_k|_{W_p^{-1/p}(\rr^{n-1})}\\
    &\le C|\nabla_{x'}\eta|_{C^\alpha(\rr^{n-1})}|v_k|_{W_p^{1-1/p}(\rr^{n-1})}\\
    &\le C|\nabla_{x'}\eta|_{C^\alpha(\rr^{n-1})}|v_k|_{W_p^{1}(\dot{\rr}^{n})},
    \end{align*}
where $\alpha\in (1/p,1)$. These estimates show that the right
hand side of \eqref{trans9} may be estimated by terms that are
either of lower order or of highest order, but the higher order
terms carry a factor of the form
$|\nabla_{x'}\eta|_\infty^\theta$, $\theta>0$, which becomes
small, by decreasing the size of the chart $U_k$. Applying
perturbation theory it follows that there exists $\lambda_0\ge 1$
such that for each chart $U_k$, $k=2,\ldots,N$, the linear
problem \eqref{trans9} has a bounded solution operator
    $$S_k:W_p^{-1}(\rr^{n})\times W_p^{1-1/p}(\rr^{n-1})\times
W_p^{-1/p}(\rr^{n-1})\to W_p^1(\dot{\rr}^{n}),$$ provided
$\lambda\ge \lambda_0$. This in turn yields that
$\Theta^{-1}S_k\Theta$ is the corresponding solution operator for
problem \eqref{trans8}, i.e. we have
    $$q_k=(\Theta^{-1}S_k\Theta)(F_k,G_k,H_k),$$
for each $k=2,\ldots,N$, where $(F_k,G_k,H_k)$ denotes the right
hand side of \eqref{trans8}. Since $\Theta$ is a
$C^2$-diffeomorphism, we obtain the estimate
    \begin{equation}\label{trans10}
    \lambda^{1/2}|q_k|_{L_p(\Omega)}+|\nabla q_k|_{L_p(\Omega)}\le
C\left(|g|_{W_p^{1-1/p}(\Gamma)}+|q|_{W_p^s(\Omega\backslash\Gamma)}\right),
    \end{equation}
for some $s\in (1/p,1)$ and for each $k=2,\ldots,N$. Here the
constant $C>0$ does not depend on $\lambda\ge \lambda_0$, as we
have already shown in the investigation of \eqref{trans1}. Let us
introduce
$$|v|_{\lambda,W_p^1(\Omega)}:=|\lambda|^{1/2}|v|_{L_p(\Omega)}+|\nabla
v|_{L_p(\Omega)},\quad \lambda\ge 1,\ v\in
W_p^1(\Omega\backslash\Gamma),$$ which is an equivalent norm in
$W_p^1(\Omega\backslash\Gamma)$. This yields
    \begin{align*}
    |q|_{\lambda,W_p^1(\Omega)}\le\sum_{k=0}^N
|q_k|_{\lambda,W_p^1(\Omega)}\le
C\Big(|g|_{W_p^{1-1/p}(\Gamma)}+|q|_{W_p^s(\Omega)}\Big).
    \end{align*}
with constants $C,M>0$, being independent of $\lambda$. Since
$s\in (1/p,1)$ we may apply interpolation theory to the result
    \begin{align*}
    |q|_{W_p^s(\Omega)}&\le \varepsilon
|q|_{W_p^1(\Omega)}+C(\varepsilon)|q|_{L_p(\Omega)}\\
&\le\varepsilon
|q|_{\lambda,W_p^1(\Omega)}+C(\varepsilon)|q|_{L_p(\Omega)}\\
&\le\left(\varepsilon+C(\varepsilon)/|\lambda|^{1/2}\right)
|q|_{\lambda,W_p^1(\Omega)},
    \end{align*}
since by assumption $\lambda\ge 1$. Choosing first
$\varepsilon>0$ small enough and then $\lambda\ge 1$ sufficiently
large, we finally obtain the estimate
    \begin{equation}\label{trans11}
    |q|_{W_p^1(\Omega)}\le C|g|_{W_p^{1-1/p}(\Gamma)}
    \end{equation}
for the strong solution $q\in W_p^2(\Omega\backslash\Gamma)$ of
\eqref{trans3}. Now we want to reduce the regularity of $g$. Fix
$g\in W_p^{1-1/p}(\Gamma)$. Then there exists a sequence
$(g_m)\subset W_p^{2-1/p}(\Gamma)$, such that $g_m\to g$ as $m\to
\infty$ in $W_p^{1-1/p}(\Gamma)$. We denote by $q_m\in
W_p^2(\Omega\backslash\Gamma)$ the corresponding solutions of
\eqref{trans3}. Then it follows from \eqref{trans11} that $(q_m)$
is a Cauchy sequence in $W_p^{1}(\Omega\backslash\Gamma)$.
Therefore the limit $\lim_{m\to\infty} q_m=:q_\infty$ exists and
$q_\infty\in W_p^{1}(\Omega\backslash\Gamma)$ is the unique weak
solution of \eqref{trans3} for sufficiently large $\lambda\ge 1$.
\begin{lemma}\label{translem1}
Let $1<p<\infty$, $1/p+1/p'=1$ and let $g\in W_p^{1-1/p}(\Gamma)$
be given. Then there exists $\lambda_0\ge 1$ such that the problem
    \begin{align*}
    \lambda (q|\phi)_{L_2(\Omega)}+(\nabla
q|\nabla\phi)_{L_2(\Omega)}&=0,\quad \phi\in
H_{p'}^1(\Omega),\\
    [\![\rho q]\!]&=g,\quad x\in\Gamma,
    \end{align*}
has a unique solution $q\in H_p^1(\Omega\backslash\Gamma)$,
provided $\lambda\ge \lambda_0$. Moreover, the solution $q\in
H_p^1(\Omega\backslash\Gamma)$ satisfies the estimate
    \begin{equation}\label{estweaksol}
    |q|_{H_p^{1}(\Omega)}\le C|g|_{W_p^{1-1/p}(\Gamma)}.
    \end{equation}
\end{lemma}
In a next step we consider the problem
    \begin{align}
    \begin{split}\label{trans12}
    \lambda (q|\phi)_{L_2(\Omega)}+(\nabla
q|\nabla\phi)_{L_2(\Omega)}
    &=(f|\nabla\phi)_{L_2(\Omega)},\quad \phi\in
H_{p'}^1(\Omega),\\
    [\![\rho q]\!]&=0,\quad x\in\Gamma,
    \end{split}
    \end{align}
where $f\in L_p(\Omega;\rr^n)$ is given. Observe that the mapping
$\psi_f:H_{p'}^1(\Omega\backslash\Gamma)\to \rr$ defined by
    $$\psi_f(\phi):=\langle
\psi_f,\phi\rangle:=\int_{\Omega}(f|\nabla\phi) dx,$$ is linear
and continuous, since
    $$|\psi_f(\phi)|\le
|f|_{L_p(\Omega;\rr^n)}|\phi|_{H_{p'}^1(\Omega)},$$ hence
$\psi_f\in \left(H_{p'}^1(\Omega\backslash\Gamma)\right)^*$. With
the help of the Dirichlet form
    $$a:H_p^1(\Omega\backslash\Gamma)\times
H_{p'}^1(\Omega\backslash\Gamma)\to\rr,\quad
a(q,v):=\int_{\Omega}\nabla q\cdot\nabla v dx,$$ we define an
operator $A:H_p^1(\Omega\backslash\Gamma)\to
\left(H_{p'}^1(\Omega\backslash\Gamma)\right)^*$ by means of
    $$\langle Aq,v\rangle:=a(q,v),$$
with domain $D(A)=\{q\in H_p^1(\Omega\backslash\Gamma):[\![\rho
q]\!]=0\ \mbox{on}\ \Gamma\}$. Making use of these definitions,
we may rewrite \eqref{trans12} in the abstract form
    \begin{equation}\label{trans13}
    \lambda q+Aq=\psi_f,\quad \mbox{in}\
\left(H_{p'}^1(\Omega\backslash\Gamma)\right)^*.
    \end{equation}
Since
    $$H_p^1(\Omega\backslash\Gamma)\compemb
\left(H_{p'}^1(\Omega\backslash\Gamma)\right)^*,$$ the resolvent
of $A$ is compact and therefore the spectral set $\sigma(A)$
consists solely of a countably infinite sequence of isolated
eigenvalues. By a bootstrap argument it is easily seen that the
corresponding eigenfunctions are smooth. Hence, defining $A_2$ to
be the part of $A$ in $L_2(\Omega\backslash\Gamma)$ with domain
$D(A_2)=\{q\in D(A):Au\in L_2(\Omega\backslash\Gamma)\}$, it
follows that $\sigma(A)=\sigma(A_2)$. Integrating by parts, we
obtain
    $$D(A_2)=\{q\in H_2^2(\Omega\backslash\Gamma):[\![\rho q]\!]=0,\
[\![\partial_{\nu_\Gamma} q]\!]=0\ \mbox{on}\ \Gamma,\
\partial_\nu q=0\ \mbox{on}\ \partial\Omega\}$$ and $A_2
q=-\Delta q$ in $\Omega\backslash\Gamma$. Let
$\lambda\in\sigma(-A)=\sigma(-A_2)$ and let $q\in D(A_2)$ be a
corresponding eigenfunction. Then $q$ satisfies the problem
    \begin{align}
    \begin{split}\label{trans14}
    \lambda q-\Delta q&=0,\quad x\in\Omega\backslash\Gamma,\\
    [\![\rho q]\!]&=0,\quad x\in\Gamma,\\
    [\![\partial_{\nu_\Gamma} q]\!]&=0,\quad x\in\Gamma,\\
    \partial_\nu q&=0,\quad x\in\partial\Omega.
    \end{split}
    \end{align}
Multiplying $\eqref{trans14}_1$ by $\rho q$ and integrating by
parts, we obtain by $\eqref{trans14}_{2,3,4}$
    \begin{align*}
    -\lambda\int_{\Omega\backslash\Gamma}\rho |q|^2dx&=-\int_{\Omega\backslash\Gamma}\rho q \Delta q
dx=-\rho_1\int_{\Omega_1} q_1 \Delta q_1
dx-\rho_2\int_{\Omega_2} q_2
\Delta q_2 dx\\
    &=\rho_1|\nabla q_1|_2^2+\rho_2|\nabla q_2|_2^2+\int_\Gamma
\left(\partial_{\nu_\Gamma} q_2 \rho_2q_2-\partial_{\nu_\Gamma}
q_1
\rho_1q_1\right) d\Gamma\\
    &=\rho_1|\nabla q_1|_2^2+\rho_2|\nabla q_2|_2^2+\int_\Gamma
\partial_{\nu_\Gamma} q_2 [\![\rho q]\!] d\Gamma\\
    &=\rho_1|\nabla q_1|_2^2+\rho_2|\nabla q_2|_2^2\ge 0,
    \end{align*}
where $q_j$ denotes the part of $q$ in $\Omega_j$. In particular
it follows that $\lambda$ is real and $\lambda\le 0$ for all $\lambda\in\sigma
(-A)$ and if $\lambda =0$ then $q_1$ and $q_2$ are both equal to
a constant in $\Omega_1$ and $\Omega_2$, respectively, satisfying
the identity $\rho_1 q_1=\rho_2 q_2$. In  other words the
eigenvalue $\lambda=0$ is simple and the kernel $N(A)=N(A_2)$ is
given by
    $$N(A)=\K\mathds{1}_\rho,\quad
\mathds{1}_\rho(x):=\chi_{\Omega_1}(x)+\frac{\rho_1}{\rho_2}\chi_{\Omega_2}(x),\
x\in\Omega.$$ Therefore, spectral theory implies
$\left(H_{p'}^1(\Omega\backslash\Gamma)\right)^*=N(A)\oplus R(A)$
and $H_p^1(\Omega\backslash\Gamma)=N(A)\oplus Y$, where $Y$ is a
closed subspace of $H_p^1(\Omega\backslash\Gamma)$. Note that
these decompositions reduce the linear operator $A$. It follows
that the equation $Aq=F$ has a unique solution $q\in Y\subset
H_p^1(\Omega\backslash\Gamma)$ if and only if $F\in R(A)$, or
equivalently $\langle F,\mathds{1}_\rho\rangle=0$. If $c\in\K$,
then \emph{any} other solution $\tilde{q}\in
H_p^1(\Omega\backslash\Gamma)$ of $A\tilde{q}=F$ is given by
$\tilde{q}=q+c\mathds{1}_\rho$ and we have the estimate
    $$|\nabla \tilde{q}|_{L_p(\Omega)}\le
C|F|_{\left(H_{p'}^1(\Omega\backslash\Gamma)\right)^*}\ .$$
\begin{lemma}\label{translem2}
Let $1<p<\infty$, $1/p+1/p'=1$ and let $f\in L_p(\Omega;\rr^n)$
be given. Then the problem
    \begin{align*}
    (\nabla q|\nabla\phi)_{L_2(\Omega)}&=(f|\nabla\phi)_{L_2(\Omega)},\quad
\phi\in H_{p'}^1(\Omega),\\
    [\![\rho q]\!]&=0,\quad x\in\Gamma,
    \end{align*}
has a unique solution $q\in \dot{H}_p^1(\Omega\backslash\Gamma)$,
satisfying the estimate
    $$|\nabla q|_{L_p(\Omega)}\le
C|f|_{L_p(\Omega;\rr^n)}\ .$$
\end{lemma}
For the final step, let $v\in H_p^1(\Omega\backslash\Gamma)$ be
the unique solution of
    \begin{align*}
    \lambda_0 (v|\phi)_{L_2(\Omega)}+(\nabla
v|\nabla\phi)_{L_2(\Omega)}&=0,\quad \phi\in
H_{p'}^1(\Omega),\\
    [\![\rho v]\!]&=g,\quad x\in\Gamma,
    \end{align*}
which is well-defined, thanks to Lemma \ref{translem1}. With the
help of this solution $v$, we define a functional $\psi_v\in
\left(H_{p'}^1(\Omega\backslash\Gamma)\right)^*$ by
    $$\psi_v(\phi):=\int_{\Omega}\nabla v\cdot\nabla\phi
dx.$$ By definition it holds that $\psi_v(\mathds{1}_\rho)=0$.
Since also $\psi_f(\mathds{1}_\rho)=0$ for all $f\in
L_p(\Omega;\rr^n)$, Lemma \ref{translem2} yields a unique
solution $w\in \dot{H}_p^1(\Omega\backslash\Gamma)$ of
    \begin{align*}
    (\nabla w|\nabla\phi)_{L_2(\Omega)}&=\psi_f(\phi)-\psi_v(\phi),\quad
\phi\in H_{p'}^1(\Omega),\\
    [\![\rho w]\!]&=0,\quad x\in\Gamma.
    \end{align*}
Finally, the sum $q:=v+w\in\dot{H}_p^1(\Omega\backslash\Gamma)$
is the unique solution of
    \begin{align*}
    (\nabla q|\nabla\phi)_{L_2(\Omega)}&=\psi_f(\phi)=(f|\nabla\phi)_{L_2(\Omega)},\quad
\phi\in H_{p'}^1(\Omega),\\
    [\![\rho q]\!]&=g,\quad x\in\Gamma
    \end{align*}
and we have the estimate
    $$|\nabla q|_{L_p(\Omega)}\le
C\left(|f|_{L_p(\Omega;\rr^n)}+|g|_{W_p^{1-1/p}(\Gamma)}\right).$$
\begin{theorem}\label{thmtrans}
Let $1<p<\infty$, $1/p+1/p'=1$, $f\in L_p(\Omega;\rr^n)$ and
$g\in W_p^{1-1/p}(\Gamma)$ be given. Then the problem
    \begin{align*}
    (\nabla q|\nabla\phi)_{L_2(\Omega)}&=(f|\nabla\phi)_{L_2(\Omega)},\quad
\phi\in H_{p'}^1(\Omega),\\
    [\![\rho q]\!]&=g,\quad x\in\Gamma
    \end{align*}
has a unique solution $u\in \dot{H}_p^1(\Omega\backslash\Gamma)$
satisfying the estimate
    $$|\nabla q|_{L_p(\Omega)}\le
C_1\left(|f|_{L_p(\Omega;\rr^n)}+|g|_{W_p^{1-1/p}(\Gamma)}\right).$$
If $J=[0,a]$, $f=f(t,x)$, $f\in H_p^1(J;L_p(\Omega;\R^n))$, $g=0$, then $q\in H_p^1(J;\dot{H}_p^1(\Omega\backslash\Gamma))$ and
    $$||\nabla q||_{H_p^1(J;L_p(\Omega))}\le
C_2||f||_{H_p^1(J;L_p(\Omega;\rr^n))}.$$
\end{theorem}

\subsection{Higher regularity in the bulk phases}

The next problem we consider, is about higher
regularity in the bulk phases $\Omega\backslash\Gamma$. To be precise, we study
the elliptic transmission problem
    \begin{align}
    \begin{split}\label{trans15}
    \lambda q-\Delta  q&=f,\quad x\in\Omega\backslash\Gamma,\\
    [\![\rho q]\!]&=0,\quad x\in\Gamma,\\
    [\![\partial_{\nu_\Gamma}  q]\!]&=0,\quad x\in\Gamma,\\
    \delta\partial_{\nu_{\Omega}} q+(1-\delta) q&=0,\quad x\in\partial\Omega,\ \delta\in\{0,1\},
    \end{split}
    \end{align}
where $f\in L_p(\Omega)\cap W_p^s(\Omega\backslash\Gamma)$,
$s>0$, is given and $\lambda\ge 1$. It is our aim to find a
unique solution $ q\in W_p^{2+s}(\Omega\backslash\Gamma)$ of \eqref{trans15}. Note that
by Theorem \ref{thmtrans0} there exists a unique solution $ q\in
H_p^2(\Omega\backslash\Gamma)$ of \eqref{trans15}. Moreover, there exists a
constant $C>0$ being independent of $\lambda\ge 1$ such that the
estimate
    \begin{equation}\label{trans16}
    | q|_{H_p^2(\Omega\backslash\Gamma)}\le C|f|_{L_p(\Omega)}
    \end{equation}
is valid. Thus, it remains to show that in addition $ q\in W_p^{2+s}(\Omega\backslash\Gamma)$, provided
$f\in L_p(\Omega)\cap W_p^s(\Omega\backslash\Gamma)$. For this purpose let $\partial\Omega\in C^3$ and cover
the compact set $\bar{\Omega}$ by a union of finitely many open
sets $U_k,\ k=0,\ldots,N$ which are subject to the following
conditions
    \begin{itemize}
        \item $\partial\Omega\subset U_0$ and $U_0\cap\Gamma=\emptyset$;
        \item $U_1\subset\Omega_1$ and $U_1\cap\Gamma=\emptyset$;
        \item $U_k\cap\Gamma\neq\emptyset$,
$U_k\cap\partial\Omega=\emptyset\ k=2,\ldots,N$ and
            $$\bigcup_{k=2}^N U_k\supset \Gamma.$$
    \end{itemize}
For $k\ge 2$, the sets $U_k$ may be balls with a fixed but
arbitrarily small radius $r>0$. As before, let
$\{\varphi_k\}_{k=0}^N$ be a partition of unity, such that ${\rm
supp}\,\varphi_k\subset U_k$ and $0\le\varphi_k(x)\le 1$ for all
$x\in\bar{\Omega}$. Let $ q_k:= q\varphi_k$ and
$f_k:=f\varphi_{k}$.

Multiplying \eqref{trans15} by $\varphi_0$ yields the problem
    \begin{align}
    \begin{split}\label{trans17}
    \lambda q_0-\Delta  q_0&=f_0-2(\nabla q|\nabla\varphi_0)- q\Delta\varphi_0,\quad x\in \Omega,\\
    \delta\partial_{\nu_{\Omega}} q_0+(1-\delta) q_0&=\delta q\partial_{\nu_{\Omega}}\varphi_0,\quad x\in\partial\Omega,\ \delta\in\{0,1\}.
    \end{split}
    \end{align}
Since $\varphi_0$ is smooth and $ q\in H_p^2(\Omega)$, the
right hand side $(F_0,G_0)$ in \eqref{trans17} is in
$W_p^s(\Omega)\times W_p^{1+s-1/p}(\partial\Omega)$, at least for $s\in(0,1]$. It follows from
\cite[Theorem 5.5.1 \& Remark 5.5.2/2]{Tri95} that $ q_0\in
W_p^{2+s}(\Omega)$, $s\in [0,1]$ and
    $$| q_0|_{W_p^{2+s}(\Omega)}\le C(|F_0|_{W_p^s(\Omega)}+|G_0|_{W_p^{1+s-1/p}(\partial\Omega)})\le C\left(|f|_{W_p^{s}(\Omega\backslash\Gamma)}+|f|_{L_p(\Omega)}\right),$$
by \eqref{trans16}, where the constant $C>0$ does not depend on
$\lambda\ge 1$. Multiplying \eqref{trans15} by $\varphi_1$ we
obtain the full space problem
    \begin{equation}\label{trans18}
    \lambda q_1-\Delta  q_1=f_1-2(\nabla q|\nabla\varphi_1)- q\Delta\varphi_1,\quad x\in \rr^n,
    \end{equation}
with a right hand side in $W_p^s(\rr^n)$, $s\in(0,1]$, which we denote by
$F_1$. Then the solution of \eqref{trans18} is given by
    $$ q_1=(\lambda-\Delta)^{-1}F_1.$$
If $\alpha\in\{0,1\}$ and $F_1\in H_p^\alpha(\rr^n)$ then
$ q_1\in H_p^{2+\alpha}(\rr^n)$ and
    \begin{align*}
    | q_1|_{H_p^{2+\alpha}(\rr^n)}&\le C|(I-\Delta)^{1+\alpha/2} q_1|_{L_p(\rr^n)}=C|(I-\Delta)^{1+\alpha/2}(\lambda-\Delta)^{-1}F_1|_{L_p(\rr^n)}\\
    &=C|(I-\Delta)(\lambda-\Delta)^{-1}(I-\Delta)^{\alpha/2}F_1|_{L_p(\rr^n)}\\
    &\le C||(I-\Delta)(\lambda-\Delta)^{-1}||_{\cB(L_p,L_p)}|(I-\Delta)^{\alpha/2}F_1|_{L_p(\rr^n)}\\
    &\le C||(I-\Delta)(\lambda-\Delta)^{-1}||_{\cB(L_p,L_p)}|F_1|_{H_p^\alpha(\rr^n)},
    \end{align*}
since $|(I-\Delta)^{1+\alpha/2}\cdot|_{L_p(\rr^n)}$ is an
equivalent norm in $H_p^{2+\alpha}(\rr^n)$, $\alpha\in \{0,1\}$.
Note that the term
    $$||(I-\Delta)(\lambda-\Delta)^{-1}||_{\cB(L_p,L_p)}$$
is independent of $\lambda\ge 1$, which follows e.g.\ from
functional calculus. The real interpolation method and
\eqref{trans16} then yield the estimate
    $$| q_1|_{W_p^{2+s}(\rr^n)}\le C|F_1|_{W_p^s(\rr^n)}\le C\left(|f|_{W_p^{s}(\Omega\backslash\Gamma)}+|f|_{L_p(\Omega)}\right),$$
for $s\in(0,1]$, where $C>0$ does not depend on $\lambda\ge 1$.
Next, we multiply \eqref{trans15} by $\varphi_k$,
$k\in\{2,\ldots,N\}$, to obtain the pure transmission problems
    \begin{align}
    \begin{split}\label{trans21}
    \lambda q_k-\Delta  q_k&=f_k-2(\nabla q|\nabla\varphi_k)- q\Delta\varphi_k,\quad x\in \rr^n\backslash\Gamma,\\
    [\![\rho q_k]\!]&=0,\quad x\in\Gamma,\\
    [\![\partial_\nu  q_k]\!]&=[\![ q]\!]\partial_\nu\varphi_k,\quad x\in\Gamma,
    \end{split}
    \end{align}
with some function $f_k\in W_p^s(\rr^n\backslash\Gamma)\cap
L_p(\rr^n)$. For each fixed $k\in\{2,\ldots,N\}$ we may use the
transformation described above, to reduce \eqref{trans21} to the
problem
    \begin{align}
    \begin{split}\label{trans22}
    \lambda\psi-\Delta \psi&=F,\quad x'\in\rr^{n-1},\ y\in\dot{\rr},\\
    [\![\rho\psi]\!]&=0,\quad x'\in\rr^{n-1},\ y=0,\\
    [\![\partial_y \psi]\!]&=G,\quad x'\in\rr^{n-1},\ y=0,
    \end{split}
    \end{align}
with given functions $F\in W_p^{s}(\dot{\rr}^n)$ and $G\in
W_p^{1+s-1/p}(\rr^{n-1})$, $s\in(0,1]$. First we remove the
inhomogeneity $F$. To this end we solve the Dirichlet problems
    $$\lambda\psi-\Delta_{x'}\psi^+-\partial_y^2\psi^+=F^+,\ x'\in\rr^{n-1},\ y>0,\quad \psi^+(x',0)=0,$$
and
    $$\lambda\psi-\Delta_{x'}\psi^--\partial_y^2\psi^-=F^-,\ x'\in\rr^{n-1},\ y<0,\quad \psi^-(x',0)=0,$$
where $F^+:=F|_{y>0}$ and $F^-:=F|_{y<0}$. Let $\psi^{\pm}\in
W_p^{2+s}(\dot{\rr}^n)$ be defined as
    $$\psi^{\pm}(x',y):=
    \begin{cases}
    \psi^+(x',y),&\ y>0,\\
    \psi^-(x',y),&\ y<0.
    \end{cases}
    $$
Since $\psi^+(x',0)=\psi^-(x',0)=0$ and $[\![\rho\tilde{\psi}]\!]=[\![\rho]\!]\psi^{\pm}=0$, the shifted function
$\tilde{\psi}:=\psi-\psi^{\pm}$ solves the problem
    \begin{align}
    \begin{split}\label{trans23}
    \lambda\tilde{\psi}-\Delta \tilde{\psi}&=0,\quad x'\in\rr^{n-1},\ y\in\dot{\rr},\\
    [\![\rho\tilde{\psi}]\!]&=0,\quad x'\in\rr^{n-1},\ y=0,\\
    [\![\partial_y \tilde{\psi}]\!]&=\tilde{G},\quad x'\in\rr^{n-1},\ y=0,
    \end{split}
    \end{align}
where $\tilde{G}:=G-[\![\partial_y \psi^{\pm}]\!]\in
W_p^{1+s-1/p}(\rr^{n-1})$. According to \eqref{trans2.1} and
\eqref{trans2.2} the unique solution of \eqref{trans23} is given
by
    \begin{equation}\label{trans24}
    \tilde{\psi}(y)=-\frac{1}{\rho_1+\rho_2}L^{-1}
    \begin{cases}
    \rho_1e^{-Ly}\tilde{G},&\ y>0,\\
    \rho_2e^{Ly}\tilde{G},&\ y<0,
    \end{cases}
    \end{equation}
where $L:=(\lambda-\Delta_{x'})^{1/2}$ with domain
$D(L)=H_p^1(\rr^{n-1})$. Assume for a moment that $\tilde{G}\in
W_p^{2-1/p}(\rr^{n-1})$. Then it follows from semigroup theory
and \eqref{trans24} that the solution of \eqref{trans23}
satisfies the estimates
    $$|\tilde{\psi}|_{H_p^3(\dot{\rr}^n)}\le C|\tilde{G}|_{W_p^{2-1/p}(\rr^{n-1})}$$
as well as
    $$|\tilde{\psi}|_{H_p^2(\dot{\rr}^n)}\le C|\tilde{G}|_{W_p^{1-1/p}(\rr^{n-1})},$$
where the constant $C>0$ does not depend on $\lambda\ge 1$. This
can be seen as in the proof of Lemma \ref{translem1}. Applying
the real interpolation method yields
    $$|\tilde{\psi}|_{W_p^{2+s}(\dot{\rr}^n)}\le C|\tilde{G}|_{W_p^{1+s-1/p}(\rr^{n-1})},$$
for some $s\in(0,1]$ and if $\tilde{G}\in
W_p^{1+s-1/p}(\rr^{n-1})$. We have thus shown that the
transmission problem \eqref{trans22} has a unique solution
$\psi\in W_p^{2+s}(\dot{\rr}^{n})$ if and only if $F\in
W_p^s(\dot{\rr}^n)$ and $G\in W_p^{1+s-1/p}(\rr^{n-1})$. By
perturbation theory, there exists $\lambda_0\ge 1$ such that
\eqref{trans21} has a unique solution $ q_k\in W_p^{2+s}(\rr^n\backslash\Gamma)$, $s\in (0,1]$, satisfying the estimate
    \begin{multline*}
    | q_k|_{W_p^{2+s}(\rr^n\backslash\Gamma)}\le C\Big(|f_k|_{W_p^s(\rr^n\backslash\Gamma)}+|(\nabla q|\nabla\varphi_k)|_{W_p^s(\rr^n\backslash\Gamma)}
    +| q\Delta\varphi_k|_{W_p^s(\rr^n\backslash\Gamma)}\\
    +|[\![ q]\!]\partial_\nu\varphi_k|_{W_p^{s+1-1/p}(\Gamma)}\Big),
    \end{multline*}
provided $\lambda\ge \lambda_0$. By the smoothness of $\varphi_k$
and by \eqref{trans16} we obtain the estimate
    \begin{align*}
    | q_k|_{W_p^{2+s}(\rr^n\backslash\Gamma)}\le C\left(|f|_{W_p^s(\Omega\backslash\Gamma)}+| q|_{W_p^{1+s}(\Omega)}\right)
    \le C\left(|f|_{W_p^s(\Omega\backslash\Gamma)}+|f|_{L_p(\Omega)}\right),
    \end{align*}
valid for all $k\in \{2,\ldots,N\}$ and $s\in(0,1]$. Since
$\{\varphi_k\}_{k=0}^N$ is a partition of unity, we obtain
    $$| q|_{W_p^{2+s}(\Omega\backslash\Gamma)}\le\sum_{k=0}^N| q_k|_{W_p^{2+s}(\Omega\backslash\Gamma)}\le C\left(|f|_{W_p^s(\Omega\backslash\Gamma)}+|f|_{L_p(\Omega)}\right),$$
showing that $ q\in W_p^{2+s}(\Omega\backslash\Gamma)$, $s\in (0,1]$.
It is easy to extend this result to the case $\lambda\in[0,\lambda_0)$. To this end, let $f\in L_p(\Omega)\cap W_p^s(\Omega\backslash\Gamma)\cap R(A_\delta)$, $s>0$, where $A_\delta:H_p^2(\Omega\backslash\Gamma)\to L_p(\Omega)$ was defined at the beginning of Section 3.
Note that $R(A_\delta)=\{f\in L_p(\Omega):(f|\mathds{1}_\rho)_2=0\}$ if $\delta=1$ and $\lambda=0$ and $R(A_\delta)=L_p(\Omega)$ if either $\delta=0$ and $\lambda\ge 0$ or $\delta=1$ and $\lambda>0$. Consider the solution $ q\in H_p^2(\Omega\backslash\Gamma)$ of \eqref{trans15}
with $\lambda\in [0,\lambda_0)$, which is well-defined thanks to
Theorem \ref{thmtrans0} and which satisfies the estimate
\eqref{trans16}. Rewriting $\eqref{trans15}_1$ as
    $$\lambda_0 q-\Delta q=f+(\lambda_0-\lambda) q,$$
we may regard the new right hand side $f+(\lambda_0-\lambda) q$
as a given function, say $\tilde{f}\in
W_p^s(\Omega\backslash\Gamma)$, $s\in(0,1]$. The above result for
\eqref{trans15} then yields the estimate
    \begin{align*}
    | q|_{W_p^{2+s}(\Omega\backslash\Gamma)}&\le C\left(|\tilde{f}|_{W_p^s(\Omega\backslash\Gamma)}+|\tilde{f}|_{L_p(\Omega)}\right)\\
    &\le C\left(|f|_{W_p^s(\Omega\backslash\Gamma)}+|f|_{L_p(\Omega)}\right),
    \end{align*}
since
    $$| q|_{W_p^s(\Omega\backslash\Gamma)}=| q|_{W_p^s(\Omega)}\le C| q|_{H_p^2(\Omega)}\le C|f|_{L_p(\Omega)},$$
by the smoothness of $ q$ and by \eqref{trans16}. If $s>1$ and $f\in L_p(\Omega)\cap W_p^s(\Omega\backslash\Gamma)$, then $ q\in H_p^3(\Omega\backslash\Gamma)$, since $f\in L_p(\Omega)\cap H_p^1(\Omega\backslash\Gamma)$. This additional regularity for $ q$ and the preceding steps allow us to conclude that $ q\in W_p^{2+s}(\Omega\backslash\Gamma)$, at least for $s\in[1,2]$. By an obvious argument it follows that $ q\in W_p^{2+s}(\Omega\backslash\Gamma)$ for each fixed $s>0$, provided $f\in L_p(\Omega)\cap W_p^{s}(\Omega\backslash\Gamma)$.
This yields the following result.
\begin{theorem}\label{thmtrans2}
Let $\Omega\subset\R^n$ be a bounded domain with boundary $\partial\Omega\in C^{2+s}$, let $1<p<\infty$, $s>0$ and $f\in L_p(\Omega)\cap
W_p^s(\Omega\backslash\Gamma)$. Then the following assertions hold.
    \begin{enumerate}
    \item If $\delta=1$ and $\lambda=0$, then there exists a unique solution $q\in W_p^{2+s}(\Omega\backslash\Gamma)\ominus\mathbb{K}\mathds{1}_\rho$ of \eqref{trans15}, provided that $(f|\mathds{1}_\rho)=0$.
    \item If either $\delta=1$ and $\lambda>0$ or $\delta=0$ and $\lambda\ge 0$, then there exists a unique solution $q\in W_p^{2+s}(\Omega\backslash\Gamma)$ of \eqref{trans15}.
    \end{enumerate}
If in addition $J=[0,a]$, $f=f(t,x)$ and $f\in H_p^1(J;L_p(\Omega)\cap W_p^{2+s}(\Omega\backslash\Gamma))$ s.t.\ $f(t,\cdot)\in R(A_\delta)$ for a.e.\ $t\in J$, then $q\in H_p^1(J;W_p^{2+s}(\Omega\backslash\Gamma)\ominus N(A_\delta))$.
\end{theorem}

%%%%%%%%%%%%%%%%%%%%%%%%%%%%%%%%%
\bibliographystyle{plain}

\end{document}